\documentclass[11pt,a4paper]{article}
\usepackage[T1]{fontenc}
\usepackage{amsfonts}
\usepackage{amsmath}
\usepackage{empheq, etoolbox}

\usepackage[dvipsnames]{xcolor}
\usepackage{tikz,pgfplots}
\pgfplotsset{compat=1.18}
\usepackage{float}
\usepackage{amssymb}
\usepackage{enumitem}
\setlist{leftmargin=5mm}
\usepackage{setspace}
\usepackage{skull}
\usepackage{nicematrix}
\usepackage[english]{babel}
\usepackage{hyperref}
\usepackage{comment}
\usepackage{amsthm}
\usepackage{changepage}   % for the adjustwidth environment
\usepackage{dsfont}
\usepackage[
backend=biber,
style=alphabetic,
sorting=nyt,
maxbibnames=99]{biblatex}
\usepackage{graphicx}
\usepackage{caption}
\usepackage[pagewise]{lineno} %for reference
\newtheorem{theorem}{Theorem}
\newtheorem{cor}[theorem]{Corollary}
\newtheorem{prop}{Proposition}
\theoremstyle{definition}
\newtheorem{definition}{Definition}
\newtheorem{lemma}{Lemma}
\newtheorem{remark}{Remark}

\numberwithin{theorem}{section}
\numberwithin{equation}{section}
\numberwithin{remark}{section}
\numberwithin{prop}{section}
\numberwithin{definition}{section}
\numberwithin{lemma}{section}
\usepackage{lipsum}

\newcommand\phantomarrow[2]{%definition de la fonction
  \setbox0=\hbox{$\displaystyle #1\to$}%
  \hbox to \wd0{%
    $#2\mapstochar
     \cleaders\hbox{$\mkern-1mu\relbar\mkern-3mu$}\hfill
     \mkern-7mu\rightarrow$}%
  \,}
\makeatletter%pour les chiffres romains
\newcommand*{\rom}[1]{\expandafter\@slowromancap\romannumeral #1@}
\makeatother

\newcommand{\R}{\mathbb{R}}

\newcommand{\K}{\mathbb{K}}

\newcommand{\N}{\mathbb{N}}

\newcommand{\T}{\mathbb{T}}

\newcommand{\I}{\mathds{1}}
\newcommand{\Pro}{\mathbb{P}}

\newcommand{\tend}[2]{\underset{#1 \to #2}{\longrightarrow}}
\newcommand{\tendf}[2]{\underset{#1 \to #2}{\rightharpoonup}}

\DeclareMathOperator{\cE}{{\cal E}}
\DeclareMathOperator{\cM}{{\cal M}}

\newcommand{\ps}[2]{\langle #1,#2\rangle}
\newcommand{\PS}[2]{\left\langle #1,#2\right\rangle}

\DeclareMathOperator{\Div}{div}
\DeclareMathOperator{\cv}{\mathfrak{c}_v}

\usepackage{csquotes} 
\title{Two-phase averaged system justification \\for ideal gases without conductivity}
\date{\today}

\author{D. Bresch\thanks{Universit\'e Savoie Mont Blanc, CNRS, LAMA, 73376 Le Bourget du lac, France, \texttt{didier.bresch@univ-smb.fr}} ,
C. Burtea\thanks{Universit\'e Paris Diderot UFR Mathématiques Batiment Sophie Germain, Bureau 727 8 place Aur\'elie Nemours, 75013 Paris, \texttt{cburtea@math.univ-paris-diderot.fr}},
P. Gonin-{}-Joubert
\thanks{Universit\'e Claude Bernard Lyon 1, CNRS, \'Ecole Centrale de Lyon, INSA Lyon, Université Jean Monnet, ICJ UMR5208, 69622 Villeurbanne, France, \texttt{goninjoubert@math.univ-lyon1.fr}},
F. Lagouti\`ere\thanks{Universit\'e Claude Bernard Lyon 1, CNRS, \'Ecole Centrale de Lyon, INSA Lyon, Université Jean Monnet, ICJ UMR5208, 69622 Villeurbanne, France,
\texttt{lagoutiere@math.univ-lyon1.fr}}}

\begin{document}

\maketitle

\begin{abstract}
This article concerns the mathematical justification of an averaged system of partial differential equations governing the evolution of a two-phase mixture of compressible ideal fluids, with viscosity and without conductivity, 
%. We obtain results 
in space dimension $1$ with periodic boundary conditions. The derivation is done by some homogenization procedure. 
The originality and the difficulty of the paper consists in the fact that \textit{both the density and temperature are allowed to oscillate (because of the absence of heat conduction), so that the limiting model is a six-equations, two-pressures, two-temperatures model}.
%We introduce a color function allowing us to encode in which phase we are at the mesocopic scale where the two phases are separated by interfaces. 
%We manage for the first time to 
%We justify the homogenization (passage from the mesoscopic situation to the macroscopic situation) in a framework without thermal conductivity 
%which complicates the analysis by 
%characterized by the presence of oscillations in density and also in temperature. 
The key point is to show the strong convergence of the stress tensor in $L^2((0,T)\times (0, 1))$. 
%The analysis is 
%much more complicated than in the barotropic case by
The main difficulties are to obtain uniform estimates in spite of the presence of oscillating coefficients in the energy equation. It requires to look at solutions with low regularity for the density and the temperature.
\end{abstract}

%\tableofcontents

%%%%%% Section 1 
\section{Introduction}

The mathematical justification of multiphase averaged models by homogenization in compressible fluid mechanics depends on the framework in which the type of solutions is sought. In recent works, it has been shown 
%when viscosities are taken into account 
in the barotropic case %(no temperature equation) 
that an appropriate framework is the one of the so-called Hoff-Desjardins solutions (see \cite{De}, \cite{Ho2}); it allows to obtain multiphase systems close to those described by formal approaches in classical references like \cite{IsHi}, \cite{DrPa}.

The procedure we adopt in the present paper is described in the barotropic case as follows:
-- In a first step, we introduce and study a model which reflects the physics at a mesoscale. This model differs from the classical one-phase one by the presence of a color function, i.e. the characteristic function of one of the phases, which labels each fluid particle and allows to prescribe different viscosity coefficients and pressure laws in the two phases. 
%First we define an appropriate mesoscale model introducing a color function that encodes the separated phases. 
In this context we prove the existence of solutions à la Hoff, global in time. 
%is proved on a time interval depending only on the $L^\infty$ bound on the initial density and on the $H^1$ bound on the initial velocity. 
%In the one-dimensional case, 
%the existence time could be shown to be global. 
-- In a second step, we study the behavior of solutions constructed from highly oscillating initial data. We show that the limit is described by an averaged system which features a kinetic equation for the family of Young measures associated with the sequence of densities and temperatures. %This equation should be compared in the incompressible Euler setting with the kinetic equation obtained by Lions, Perthame and Tadmor \cite{Lions_et_al1994} for scalar conservation laws.
%Then assuming highly oscillating initial color function with an initial density profile related to this color function and $H^1$ initial velocity, it is now possible to develop an homogenization procedure to get an averaged system including a family of Young measures satisfying a kinetic equation. Such a procedure strongly uses the renormalized procedure developped in \cite{DPLi} 
%\textcolor{green}{PGJ : en fait pas vraiment, en dimension 1 on a div u = grad u donc peut être qu'invoquer DiPerna Lions ici est trop général}.
-- In the last step, it remains to prove that if the family of Young measures is initially a linear combination of two Dirac masses (related to the two densities and temperatures of the components) then this persists in time. Using such a property, it is then possible to get a system governing the two-phase compressible mixture, system similar to the Baer-Nunziato system: see \cite{BaNu} for the presentation of this mixture theory. The interested reader is referred to \cite{BrBuLa1} in the continuous setting and \cite{BrBuLa2} in the semi-discrete case, for the barotropic case. We also refer to the recent contributions \cite{Hillairet_et_al_2022a,hillairet2023analysis} for barotropic compressible fluids in the presence of surface tension between the phases.

An important remark is that if the viscosity coefficients tend to zero in the obtained two-phase compressible system, we formally get the usual two-phase averaged system with the pressure equilibrium constraint. 
Note that if we are only concerned with 
the macroscopic model without specifying the family of Young measures, it suffices to consider global weak solutions \`a la Leray (see the book \cite{NoSt} for an introduction to this type of solutions for compressible barotropic fluids): see for instance  \cite{Hi}, \cite{PlotnikovSokolowski2012}, \cite{AmZl}. 
%In this case it does not seem possible to obtain Baer-Nunziato type averaged models.
However in the latter case we do not have enough information for the solutions in order to justify Baer-Nunziato type averaged models.

Concerning the compressible {\em heat-conducting} Navier-Stokes equations, it has been discussed for ideal gases in the recent paper \cite{Hi2} (see also \cite{AmZl96GlobalProperties} and \cite{AmZl97} for well-posedness studies of macroscopic averaged systems). Note that in dimension 1, with initial data close to an equilibrium, heat conduction and viscosity provide, in small time, $H^1$ bound in space on the temperature and the velocity field and $L^\infty$ bound occurs on the density if initially. 
It is therefore possible to follow the same procedure explained previously in the barotropic case (but for small time). The present paper concerns the compressible Navier-Stokes equations with an energy equation but {\em without heat conduction} (a physical justification of the absence of thermal conduction can be found in \cite{Kapet2001}, in which it is explained that in some situations the conductivity coefficient is much smaller than the viscosity one even if this last one is small). Assuming no thermal conduction implies that similar behaviors on the density and on the temperature may occur, namely oscillations and concentrations. Here we propose to investigate this problem, in one dimension in space, following the procedure explained here above. 
%\textcolor{green}{.}\textcolor{lightgray}{: First write %the appropriate mesoscale model using a color function %where existence of solutions à la Hoff is proved on a %global time interval.
%This part is based on a recent work by J. Li \cite{Li} by %writing an equation on the stress tensor taking care of %oscillating coefficients.
% Then assuming a highly oscillating color function and %bounded initial density and temperature related to such %color function and $H^1$ initial velocity, we show that %it is possible to develop an homogenization procedure to %get an averaged system including a family of Young %measures satisfying a kinetic equation. The last
%step is then to prove that if the family of Young %measures is initially a linear combination of two dirac %masses then it persists in time. Using such property, it %is then possible to get a two-phase compressible averaged %system with temperature.} 

Performing a formal viscosity limit in the proposed two-phase compressible averaged system with temperature, it
is interesting to note that we get a two-phase averaged system with algebraic constraint on the entropies, see \cite{La}, \cite{Deslag} and \cite{AllaireClercKokh2002} (in these last three references the mixture models are intended to treat mixed {\em cells} at the numerical level only, whereas the components are assumed to be pure at the continuous level, justifying the term "diffuse interface"). Note that the relaxation limit we perform formally assumes regular solutions, which is not the case in these references. It would be really interesting to study the viscous shock limit: For references about this topic, see for example \cite{Bercoq2002}, \cite{Chacoq2003}, \cite{AbgrallSaurel2003}
and references therein. 

The paper will be organized as follows. First, in Section \ref{section2}, we present the modeling of the mesoscopic two-phase unmixed equations and the formal derivation of the associated two-phase averaged system. We also formally discuss the weak limit relaxation (as the viscosity coefficient tends to $0$) of the two-phase averaged system to an original algebraic equilibrium law concerning the associated pressures and an evolution equation for the specific entropies. Then, in Sections \ref{formal_main} and \ref{glob_sol}, we present the global existence and uniqueness results for the compressible mesoscopic unmixed system assuming the initial intermediate regularity $\rho_0,\theta_0\in L^\infty$, $\rho_0$ far from the vacuum, and $u_0 \in H^1$. To get such an existence and uniqueness result, we develop, in Section \ref{NRJ}, the necessary energy estimates 
%including Proposition \ref{dxu} 
and uniform bounds starting from the local existence of strong solutions in \cite{Li} by writing an equation on the stress tensor taking care of oscillating coefficients. A stability result will enable us, smoothing the initial data, to obtain an existence result for Hoff solutions from those of Li. In a fourth part, Section \ref{averaging}, we use the uniform bounds obtained previously to present some convergence properties when we consider highly oscillating data related to a color function $c_0^\varepsilon$ and therefore $\rho_0^\varepsilon$ and $\theta_0^\varepsilon$ on the initial density and temperature profile.  The original point is to prove a strong convergence in $L^2((0,T)\times\Omega)$ of the stress tensor $\sigma$ with the lower and upper bounds on the density and temperature: see Proposition \ref{sigma} that is important at this stage. This allows to perform a homogenization procedure: constructing a solution of the two-phase averaged system we have in mind, formally, we can introduce Young measures and justify the formal derivation on the basis of a uniqueness result related to the associate kinetic equation. In the last section, Section \ref{num}, 
we provide numerical illustrations for the reader's convenience, comparing the solution of the mesoscopic unmixed phases 
equations with highly oscillating data to the solution of the macroscopic averaged two-phase system.

%%%%%%% Section 2 
\section{Modeling and Formal Derivation}\label{section2}

In this section we start by presenting the Navier-Stokes equations with temperature for a fluid. Then we show how to write a system with two phases separated by interfaces in a mesoscopic model valid on the whole domain, where the presence of interfaces is taken into account through a color function. We will consider in all the paper no conduction in the temperature equation.  Then introducing 
a sequence of highly oscillating unknowns depending on a parameter $\varepsilon$ and assuming some convergences
we present how to derive formally a macroscopic averaged system.

\subsection{Compressible Navier--Stokes equations with temperature}

To begin with, let us recall the Navier--Stokes equations for a Newtonian compressible non barotropic fluid, in one dimension, by adopting a Eulerian point of view (see for example \cite{NoSt} for more details). The quantities involved in describing such a fluid are the density $\rho$, the velocity $u$, the pressure $p$, the internal energy $e$ and the temperature $\theta$. We consider viscous fluids, and we denote $\mu>0$ the constant viscosity. The Cauchy stress tensor $\sigma$ takes into account the forces of pressure and friction due to viscosity, via the relationship
\begin{equation*}
    \sigma = \mu\partial_x u - p.
\end{equation*}
The total energy $E$ is the sum of kinetic $|u|^2/2$ and internal energy $e$ namely
\begin{equation*}
    E = \frac{u^2}{2} + e.
\end{equation*}
Navier-Stokes equations can then be written in a conservative form as follows
\begin{subequations}\label{NS0}
\begin{empheq}[left=\empheqlbrace]{alignat=1}
\partial_t\rho + \partial_x(\rho u)&=0
\label{cont0}\\
\partial_t(\rho u) + \partial_x(\rho u^2)&=\partial_x \sigma
\label{momts0}\\
\partial_t(\rho E)+\partial_x(\rho E u)&= \partial_x(\sigma u)
\label{energ0}.
\end{empheq}
\end{subequations}
The continuity equation \eqref{cont0} expresses the conservation of mass. The momentum equation \eqref{momts0} is simply the fundamental principle of dynamics. Finally, \eqref{energ0} expresses the total energy equation in a conservative form  which implies the conservation of total energy. Note that we do not take into account heat conductivity. To close these equations, we need some relationships between $e$, $\theta$ and $p$, $\rho$, $\theta$. We suppose that $e$ and $\theta$ are linearly dependent:
\begin{equation*}
    e = \cv \theta
\end{equation*}
where $\cv>0$ is a constant (the specific heat).
Finally, we suppose that $p$ is a function of $\rho$ and $\theta$. We will consider perfect gases, so we have 
\begin{equation*}
    p = R\rho\theta
\text{ where } 
R = (\gamma-1)\cv
\end{equation*}
with $\gamma>1$ is a constant (the heat capacity ratio).
Multiplying \eqref{momts0} by $u/2$ then substracting the result to \eqref{energ0}, we obtain an equation for the internal energy
\begin{equation*}
    \partial_t(\rho e) + \partial_x (\rho e u) = \sigma\partial_x u.
\end{equation*}

\subsection{Mesoscopic unmixed phases equations}
Let us consider a mixing of two compressible fluids, called $+$ and $-$ in a domain $\Omega$. At a given time $t>0$, fluid $+$ is present in the domain $\Omega_+(t)$, and fluid $-$ in the domain $\Omega_-(t)$, with
\begin{equation}\label{deuxdomaines}
    \overline{\Omega_+(t)}\cup \overline{\Omega_-(t)} = \Omega,\quad \Omega_+(t)\cap \Omega_-(t) = \varnothing.
\end{equation}
They satisfy the Navier-Stokes equations
\begin{subequations}\label{NS+}
\begin{empheq}[left=\empheqlbrace]{alignat=1}
\partial_t\rho_+ + \partial_x(\rho_+ u_+)&=0
\label{cont+}\\
\partial_t(\rho_+ u_+) + \partial_x(\rho_+ u_+^2)&=\partial_x \sigma_+
\label{momts+}\\
\partial_t (\rho_+E_+)+\partial_x(\rho_+E_+u_+)&=\partial_x(\sigma_+ u_+)
\label{energ+}
\end{empheq}
\end{subequations}
with
\begin{equation*}
    E_+ = u_+^2/2 + e_+,\quad \sigma_+ = \mu_+ \partial_x u_+ - p_+,\quad p_+ = R_+\rho_+\theta_+, \quad e_+ = \cv_+\theta_+
\end{equation*}
on the domain $\Omega_+$, and
\begin{subequations}\label{NS-}
\begin{empheq}[left=\empheqlbrace]{alignat=1}
\partial_t\rho_- + \partial_x(\rho_- u_-)&=0
\label{cont-}\\
\partial_t(\rho_- u_-) + \partial_x(\rho_- u_-^2)&=\partial_x \sigma_-
\label{momts-}\\
\partial_t (\rho_-E_-)+\partial_x(\rho_-E_-u_-)&=\partial_x(\sigma_- u_-)
\end{empheq}
\end{subequations}
with
\begin{equation*}
    E_- = u_-^2/2 + e_-,\quad \sigma_- = \mu_- \partial_x u_- - p_-,\quad p_- = R_-\rho_-\theta_-, \quad e_- = \cv_-\theta_-
\end{equation*}
on the domain $\Omega_-$, where $\mu_+, \mu_-, \gamma_+,\gamma_-,\cv_+,\cv_-$ are positive constants, $\gamma_+>1$  and $\gamma_->1$.

We assume now that both fluids move at the same speed $u=u_+=u_-$ (a natural assumption in order to satisfy \eqref{deuxdomaines}), and that the Cauchy stress is continuous at the interfaces, a physically reasonable assumption given the principle of reciprocal actions.
For the moment, the model is assumed to be \textit{mesoscopic}, in the sense that the torus is broken down into a succession of slices of the two different components.

\begin{remark} Later on, we will consider highly oscillating initial data related to a color function in a periodic domain $\Omega= \T$. More precisely, for a partition $\overline{\Omega^\varepsilon_{+,0}}\cup \overline{\Omega^\varepsilon_{-,0}}=\Omega$ such that the connecting components of $\Omega^\varepsilon_{\pm,0}$ have a size of the order of $\varepsilon$, we will define $c_0^\varepsilon=\I_{\Omega^\varepsilon_{+,0}}$. The approach is to make $\varepsilon$ tend towards $0$ in order to obtain a \textit{macroscopic} mixture, where we no longer consider two unmixed fluids separated by interfaces at small scale but just one as an
averaged system at large scale. If we assume $\rho_{+,0}$ and $\rho_{-,0}$ constant, then denoting 
$\rho_0^\varepsilon = \rho_{+,0} c_0^\varepsilon + \rho_{-,0} (1-c_0^\varepsilon)$, we will assume that
$\rho^\varepsilon$ weak star converges in $L^\infty(0,1)$
to $\alpha_0\rho_{+,0}  + (1-\alpha_0) \rho_{-,0} = \ps{\nu_{t,x}}{\rm Id}$  where $\nu_{t,x} = (1-\alpha_0) \delta_{\rho_{-,0}} + \alpha_0 \delta_{\rho_{+,0}}.$ 
This will be typically the initial data that we will consider with $\rho_{-,0}$ and $\rho_{+,0}$ depending on the large scale $(t,x)$.
$\square$
\end{remark}

\medskip

To begin with, we would like to write \eqref{NS+} and \eqref{NS-} as a single Navier-Stokes equation before performing an averaging procedure. The idea is a natural one: proceeding in this way, the fact that the system models a mixture will only be seen through its initial conditions. The hope is that this will lead to the study of Navier-Stokes with spatially variable viscosity and pressure law, for specific initial conditions. As in \cite{BrBuLa1}, we define $c$ as
\begin{equation*}
    \forall (t,x)\in [0,T]\times \Omega,\quad c(t,x) = \I_{\Omega_+(t)}(x).
\end{equation*}
Thanks to this function, which we will call the color function, we can easily define the density $\rho$ and the temperature $\theta$ of the mesoscopic fluid by
\begin{equation*}
    \rho = c\rho_+ + (1-c)\rho_-,\quad \theta 
    = c\theta_+ + (1-c)\theta_-.
\end{equation*}
Moreover, we define $\mu$, $\gamma$, $R$ and $\cv$ by
\begin{equation*}
    \mu(c) = c\mu_+ + (1-c)\mu_-,\quad \gamma(c)= c\gamma_+ + (1-c)\gamma_-,\quad R(c)= c R_+ + (1-c)R_-,\quad \cv(c) = c \cv_{+} + (1-c) \cv_{-}.
\end{equation*}
The equations \eqref{NS+} and \eqref{NS-} can be rewritten, thanks to the assumptions of continuity at the interfaces on $u$ and $\sigma$:
\begin{subequations}\label{NS}
\begin{empheq}[left=\empheqlbrace]{alignat=1}
\partial_t\rho + \partial_x(\rho u)&=0
\label{cont}\\
\partial_t(\rho u) + \partial_x(\rho u^2)&= \partial_x\sigma
\label{momts}\\
\partial_t (\rho E) +\partial_x (\rho E u) &=\partial_x(\sigma u) 
\label{energ}
\end{empheq}
\end{subequations}
with
\begin{equation*}
    E=u^2/2 + e,\quad \sigma = \mu(c)\partial_x u - p(c,\rho,\theta),\quad p(c,\rho,\theta) = R(c)\rho\theta,\quad e= \cv(c)\theta.
\end{equation*}
In order to close the system, we need to follow the interfaces between phases $+$ and $-$. This requires considering an equation for $c$. In our study we consider $c$ to be transported by the velocity $u$ namely
\begin{equation}\label{c}
    \partial_t c + u\partial_x c = 0
\end{equation}
which simply expresses the fact that a particle does not change its nature over time. Using the mass equation and that we will look for non vanishing density $\rho$, it may be written in a conservative form 
\begin{equation}\label{crho}
    \partial_t (\rho c)  + \partial_x(\rho u c) = 0
\end{equation}
Let us observe that from \eqref{c}, we can write
$$\partial_t (c(1-c)) + u\partial_x (c(1-c))=0$$
and therefore
\begin{equation*}
    c \in \{0,1\}
\end{equation*}
if it holds initially.
From \eqref{energ} and \eqref{momts} we deduce 
\begin{equation*}
    \partial_t(\rho e) + \partial_x(\rho e u) = \sigma\partial_x u.
\end{equation*}
Note that equations \eqref{cont}--\eqref{energ} and \eqref{crho} are considered with the initial conditions
\begin{equation}\label{inii}
\rho\vert_{t=0} = \rho_0, \qquad 
 (\rho c)\vert_{t=t_0} = \rho_0 c_0, \qquad 
 (\rho u)\vert_{t=t_0} = \rho_0 u_0, \qquad
 (\rho E)\vert_{t=t_0} = \rho_0 E_0,
\end{equation}
with $E= u^2/2+ c_v(c)\theta$.
Note that the initial conditions will be chosen later on compatible
to the mesoscopic unmixed phases system.

\subsection{Macroscopic averaged equations: Formal derivation}
\label{formderiv}

\begin{prop}\label{CrucialProp} Let us assume that $\sigma^\varepsilon$ and $u^\varepsilon$ strongly converges to $\sigma$ and $u$ in $L^2(0,T;L^2(\T))$. Moreover let us consider that there exists $\alpha_\pm \in [0,1]$, $\rho_\pm\in [\underline\rho, \overline\rho]$, $\theta_\pm\in  [\underline\theta, \overline\theta]$ some measurable functions in time and space (where $0<\underline{\rho}<\overline{\rho}$, $0<\underline{\theta}<\overline{\theta}$ are some constants) such that, for all continuous $\beta : [0,1]\times [\underline{\rho},\overline{\rho}]\times [\underline{\theta},\overline{\theta}]\rightarrow\R$, defining a set of probability measures $\nu_{t,x}^\varepsilon$ (see Theorem \ref{Theorem_main2}) such
that
$$ \langle \nu_{t,x}^\varepsilon, \beta\rangle =   \beta(c^\varepsilon,\rho^\varepsilon,\theta^\varepsilon).
$$
Assuming 
    \begin{equation*}
     \langle \nu_{t,x}^\varepsilon, \beta\rangle \tendf{\varepsilon}{0}
     \langle \nu_{t,x}, \beta\rangle
             \hbox{ in } L^2([0,T]\times \T)
      \end{equation*}
      where 
      \begin{equation*}
    \nu_{t,x} =  \alpha_+ \delta_{(1,\rho_+,\theta_+)}
    + \alpha_- \delta_{(0,\rho_-,\theta_-)},
    \end{equation*}
    namely
    \begin{equation}\label{ansatz}
        \beta(c^\varepsilon,\rho^\varepsilon,\theta^\varepsilon) \tendf{\varepsilon}{0}  \alpha_+ \beta(1,\rho_+,\theta_+) + \alpha_- \beta(0,\rho_-,\theta_-)
        \hbox{ in } L^2([0,T]\times \T),
    \end{equation}
    then
     \begin{subequations}\label{systmacro2}
    \begin{empheq}[left=\empheqlbrace]{alignat=1}
        \partial_t \alpha_\pm + u\partial_x \alpha_\pm &= \frac{\alpha_+\alpha_-}{\alpha_-\mu_+ + \alpha_+\mu_-}(\sigma_\mp-\sigma_\pm)\\
        \partial_t(\alpha_\pm \rho_\pm) + \partial_x(\alpha_\pm \rho_\pm u) &= 0\label{systmacro_m2}\\
        \partial_t(\rho u) + \partial_x(\rho u^2)&= \partial_x\sigma
        \hbox{ with } 
        \rho = \alpha_+\rho_++ \alpha_-\rho_-\\
        \label{systmacro-energies}
        \partial_t(\alpha_\pm \rho_\pm e_\pm) + \partial_x(\alpha_\pm \rho_\pm e_\pm u) &= \frac{\alpha_+\alpha_-}{\alpha_-\mu_+ + \alpha_+\mu_-}\sigma(\sigma_\mp-\sigma_\pm) + \alpha_\pm \sigma \partial_x u
    \end{empheq}
    \end{subequations}
    where $e_\pm=\cv_\pm \theta_\pm$ and 
    \begin{equation*}
        \sigma_\pm = \mu_\pm \partial_x u - p_\pm
    \end{equation*}
    with $p_\pm = R_\pm\rho_\pm \theta_\pm$, and with the homogenized stress tensor $\sigma$ given by
$$
    \sigma = \frac{\alpha_+/\mu_+}{\alpha_+/\mu_+ + \alpha_-/\mu_-}\sigma_+ + \frac{\alpha_-/\mu_-}{\alpha_+/\mu_+ + \alpha_-/\mu_-} \sigma_-.
$$

\end{prop}

\begin{remark}
Note that $\sigma$ may be written as $\sigma = \mu_{eff}\partial_x u - p_{eff}$, where
\begin{equation*}
    \mu_{eff} = \frac{1}{\alpha_+/\mu_++\alpha_-/\mu_-},\quad p_{eff} = \frac{\alpha_+p_+/\mu_++\alpha_-p_-/\mu_-}{\alpha_+/\mu_++\alpha_-/\mu_-}.
\end{equation*}
\end{remark}

\begin{remark}
For obtaining the set of equations at the macroscopic scale,
it is related to averaging procedure. 
To do so, we need to consider an averaging operator $E[\cdot]$ (see for instance \cite{Perrier2021} using expectation on a simple Stochastic model).  In our case, 
$E(b(c, \rho,\theta)) = \alpha_+ b(1,\rho_+,\theta_+) + \alpha_- b(0,\rho_-,\theta_-).$ 
\end{remark}

\begin{proof}
We will only give the proofs of the system "$+$". For "$-$" we just have to replace $c$ by $1-c$, remarking that, taking $\beta(c,\rho,\theta) = 1$ in \eqref{ansatz}, we have
\begin{equation*}
    \alpha_+ + \alpha_- = 1.
\end{equation*}
We will denote $\beta(c^\varepsilon,\rho^\varepsilon,\theta^\varepsilon)=:\beta^\varepsilon$ in the following.
The main idea is to choose appropriate nonlinear functions
localized around the $+$ phase playing with property that
the triple $(1,\rho_+,\theta_+)$ never cross the triple $(0,\rho_-,\theta_-)$. Of course calculations are formal in this subsection but important for readers to understand where the macroscopic averaged equations may come from.
\begin{itemize}
    \item To begin with, recall $\sigma^\varepsilon=\mu^\varepsilon\partial_x u^\varepsilon - p^\varepsilon$, we write
\begin{equation*}
    0 = \partial_t c^\varepsilon + u^\varepsilon \partial_x c^\varepsilon = \partial_t c^\varepsilon + \partial_x(c^\varepsilon u^\varepsilon) - c^\varepsilon\partial_x u^\varepsilon  = \partial_t c^\varepsilon + \partial_x(c^\varepsilon u^\varepsilon) - c^\varepsilon\frac{\sigma^\varepsilon + p^\varepsilon}{\mu^\varepsilon}
\end{equation*}
and passing to the limit, using \eqref{ansatz} with $\beta(c,\rho,\theta)=c$, $\beta(c,\rho,\theta)= c/\mu(c)$ then $\beta(c,\rho,\theta)= c\,  p(c,\rho,\theta)/\mu(c)$ and the strong convergence on $\sigma^\varepsilon$ and $u^\varepsilon$,  we get
\begin{equation*}
    \partial_t \alpha_+ + \partial_x(\alpha_+ u) - \frac{\alpha_+}{\mu_+}( \sigma + p_+) =0
\end{equation*}
which we can rewrite as 
\begin{equation*}
    \partial_t \alpha_+ + u\partial_x \alpha_+ = \frac{\alpha_+}{\mu_+}(\sigma-\sigma_+).
\end{equation*}
\item To prove \eqref{systmacro_m2} we just pass to the limit the equality
\begin{equation*}
    \partial_t (c^\varepsilon\rho^\varepsilon) + \partial_x (c^\varepsilon\rho^\varepsilon u^\varepsilon) = 0
\end{equation*}
using \eqref{ansatz} with $\beta(c,\rho,\theta)= c\rho$
and the strong convergence on $u^\varepsilon$.
\item Starting from the equality
\begin{equation*}
    \partial_t (\cv(c^\varepsilon)\rho^\varepsilon\theta^\varepsilon) + \partial_x (\cv(c^\varepsilon) \rho^\varepsilon \theta^\varepsilon u^\varepsilon) =  \sigma^\varepsilon\partial_x u^\varepsilon
\end{equation*}
we rewrite the right--hand side as
\begin{equation*}
     \sigma^\varepsilon \partial_x u^\varepsilon = \sigma^\varepsilon \frac{\sigma^\varepsilon + p^\varepsilon}{\mu^\varepsilon}
\end{equation*}
and passing to the limit using \eqref{ansatz} with $\beta(c,\rho,\theta)= \cv(c)\rho\theta$, $\beta(c,\rho,\theta)=1/\mu(c)$ then $\beta(c,\rho,\theta)= p(c,\rho,\theta)/\mu(c)$, and the strong of $\sigma^\varepsilon$ and $u^\varepsilon$, we get
\begin{equation*}
    \cv_+\partial_t(\alpha_+\rho_+\theta_+) + \cv_+ \partial_x(\alpha_+\rho_+\theta_+ u) = \frac{\alpha_+}{\mu_+} \sigma^2 + \alpha_+\frac{p_+}{\mu_+}\sigma
\end{equation*}
which we can rewrite
\begin{equation*}
    \partial_t (\alpha_+ \rho_+ e_+) 
    + \partial_x (\alpha_+\rho_+e_+ u ) = \frac{\alpha_+}{\mu_+}\sigma(\sigma-\sigma_+) + \alpha_+\sigma\partial_x u.
\end{equation*}
\end{itemize}
We obtain the following averaged system    
    \begin{subequations}\label{systmacro}
    \begin{empheq}[left=\empheqlbrace]{alignat=1}
        \partial_t \alpha_\pm + u\partial_x \alpha_\pm &= \frac{\alpha_\pm}{\mu_\pm}(\sigma-\sigma_\pm)\\
        \partial_t(\alpha_\pm \rho_\pm) + \partial_x(\alpha_\pm \rho_\pm u) &= 0\label{systmacro_m}\\
        \partial_t(\alpha_\pm \rho_\pm e_\pm) + \partial_x(\alpha_\pm \rho_\pm e_\pm u) &= \frac{\alpha_\pm}{\mu_\pm}\sigma(\sigma-\sigma_\pm) + \alpha_\pm \sigma \partial_x u 
    \end{empheq}
    \end{subequations}
Note that we can pass to the limit in the momentum equation and find directly

\begin{equation}\label{momtsformal}
\partial_t (\rho u) + \partial_x(\rho u^2)
= \partial_x \sigma    
\end{equation}

where 
$\rho = \alpha_+ \rho_+ +\alpha_- \rho_-.$
We need now to find the expression of $\sigma$.
Even if $\sigma^\varepsilon$ converges strongly to $\sigma$ in $L^2([0,T]\times\Omega)$, the expression of the limit Cauchy stress is not clear. Indeed, passing to the limit in
\begin{equation}
    \sigma^\varepsilon = \mu(c^\varepsilon)\partial_x u^\varepsilon - p(c^\varepsilon,\rho^\varepsilon,\theta^\varepsilon)
    \label{deffluxeff}
\end{equation}
is not immediate because of the nonlinearity of $p$ and the fact that $\partial_x u^\varepsilon$ converges only weakly. However, following the classic idea of homogenization observed by S. Spagnolo and F. Murat, L. Tartar (see for example \cite{Sp}, \cite{CiDo} Theorem 5.5 and \cite{MuTa}), we rewrite \eqref{deffluxeff}
\begin{equation*}
    \partial_x u^\varepsilon = \frac{1}{\mu(c^\varepsilon)}\sigma^\varepsilon + \frac{p(c^\varepsilon,\rho^\varepsilon,\theta^\varepsilon)}{\mu(c^\varepsilon)}.
    \label{defffluxeff2}
\end{equation*}
Using now the strong convergence of $\sigma^\varepsilon$, we get
\begin{equation*}
    \partial_x u = 
    \Bigl(\frac{\alpha_+}{\mu_+}+ \frac{\alpha_+}{\mu_+}  \Bigr)\sigma 
    + \Bigl(\frac{\alpha_+ p_+}{\mu_+} +   
        \frac{\alpha_- p_-}{\mu_-}\Bigr)
\end{equation*}
and therefore, with the notations of the proposition,
\begin{equation*}
    \sigma = \mu_{eff}\partial_x u - p_{eff}.
\end{equation*}

In particular, $\sigma$ can be seen as a convex combination of $\sigma_+$ and $\sigma_-$ observing that:
\begin{align*}
    \sigma &= \frac{1}{\alpha_+/\mu_+ + \alpha_-/\mu_-}(\partial_x u - \alpha_+ p_+/\mu_+ - \alpha_- p_-/\mu_-)
    \\&= \frac{1}{\alpha_+/\mu_+ + \alpha_-/\mu_-}(\alpha_+(\partial_x u - p_+/\mu_+) + \alpha_-(\partial_x u - p_-/\mu_-))
    \\&= \frac{\alpha_+/\mu_+}{\alpha_+/\mu_+ + \alpha_-/\mu_-}\sigma_+ + \frac{\alpha_-/\mu_-}{\alpha_+/\mu_+ + \alpha_-/\mu_-} \sigma_-.
\end{align*}
Thus we can rewrite \eqref{systmacro},\eqref{momtsformal} as \eqref{systmacro2}.
\end{proof}

\begin{prop}\label{Relaxationprop} Let the assumptions of
Proposition \ref{CrucialProp} be satisfied. Then the entropies $s_\pm$ given by 
\begin{equation}\label{entropies}
    s_\pm = \cv_\pm\ln(\theta_\pm) - R_\pm\ln(\rho_\pm)
\end{equation}
satisfy
\begin{align*}
    \alpha_\pm \rho_\pm \theta_\pm D_t s_\pm &= \frac{\alpha_\mp \mu_\pm\alpha_+\alpha_-}{(\alpha_-\mu_+ + \alpha_+\mu_-)^2}(\sigma_\mp - \sigma_\pm)^2 + \frac{2\alpha_+\alpha_-}{\alpha_-\mu_+ + \alpha_+\mu_-}(\sigma_\mp - \sigma_\pm)\mu_\pm\partial_x u + \alpha_\pm \mu_\pm (\partial_x u)^2
\end{align*}
\end{prop} 

\begin{proof}
Starting with the formula \eqref{entropies}, we get
\begin{equation}\label{relax1}
    \alpha_\pm \rho_\pm \theta_\pm D_t s_\pm = \alpha_\pm \rho_\pm \cv_\pm D_t \theta_\pm - R_\pm \theta_\pm \alpha_\pm D_t \rho_\pm.
\end{equation}
We now recall that
\begin{equation*}
    0 = \partial_t (\alpha_\pm \rho_\pm) + \partial_x(\alpha_\pm \rho_\pm u) = \alpha_\pm D_t \rho_\pm + \rho_\pm (D_t\alpha_\pm + \alpha_\pm \partial_x u)
\end{equation*}
thus
\begin{equation*}
    \alpha_\pm D_t\rho_\pm = -\frac{\alpha_+\alpha_-}{\alpha_-\mu_+ + \alpha_+\mu_-}\rho_\pm(\sigma_\mp-\sigma_\pm) - \alpha_\pm \rho_\pm \partial_x u
\end{equation*}
and therefore, using also \eqref{systmacro-energies},  \eqref{relax1} reads
\begin{align*}
    \alpha_\pm \rho_\pm \theta_\pm D_t s_\pm &= \frac{\alpha_+\alpha_-}{\alpha_-\mu_+ + \alpha_+\mu_-}\sigma(\sigma_\mp-\sigma_\pm) + \alpha_\pm \sigma \partial_x u + R_\pm\theta_\pm \left(\frac{\alpha_+\alpha_-}{\alpha_-\mu_+ + \alpha_+\mu_-}\rho_\pm(\sigma_\mp-\sigma_\pm) + \alpha_\pm \rho_\pm \partial_x u\right)
    \\&= (\sigma + p_\pm)\left(\frac{\alpha_+\alpha_-}{\alpha_-\mu_+ + \alpha_+\mu_-}(\sigma_\mp-\sigma_\pm)+\alpha_\pm \partial_x u\right).
\end{align*}
Now we remind with the definition of $\sigma_\pm$ and $\sigma$ that
\begin{equation*}
    \sigma + p_\pm = \sigma - \sigma_\pm + \mu_\pm \partial_x u = \frac{\alpha_\mp \mu_\pm}{\alpha_+\mu_- + \alpha_-\mu_+}(\sigma_\mp - \sigma_\pm) + \mu_\pm \partial_x u
\end{equation*}
thus we conclude the proof.
\end{proof}

\bigskip

\begin{cor}
Let us now assume that $\mu_\pm$ is of order $\eta$ with $\eta>0$ a small parameter. Assume that all the unknowns may be
developed in Taylor series with respect to $\eta$. 
Then at the main order
\begin{equation*}
p_+^0 = p_-^0 \text{ and }
\alpha_+^0 \rho_+^0 \theta_+^0 D_t s_+^0 
    = \alpha_-^0 \rho_-^0 \theta_-^0 D_t s_-^0 = 0.
\end{equation*} 
\end{cor} 

\begin{proof} 

Let us start with the entropies equations
\eqref{entropies}. Looking at order $-1$ and order $0$ with respect to $\eta$, we can prove the corollary. 
First observe that the more singular quantity with respect to the viscosity coefficients is the first term in the right-hand side
of \eqref{entropies}. Thus we get
$$\sigma_-^0 = \sigma_+^0$$
which reads $p_-^0=p_+^0.$ 
Then looking at order $0$, we conclude since 
$$(\sigma_--\sigma_+)^2 = O(\eta^2).$$
\end{proof}

\begin{remark} We can write equations $\alpha^0_\pm\rho^0_\pm\theta^0_\pm D_t s^0_\pm=0$ as $D_t s^0_\pm=0$ assuming $\alpha^0_\pm \rho^0_\pm \theta^0_\pm$ does not vanish, see \cite{La}.
\end{remark} 
\begin{remark}
In the barotropic case this problem has been addressed in the recent paper \cite{BurteaCrin-BaratTan2023} for the case of regular solutions close to equilibrium. 
\end{remark}

%%%%%% Section 3

\section{Formal statement of the main results}\label{formal_main}

%The convergence behavior of the initial data (i.e. weak convergence for the color function, density and temperature, strong convergence for velocity) is preserved over time.
We first define the functional framework in which we will consider the solutions of equations \eqref{NS}-\eqref{c} on a time interval $(0,T)$ fixed (in a periodic setting in space chosen to be $(0,1)$ later-on) which can be called the Hoff framework. To construct such solutions,
we regularize initial data (with a family of approximations of the identity indexed by the parameter $n$) and use the existence of solutions provided by \cite{Li} to have in hand a sequence of regular solutions. We can prove that such a sequence satisfies a certain number of estimates uniformly with respect to the regularization. %that we will present  
These estimates that define the class of solutions usually called \`a la Hoff. This class of solutions is somehow intermediate between the class of bounded energy weak solutions (see D. Bresch, P-E. Jabin, F. Wang \cite{BrJaWu}, and references cited therein) and regular solutions as constructed by Kazhikhov and Shelukhin \cite{KaSh}, but here in the case without heat-conductivity. In particular it allows to consider discontinuous densities but the velocity is sufficiently regular so as to trace the evolution of these discontinuities. We refer to \cite{Hoff1986,Hoff1987,Ho1,Ho2} for Hoff's earlier work on compressible fluids one--phase flows.
Of course, it suffices then to pass to the limit to obtain the desired solution (weak nonlinear stability procedure). We then prove uniqueness of the limit in the Hoff class of solutions to get a convergence on the whole sequence. It should be noted that while the results given in the following are valid in the presence of a color function, they are interesting in their own right in the study of Navier-Stokes with temperature and without thermal conductivity.

Since $\mu$ is not a constant but a discontinuous function, the simplest way to give meaning to "weak solution" of \eqref{cont}--\eqref{energ} and \eqref{crho} with initial conditions \eqref{inii} is the following:

\begin{definition}\label{Def1}
The quadruplet $(c,\rho,\theta,u)\in L^\infty(0,T,L^\infty(\T))\times L^\infty(0,T,L^1(\T)) ^2\times L^1(0,T,W^{1,1}(\T))$ is called "weak solution" of \eqref{cont}--\eqref{energ} and \eqref{crho} if \eqref{cont}--\eqref{energ} and \eqref{crho} are satisfied in a distribution sense and the initial data \eqref{inii} in a weak sense. 
%   and for all $\varphi \in C^1([0,T]\times\T)$ we have
%   \begin{align*}
%        \int_0^1 c\rho\varphi - \int_0^t\int_0^1 c\rho\partial_t \varphi - %\int_0^t\int_0^1 c\rho u \partial_x\varphi &= \int_0^1 %c_0\rho_0\varphi(0,\cdot)
%        \\
%        \int_0^1\rho\varphi - \int_0^t\int_0^1 \rho \partial_t\varphi - %\int_0^t\int_0^1 \rho u \partial_x\varphi &= \int_0^1 %\rho_0\varphi(0,\cdot)
%        \\
%        \int_0^1 \rho u \varphi - \int_0^t\int_0^1 \rho u \partial_t\varphi %- \int_0^t\int_0^1 \rho u^2 \partial_x\varphi +\int_0^t\int_0^1 %\sigma\partial_x\varphi &= \int_0^1 \rho_0 u_0\varphi(0,\cdot)
%        \\
%       \int_0^1 \rho E \varphi - \int_0^t\int_0^1 \rho E\partial_t\varphi - %\int_0^t\int_0^1 \rho E u \partial_x\varphi + \int_0^t\int_0^1 \sigma %u\partial_x\varphi &= \int_0^1 \rho_0 E_0\varphi(0,\cdot). 
%    \end{align*}
%where $\sigma = \mu(c) \partial_x u - p(c,\rho,\theta).$
\end{definition}

%\bigskip

\begin{remark}
Concerning two-phase mixtures, a relevant subclass of initial data is given by
\begin{equation*}
    \rho_0^\varepsilon = c_0^\varepsilon\rho_{+,0} + (1-c_0^\varepsilon)\rho_{-,0}
\end{equation*}
with $c_0^\varepsilon$ highly oscillating and $\rho_{\pm,0}$ non oscillating functions. Thus $\rho_0^\varepsilon$ will start to oscillate faster and faster as $\varepsilon$ tends towards $0$ in the averaging process. So we cannot ask more than a uniform bound in $L^\infty([0,T]\times \T)$ for $\rho$. That's the same for the temperature $\theta$ because we do not take into account heat conductivity. Since the velocity $u$ is assumed to be common to both fluids, it's legitimate to take $u_0\in H^1(\T)$, then we can expect to obtain greater regularity on $u$. 
\end{remark}

Let's forget the color function for a moment. Navier-Stokes solutions $(\rho,\theta,u)$ in $L^\infty\times L^\infty\times H^1$ as described above are of intermediate regularity between strong solutions (where $\rho, \theta, u$ are regular) and Leray solutions (where $\rho, \theta, u$ are not very regular). Navier-Stokes equations have been studied in this framework, in the barotropic case and in $2$ and $3$ dimensions by D. Hoff \cite{Ho1} for small data and by B. Desjardins \cite{De} in short time. D. Hoff also deals with the 1D case in \cite{Ho2}, globally in time without size restriction on the initial data. His results are extended in \cite{BrBuLa1} and \cite{BrBuLa2} taking into account the color function. In the next section, we will give some energy estimates for our problem, essentially by mixing and improving some results from \cite{Li}, \cite{BrBuLa1} and \cite{BrBuLa2}.

\ 

We introduce here the framework of solutions "à la Hoff" for Navier-Stokes with temperature and without conductivity, with transported coefficients.

\begin{definition}\label{defhoff} Let $(c_0,\rho_0,\theta_0,u_0)\in L^\infty(\T)^3\times H^1(\T)$. Then $(c,\rho,\theta,u)$ is called a solution "à la Hoff" on $[0,T]$ of \eqref{cont}--\eqref{energ} and \eqref{crho}  with initial conditions  \eqref{inii} if there exists positive constants 
\newline
$\underline{\rho},\overline{\rho},\underline{\theta},\overline{\theta},C_1,C_2,C_3,C_4,C_5,C_6$, which may depend on $T$, such that
    \begin{itemize}
        \item $(c,\rho,\theta,u)$ is a weak solution of \eqref{cont}--\eqref{energ} and \eqref{crho} with initial conditions \eqref{inii} in the sense of Definition \ref{Def1}
        \item $(c,\rho,\theta,u)$ verifies the following bounds:
        \begin{equation}\label{D_UL_rho}
            \underline{\rho}\leq \rho \leq \overline{\rho} \text{ a.e. in }[0,T]\times \T,
        \end{equation}
        \begin{equation}\label{D_UL_theta}
            \underline{\theta}\leq \theta\leq \overline{\theta} \text{ a.e. in }[0,T]\times\T,
        \end{equation}
        \begin{equation}\label{D_A1}
            \int_0^1\sigma^2 + \int_0^t\int_0^1 (\partial_x \sigma)^2\leq C_1,\quad \int_0^T\Vert\sigma\Vert_\infty^2\leq C_2,
        \end{equation}
        \begin{equation}\label{D_A1inf}
            \sup_{[0,T]}\int_0^1(\partial_x u)^2\leq C_3, \quad \sup_{[0,T]}\Vert u\Vert_\infty^2\leq C_3,\quad \int_0^T\Vert\partial_x u\Vert_\infty^2 \leq C_4,\quad \int_0^T\int_0^1(\partial_t u)^2\leq C_5,
        \end{equation}
        \begin{equation}
            \sup_{0\leq t \leq T}\min(1,t)\int_0^1 (\partial_x\sigma)^2 + \int_0^T\min(1,s) \int_0^1 \dot{\sigma}^2\leq C_6.
            \label{secondhoff}
        \end{equation}
    \end{itemize}
    If $T$ may be chosen arbitrarily large we will use the term global solution "à la Hoff". A sequence of solutions "à la Hoff" indexed by $n$ is called "uniformly bounded" with respect to $n$ if it checks all previous bounds uniformly. 
\end{definition}

Our first main result is a global existence and uniqueness result that reads

\begin{theorem}
     For any $(c_0,\rho_0,\theta_0,u_0)\in L^\infty(\T)^3\times H^1(\T)$ 
 such that there exist $\underline{\rho_0},\overline{\rho_0},\underline{\theta_0},\overline{\theta_0},C_0>0$ 
 with
\begin{equation}\label{initialcond}
    c_0\in [0,1],\quad \underline{\rho_0}\leq \rho_0 \leq \overline{\rho_0},\quad \underline{\theta_0}\leq \theta_0 \leq \overline{\theta_0},\quad \Vert u_0\Vert_{H^1}\leq C_0,
\end{equation}   
there exists a unique global solution "à la Hoff" of \eqref{cont}--\eqref{energ} and \eqref{crho}  with initial conditions  \eqref{inii}. Moreover, the bounds in \eqref{D_UL_rho}--\eqref{secondhoff} can be chosen so as to depend only on $\mu_\pm,\gamma_\pm,\cv_\pm, T$ and $\underline{\rho_0},\overline{\rho_0},\underline{\theta_0},\overline{\theta_0},C_0$.
\end{theorem}

The idea of the proof is the following: Jinkai Li proved an existence result in \cite{Li} for smooth initial data (without  color function, but it is clear that this proof can be adapted in the present case for a smooth color function). Thus, there is existence for regularized initial data, then we use the stability result Theorem \ref{stability} to construct a unique solution à la Hoff.

\ 

The next result concerns the limiting behavior of a sequence of solutions
"\`{a}~ la Hoff" emanating from a sequence of uniformly bounded initial data.
More precisely, for all $\varepsilon>0$ consider%
\begin{equation*}
(c_{0}^{\varepsilon},\rho_{0}^{\varepsilon},\theta_{0}^{\varepsilon}%
,u_{0}^{\varepsilon})\in L^{\infty}(\T)^{3}\times H^{1}(\T) \label{ini_1}%
\end{equation*}
verifying \eqref{initialcond} with some bounds uniform in $\varepsilon$, and suppose that
\begin{equation}
\left\{
\begin{array}
[c]{l}%
c_{0}^{\varepsilon},\rho_{0}^{\varepsilon},\theta_{0}^{\varepsilon
}\tendf{\varepsilon}{0}c_{0},\rho_{0},\theta_{0}\text{ weakly--}\star\text{ in }L^{\infty}(\T),\\
u_{0}^{\varepsilon}\tendf{\varepsilon}{0}u_{0}\text{ in }H^{1}(\T).
\end{array}
\right.  \label{ini_2}%
\end{equation}
In particular, there could be, and for the situation that is relevant to
mixtures of fluids there are, points $x\in\T$ where the associated sequence of
values $\left(  c_{0}^{\varepsilon}(x),\rho_{0}^{\varepsilon}(x),\theta
_{0}^{\varepsilon}(x)\right)  $ oscillates around $\left(  c_{0}(x),\rho
_{0}(x),\theta_{0}(x)\right)  $. Fix $T>0$ and for all $\varepsilon>0,$
consider $(c^{\varepsilon},\rho^{\varepsilon},\theta^{\varepsilon
},u^{\varepsilon})$ the solution "\`{a}~la Hoff" of \eqref{NS} with initial
data $(c_{0}^{\varepsilon},\rho_{0}^{\varepsilon},\theta_{0}^{\varepsilon
},u_{0}^{\varepsilon})$. The oscillations initially present in the initial
data are then propagated in time (no smoothing effect is to be expected since
the subsystem of PDEs verified by $c^{\varepsilon},\rho^{\varepsilon}$ and
$\theta^{\varepsilon}$ is hyperbolic) which creates an instability: even if
the sequence of solutions $(c^{\varepsilon},\rho^{\varepsilon},\theta
^{\varepsilon},u^{\varepsilon})$ converges weakly-$\star$ to some limiting
functions $(c,\rho,\theta,u)$ the later are not, in general, weak solutions for
\eqref{NS} with initial data $\left(  c_{0},\rho_{0},\theta_{0},u_{0}\right)
$. Of course, this is due to the presence of nonlinearities which, roughly
speaking, do not commute with weak convergence. Nevertheless, we can still
describe the limiting system by introducing the Young measures associated to
$(c^{\varepsilon},\rho^{\varepsilon},\theta^{\varepsilon})$. This mathematical
entity contains all the information regarding the oscillations of the initial
sequence. In the context of 1D compressible fluids it has been observed since the $90$s (see D. Serre \cite{Serre1991}, W. E \cite{Weinan1992}, Amosov and Zlotnik \cite{AmosovZlotnik1996b} for
barotropic flows or Amosov and Zlotnik \cite{Amosov2001}, M. Hillairet \cite{Hi2} for heat
conducting fluids) that an evolution equation can be written for this object. 

We recall that we are concerned with heat non-conducting fluids and as such
the temperature is oscillating. It is this aspect that is crucial in order to obtain a "two-temperature"-system in the limit. We formalize the previous discussion in the following:
\begin{theorem}
\label{Theorem_main1}Let $T>0$ be fixed. Consider a sequence of uniformly
bounded initial conditions $(c_{0}^{\varepsilon},\rho_{0}^{\varepsilon}%
,\theta_{0}^{\varepsilon},u_{0}^{\varepsilon})$, $(c_{0},\rho_{0},\theta
_{0},u_{0})\in L^{\infty}(\T)^{3}\times H^{1}(\T)$ verifying \eqref{initialcond} with bounds uniform in $\varepsilon$ and $\eqref{ini_2}  $. For all $\varepsilon>0,$ consider
$(c^{\varepsilon},\rho^{\varepsilon},\theta^{\varepsilon},u^{\varepsilon})$
the solution "\`{a}~ la Hoff" of of \eqref{cont}--\eqref{energ} and \eqref{crho} with initial data $(c_{0}%
^{\varepsilon},\rho_{0}^{\varepsilon},\theta_{0}^{\varepsilon},u_{0}%
^{\varepsilon})$. Then, there exists some $(c,\rho,\theta)\in L^{\infty
}([0,T]\times\T)^{3},$ $\sigma\in L^{\infty}(0,T,L^{2}(\T))$ and $u\in
H^{1}([0,T]\times\T)$ such that%
\begin{align}
&  c^{\varepsilon},\rho^{\varepsilon},\theta^{\varepsilon}%
\tendf{\varepsilon}{0}c,\rho,\theta\text{ weakly}-\star\text{ in }L^{\infty
}([0,T]\times\T),\label{faibleconv}\\
&  u^{\varepsilon}\tendf{\varepsilon}{0}u\text{ weakly in }H^{1}%
([0,T]\times\T),\label{faibleconv2}\\
&  u^{\varepsilon}\tend{\varepsilon}{0}u\text{ in }L^{2}(0,T,L^{2}%
(\T)),\label{forteconv}\\
&  \sigma^{\varepsilon}\tend{\varepsilon}{0}\sigma\text{ in }L^{2}%
(0,T,L^{2}(\T)).\label{strongconvsigmaeps}%
\end{align}
Moreover there exists constants $\underline{\rho},\overline{\rho}%
,\underline{\theta},\overline{\theta}>0$ such that\footnote{We omit the time
dependency of the constants in order to ease the reading. It is clear that a
global version of this result can be formulated albeit some technical
details.}:%
\begin{equation}
\forall(t,x)\in\lbrack0,T]\times\T,\quad0\leq c(t,x)\leq1,\quad\underline
{\rho}\leq\rho(t,x)\leq\overline{\rho},\quad\underline{\theta}\leq
\theta(t,x)\leq\overline{\theta}.\label{ineglim}%
\end{equation}
Finally, there exists\footnote{$\left(  L^{1}\left(  \left(  0,T\right)
\times\mathbb{T}\right)  ,C\left(  K\right)  \right)  ^{\prime}$ represents
the dual of $L^{1}(\left(  0,T\right)  \times\mathbb{T},C\left(  K\right)
).$} 
$\nu \in \left(  L^{1}\left(
\left(  0,T\right)  \times\mathbb{T}\right),C\left(  K\right)  \right)
^{\prime}$ such that%
\begin{equation} 
\left\{
\begin{array}
[c]{l}%
\partial_{t}\rho+\partial_{x}\left(  \rho u\right)  =0,\\
\partial_{t}c+\partial_{x}\left(  cu\right)  =\left\langle \nu,\dfrac{c}%
{\mu\left(  c\right)  }\right\rangle \sigma+\left\langle \nu,\dfrac{cp\left(
c,\rho,\theta\right)  }{\mu\left(  c\right)  }\right\rangle,\\
\partial_{t}(\rho u)+\partial_{x}(\rho u^{2})-\partial_{x}\sigma=0,\\
\partial_{t}\left(  \rho\left\langle \nu,\cv\theta\right\rangle \right)
+\partial_{x}\left(  \left\langle \nu,\rho\cv\theta\right\rangle u\right)
=\sigma\partial_{x}u,\\
\left\langle \nu,\dfrac{1}{\mu\left(  c\right)  }\right\rangle \sigma
=\partial_{x}u-\left\langle \nu,\dfrac{cp\left(  c,\rho,\theta\right)  }%
{\mu\left(  c\right)  }\right\rangle ,
\end{array}
\right.  \label{system_general}%
\end{equation}
\end{theorem}
Moreover we justify that $\nu$ satisfies a kinetic equation which 
with \eqref{system_general} form a closed system. Namely we prove the following theorem
\begin{theorem}\label{Theorem_main2}
Let $K=[0,1]\times\lbrack\underline{\rho},\overline{\rho}]\times
\lbrack\underline{\theta},\overline{\theta}]$. For $\varepsilon>0$, we define
as%
\begin{equation*}
\left\{
\begin{array}
[c]{l}%
\nu^{\varepsilon}:[0,T]\times\T\rightarrow\Pro(K),(t,x)\mapsto\nu
_{t,x}^{\varepsilon}\\
\forall(t,x)\in\lbrack0,T]\times\T,\quad\forall\beta\in C(K),\quad\left\langle
{\nu_{t,x}^{\varepsilon},}{\beta}\right\rangle =\beta(c^{\varepsilon
}(t,x),\rho^{\varepsilon}(t,x),\theta^{\varepsilon}(t,x))\\
\forall x\in\T,\quad\nu_{0,x}^{\varepsilon}=c_{0}^{\varepsilon}(x)\delta
_{(1,\rho_{+,0}(x),\theta_{+,0}(x))}+(1-c_{0}^{\varepsilon}(x))\delta
_{(0,\rho_{-,0}(x),\theta_{-,0}(x))}.
\end{array}
\right.  \label{Young}%
\end{equation*}
By slightly abusing the notations, we will also denote by $\nu^{\varepsilon
}\in C(0,T,\cM(\T\times K))$ the application
\[
\left\{
\begin{array}
[c]{l}%
\nu^{\varepsilon}:\left[  0,T\right]  \rightarrow\mathcal{M}\left(
\mathbb{T}\times K\right)  \\
\forall(\beta,\psi)\in C(K)\times C\left(  \mathbb{T}\right), \left\langle \nu^{\varepsilon},\quad\beta\times\psi\right\rangle :=\int_{0}^{1}%
\beta\left(  \rho^{\varepsilon}\left(  t,x\right)  ,c^{\varepsilon}\left(
t,x\right)  ,\theta^{\varepsilon}\left(  t,x\right)  \right)  \psi\left(
x\right)  dx
\end{array}
\right.
\]
Then, there exists $\nu\in C(0,T,\cM(\T\times K))\cap\left(  L^{1}\left(
\left(  0,T\right)  \times\mathbb{T}\right)  ,C\left(  K\right)  \right)
^{\prime}$ such that
\[
\nu^{\varepsilon}\tend{\varepsilon}{0}\nu\text{ in }C(0,T,\cM(\T\times
K))\cap\left(  L^{1}\left(  \left(  0,T\right)  \times\mathbb{T}\right)
,C\left(  K\right)  \right)  ^{\prime},
\]
which verifies the following equation:%
\begin{equation}
\left\{
\begin{array}
[c]{l}%
\partial_{t}\nu_{t,x}+\partial_{x}(u\nu_{t,x})-\dfrac{\sigma+p}{\mu}\nu
_{t,x}-\partial_{\rho}\left(  \dfrac{\rho(\sigma+p)}{\mu}\nu_{t,x}\right)
+\partial_{\theta}\left(  \dfrac{\sigma(\sigma+p)}{\cv\mu\rho}\nu
_{t,x}\right)  =0,\\
\lim\limits_{t\rightarrow0}\left\langle \nu_{t,x},\beta\right\rangle =\lim
\limits_{\varepsilon\rightarrow0}\left\langle \nu_{0,x}^{\varepsilon}%
,\beta\right\rangle ,\text{ \ \ }\forall\beta\in C(K).
\end{array}
\right.  \label{KinEq}%
\end{equation}
\end{theorem}%

Some observations are necessary in order to clarify some aspects of the above theorem:

\begin{remark}
    A delicate point in the analysis is to obtain the $L^2_{t,x}$-convergence of the sequence $(\sigma_\varepsilon)_{\varepsilon>0}$, \eqref{strongconvsigmaeps}. This point is crucial in order to be able to prove that the Young measures verify the equation \eqref{Young}, contrary to the barotropic case, where a bound in $L^2(0,T,H^1(\T))$ was sufficient, using some compensated compactness lemma. 
\end{remark}

\begin{remark}
Note that $\nu^{\varepsilon}$ is weakly-measurable as for each $\beta\in
C(K)$, $(t,x)\mapsto\ps{\nu^{\varepsilon}_{t,x}}{\beta}\in L^{\infty}([0,T]\times
\T)$. The function $\nu:[0,T]\times\T\rightarrow\Pro(K),(t,x)\mapsto\nu_{t,x}$
is weakly measurable as a limit of weakly measurable functions.
\end{remark}

\begin{remark}\label{Main2.1}
As $K$ is compact, each $\beta\in C(K)$ is uniformly bounded. In particular,
up to a subsequence, $\beta(c^{\varepsilon},\rho^{\varepsilon},\theta
^{\varepsilon})$ converges weakly$-\star$ in $L^{\infty}([0,T]\times\T)$ to a
certain $\overline{\beta(c,\rho,\theta)}$. As $(C(K),\Vert\cdot\Vert_{\infty
})$ is separable, we have
\[
\forall\beta\in C(K),\quad\beta(c^{\varepsilon},\rho^{\varepsilon}%
,\theta^{\varepsilon}) \tendf{\varepsilon}{0} \overline{\beta(c,\rho,\theta)}
\text{ in }L^{2}([0,T]\times\T),
\]
modulo the extraction of a subsequence. By uniqueness of the limit, we obtain
\[
\forall\beta\in C(K),\quad\ps{\nu_{t,x}}{\beta}= \overline{\beta(c,\rho
,\theta)}.
\]
Thus, the family of Young measures $\nu_{t,x}$ characterizes all weak limits
of functions of $(c^{\varepsilon},\rho^{\varepsilon},\theta^{\varepsilon})$
which is exactly what we used to write the equations for $c,u$ and $\sigma$.
\end{remark}

It turns out that system \eqref{system_general} is too general for our purposes: indeed, as explained in the Section 2 (see also
\cite{BrHi,BrHi2,BrBuLa1,BrBuLa2}) when considering mixtures of two compressible fluids, $(c^{\varepsilon},\rho^{\varepsilon},\theta^{\varepsilon
})$ will oscillate between two states $(\alpha_{\pm},\rho_{\pm},\theta_{\pm
})$ corresponding to the characteristics of the two components of the mixtures
which we denote $+$ respectively $-$. Of course, as explained earlier,
$\alpha_{\pm}$ is interpreted as the volume fraction of the component $\pm$.
Mathematically, this is translated by asking the family of Young measures
\eqref{Young} to be at any time a convex combination of Dirac masses:%
\begin{equation}
\nu_{t,x}=\alpha_{+}\left(  t,x\right)  \delta_{\left(  1,\rho
_{+}\left(  t,x\right)  ,\theta_{+}\left(  t,x\right)  \right)  }+\alpha
_{-}\left(  t,x\right)  \delta_{\left(  0,\rho_{-}\left(  t,x\right)
,\theta_{-}\left(  t,x\right)  \right)  }.\label{convex_comb_diracs}%
\end{equation}
In view of the previous remark, this amounts to saying that%
\begin{equation}
\overline{\beta(c,\rho,\theta)}={\alpha}_{+}\left(  t,x\right)  \beta\left(
1,\rho_{+}\left(  t,x\right)  ,\theta_{+}\left(  t,x\right)  \right)
+\alpha_{-}\left(  t,x\right)  \beta\left(  0,\rho_{-}\left(  t,x\right)
,\theta_{-}\left(  t,x\right)  \right)  .\label{characterization}%
\end{equation}
This is precisely what was asked in \eqref{ansatz}. Of course, using $\left(
\text{\ref{characterization}}\right)  $ in system $\left(
\text{\ref{system_general}}\right)  $ we are able to interpret the terms
$$\left\langle \nu,\frac{1}{\mu\left(  c\right)  }\right\rangle ,\quad
\left\langle \nu,\frac{c}{\mu\left(  c\right)  }\right\rangle ,\quad
\left\langle \nu,\frac{cp\left(  c,\rho,\theta\right)  }{\mu\left(  c\right)
}\right\rangle $$
with respect to $(\alpha_{\pm},\rho_{\pm},\theta_{\pm})$ and as such
to obtain a closed system of partial differential equations for $(\alpha_{\pm},\rho_{\pm},\theta_{\pm
},u)$ namely the macroscopic averaged system. It turns out that the structure $\left(  \text{\ref{convex_comb_diracs}%
}\right)$ i.e. convex combination of Dirac masses is propagated in time from the initial data:

\begin{theorem}
\label{Theorem_main3}Assume the same hypothesis and notations as in Theorem
\ref{Theorem_main2}. Moreover, assume the existence of $\left(  \alpha
_{+,0},\rho_{\pm,0},\theta_{\pm,0}\right)  \in L^{\infty}\left(
\mathbb{T}\right)  $ such that%
\[
\nu_{0,x}=\alpha_{+,0}\left(  x\right)  \delta_{\left(  1,\rho_{+,0}\left(
x\right)  ,\theta_{+,0}(x)\right)  }+\left(  1-\alpha_{+,0}\left(  x\right)
\right)  \delta_{\left(  0,\rho_{-,0}\left(  x\right)  ,\theta_{-,0}%
(x)\right)  }.
\]
Then there exists some $\alpha_{\pm},\rho_{\pm},\theta_{\pm}\in L^{\infty
}([0,T]\times\T)^{3}$ such that
\begin{equation*}
\nu_{t,x}=\alpha_{+}\left(  t,x\right)  \delta_{\left(  1,\rho_{+}\left(
t,x\right)  ,\theta_{+}\left(  t,x\right)  \right)  }+\alpha_{-}\left(
t,x\right)  \delta_{\left(  0,\rho_{-}\left(  t,x\right)  ,\theta_{-}\left(
t,x\right)  \right)  }.
\end{equation*}
These functions can be characterized with respect to the sequence
$(c^{\varepsilon},\rho^{\varepsilon},\theta^{\varepsilon},u^{\varepsilon})$
solution "\`{a} la Hoff" of \eqref{NS} with initial data $(c_{0}%
^{\varepsilon},\rho_{0}^{\varepsilon},\theta_{0}^{\varepsilon},u_{0}%
^{\varepsilon}):$
\begin{align*}
&  c^{\varepsilon}\tend{\varepsilon}{0}\alpha_{+},\text{ }\left(
1-c^{\varepsilon}\right)  \tend{\varepsilon}{0}\alpha_{-}\text{ weakly-}%
\star\text{in }L^{\infty}([0,T]\times\T),\\
&  c^{\varepsilon}\rho^{\varepsilon},c^{\varepsilon}\theta^{\varepsilon
}\tend{\varepsilon}{0}\alpha_{+}\rho_{+},\alpha_{+}\theta_{+}\text{
weakly-}\star\text{in }L^{\infty}([0,T]\times\T)\text{ },\\
&  (1-c^{\varepsilon})\rho^{\varepsilon},(1-c^{\varepsilon})\theta
^{\varepsilon}\tend{\varepsilon}{0}\alpha_{-}\rho_{-},\alpha_{-}\theta
_{-}\text{ weakly-}\star\text{in }L^{\infty}([0,T]\times\T)\text{ },\\
&  u^{\varepsilon}\tendf{\varepsilon}{0}u\text{ in }H^{1}([0,T]\times\T),\\
&  u^{\varepsilon}\tend{\varepsilon}{0}u\text{ in }L^{2}(0,T,L^{2}(\T)),
\end{align*}
The functions $\left(  \alpha_{\pm},\rho_{\pm},\theta_{\pm},u\right)  $ verify
the system \eqref{systmacro2}.
\end{theorem}

The proof of Theorem \ref{Theorem_main3} follows the ideas introduced in
\cite{BrHi2}: 

\begin{itemize}
\item We show that given $u=u\left(  t,x\right)$, $\sigma = \sigma(t,x)$ having the regularity announced in Definition \eqref{defhoff}, the equation $\left(  \text{\ref{KinEq}%
}\right)  $ has at most one solution;

\item In a second step we construct a solution that verifies  $\left(
\text{\ref{KinEq}}\right)  $ and using uniqueness, we conclude.
\end{itemize}
Theorem \ref{Theorem_main3} allows us to characterize the terms 
$$\left\langle \nu,\frac{1}{\mu\left(  c\right)  }\right\rangle ,\quad
\left\langle \nu,\frac{c}{\mu\left(  c\right)  }\right\rangle , \quad
\left\langle \nu,\frac{cp\left(  c,\rho,\theta\right)  }{\mu\left(  c\right)
}\right\rangle,$$
to obtain finally the macroscopic averaged system \eqref{systmacro2}.

%%%%%%% Section 4
\section{Global solution and uniqueness for 
Navier-Stokes with low 
regularity temperature}\label{glob_sol}

\subsection{Energy estimates and bounds}\label{NRJ}
To justify the formal computations that follow, we assume that $(c,\rho,u,\theta)$ is smooth, and that $\rho>0$, $\theta>0$. The existence of such solutions is granted by the result of J. Li \cite{Li}, with minor modifications. We first present some notations and properties of operators. Then we present the conservation of mass and total energy and conclude with some bounds including an upper bound on the density and a lower bound on the temperature. We then present a key point in our paper which is a bound on the stress tensor $\sigma$, following an idea from J. Li that had to be adapted to the presence of the color function $c$. We can then give a bound on $\rho \theta$ and conclude with the lower bound for the density $\rho$ and the upper bound for the temperature $\theta$. The constants we will define in this section depend on $\mu_\pm,\gamma_\pm,\cv_\pm$ and $T$. For readability reasons, these dependencies will be implicit.

\subsubsection{Notations and properties of operators}
%\subsubsection{Total derivative and some reformulations} \label{Reg}
Under the assumption $\partial_x u \in L^1(0,T,L^\infty(\T))$, $u$ is Lipschitz continuous in space and the Eulerian and Lagrangian points of view are morally equivalent. We will often be switching back and forth in what follows, and this subsection gives some elementary tools in order to make this transition as efficient as possible.

\bigskip

\noindent {1) \it The operator \texorpdfstring{$D_t$}{Dt}:}
Recall the notation $D_t(f) = \dot{f} = \partial_t f + u\partial_x f$ of the total derivative of $f$. This definition comes naturally, since defining $X(t,s,x)$ as the solution of

\begin{empheq}[left=\empheqlbrace]{align*}
    \forall (t,x)\in [0,T]\times \T,\quad \partial_t X(t,s,x) &= u(t,X(t,s,x))\\
    \forall x\in \T,\quad X(s,s,x)&=x
\end{empheq}
we get
\begin{equation*}
    \frac{d}{dt} f(t,X(t,s,x)) = (D_t f)(t,X(t,s,x)).
\end{equation*}
The total derivative verifies the classical equalities of a derivative
\begin{equation*}
    D_t(fg)=D_t(f)g + fD_t(g)
\end{equation*}
 and
\begin{equation}\label{dercomposee}
    D_t(\varphi\circ f) = D_t(f)\,\varphi'(f)
\end{equation}
if $\varphi :\R\rightarrow\R$. 
In particular,
\begin{equation}\label{cquisort0}
    D_t(cf) = cD_t(f).
\end{equation}
Moreover,
\begin{equation}\label{Dtdt}
    \int_0^1 \rho D_t(f) = \frac{d}{dt}\int_0^1 \rho f.
\end{equation}

\medskip

\noindent {2) \it The operator \texorpdfstring{$I$}{I}.}
We define an operator $I$, inverting $D_t$, by
\begin{equation*}
    I(f) = \int_0^t f(s,X(s,t,x)) ds
\end{equation*}
in order to get
\begin{equation}\label{Ihihi}
    D_t(I(f)) = f,\quad I(D_t(f)) = f - f_0(X(0,t,x)).
\end{equation}
Moreover, we have
\begin{equation}\label{cquisort}
    I(c f) = c I(f).
\end{equation}

\medskip

\noindent 3) {\itshape A switching formula:}
In the calculations that follow, we may need to switch the symbols $D_t$ and $\partial_x$. We use the formula
\begin{equation}\label{switch}
    \partial_x D_t(f) - D_t(\partial_x f) = (\partial_x f)(\partial_x u)
\end{equation}
%\begin{equation}
%    \int_0^1 (\partial_x f) D_t(g) dx = \frac{d}{dt}\int_0^1g(\partial_x f) dx  + \int_0^1 D_t(f)(\partial_x g)dx.
%    \label{form6}
%\end{equation}

\medskip

\noindent {4) \it The operator \texorpdfstring{$\partial_x^{-1}$}{partial x -1}:}
We define the operator $\partial_x^{-1}$ on the space of zero-mean-space functions by
\begin{equation*}
    \partial_x^{-1} f = \int_0^1 \int_y^x f(z)dz dy
\end{equation*}
in order to get
\begin{equation*}
    \partial_x \partial_x^{-1}f = f
\end{equation*}
\begin{equation*}
    \partial_x^{-1}\partial_x f = f - \int_0^1 f.
\end{equation*}
Moreover, we have
\begin{equation}\label{dx-1}
    \vert \partial_x^{-1} f \vert \leq \Vert f\Vert_1.
\end{equation}
Note that
\begin{equation}\label{contedete}
    D_t\partial_x^{-1} f = \partial_x^{-1} (\partial_t f + \partial_x (uf))+ \int_0^1 uf.
\end{equation}

%\subsubsection{Change of variable in integrals}
%We have

%\begin{equation*}
%    \partial_x X(t,s,x) = \frac{\rho(s,x)}{\rho(t,X(t,s,x))}>0.
%\end{equation*}

%In particular,

%\begin{equation}\label{changevariable}
%    \int_0^1 f(t,y) \rho(t,y) dy = \int_{X(t,s,0)}^{X(t,s,1)} f(t,X(t,s,x))\rho(s,X(s,t,x))dx = \int_0^1 f(t,X(t,s,x))\rho(s,X(s,t,x))dx.
%\end{equation}

\subsubsection{Conservative quantities and bounds related to \texorpdfstring{$\rho$, $\theta$, $\rho \theta$ and $\sigma$}{rho, theta, rho theta and sigma}.}

\noindent{1) \it New formulation of the system:}
From the system \eqref{NS} we deduce the equalities:
\begin{subequations}\label{NS_L}
\begin{empheq}[left=\empheqlbrace]{alignat=1}
\dot{c} &= 0 \label{c_L}\\
\dot{\rho}&=-\rho \partial_x u
\label{cont_L}\\
\rho \dot{u}&=\partial_x \sigma
\label{momts_L}\\
\rho \dot{e} &= \sigma\partial_x u
\label{energint_L}
\end{empheq}
\end{subequations}
\noindent{2) \it Conservation of mass and total energy:} Integrating in space \eqref{cont} gives us conservation 
of total mass:
\begin{equation*}
   \cM(t):= \int_0^1\rho = \int_0^1\rho_0 =\cM_0 \label{mass}.
\end{equation*}
Moreover, integrating in space \eqref{energ} gives conservation of total energy:
\begin{equation*}
    \cE(t):= \int_0^1 \rho E = \int_0^1 \rho_0 E_0 = \cE_0. \label{enrgtot}
\end{equation*}
Remark that 
\begin{equation*}
{\cal{M}}_0\leq \overline{\rho_0},\quad {\cal E}_0\leq \overline{\rho_0}(R_{max}\overline{\theta_0} + C_0^2/2).
\end{equation*}
We deduce
\begin{equation}
    \int_0^1 p = \int_0^1 (\gamma-1)\rho e\leq (\gamma_{max}-1)\cE_0,
\label{p}
\end{equation}
\begin{equation}
    \int_0^1 \frac{\rho u^2}{2} \leq \cE_0.
\label{rhou2}
\end{equation}
Finally, using Cauchy-Schwartz inequality,
\begin{equation}
    \int_0^1 \vert\rho u\vert = \sqrt{2}\int_0^1 \sqrt{\rho}\sqrt{\rho}\frac{\vert u\vert}{\sqrt{2}}\leq \sqrt{2}\sqrt{\int_0^1\rho}\sqrt{\int_0^1 \frac{\rho u^2}{2}}\leq  \sqrt{2\cM_0\cE_0}.
\label{rhou}
\end{equation}
Remark that, if $v$ is a constant, the change of unknowns
\begin{equation*}
    (c(t,x),\rho(t,x),u(t,x),E(t,x)) \rightarrow (c(t,x-vt),\rho(t,x-vt), u(t,x-vt)+v, E(t,x-vt))
\end{equation*}
does not change the problem (it is the Galilean invariance principle). Choosing
\begin{equation*}
    v = -\frac{1}{\cM_0}\int_0^1\rho_0 u_0
\end{equation*}
we can suppose 
\begin{equation*}
    \int_0^1\rho_0u_0 = 0
\end{equation*}
then, integrating \eqref{momts} in space
\begin{equation}\label{eq:momtsnul}
    \int_0^1\rho u = \int_0^1\rho_0 u_0 = 0.
\end{equation}
\begin{prop}
There exists some $H_1 = H_1({\cal M}_0,{\cal E}_0)>0$ such that
\begin{equation}\label{Isigma}
\vert I(\sigma)\vert\leq H_1
\end{equation}
\end{prop}
\begin{proof}
Let's apply $\partial_x^{-1}$ to \eqref{momts}:
\begin{equation*}
    \partial_x^{-1}\partial_t (\rho u) + \partial_x^{-1}\partial_x(\rho u^2) = \partial_x^{-1}\partial_x\sigma = \sigma - \int_0^1 \sigma.
\end{equation*}
By \eqref{contedete} we get
\begin{equation*}
    D_t(\partial_x^{-1}(\rho u)) - \int_0^1\rho u^2 = \sigma - \int_0^1\sigma.
\end{equation*}
Applying $I$ and using \eqref{Ihihi} we obtain
\begin{equation*}
    I(\sigma)(t,x) =\partial_x^{-1}(\rho u)(t,x)- \partial_x^{-1}(\rho_0 u_0)(X(0,t,x)) -\int_0^t\int_0^1\rho u^2 + \int_0^t\int_0^1\sigma.
\end{equation*}
Therefore, by \eqref{dx-1}, \eqref{rhou}, \eqref{rhou2} and \eqref{p},
\begin{equation}\label{Isigma0}
    \vert I(\sigma)\vert \leq 2 \sqrt{2\cM_0\cE_0} + T\cE_0 + T(\gamma_{max}-1)\cE_0 + \left\vert \int_0^t\int_0^1 \mu\partial_x u \right\vert.
\end{equation}
Moreover, by \eqref{c}
\begin{equation}\label{astuuce}
    \left\vert\int_0^t\int_0^1 \mu\partial_x u\right\vert = \left\vert-\int_0^t\int_0^1 u \partial_x \mu \right\vert= \left\vert\int_0^t \frac{d}{dt}\int_0^1 \mu\right\vert = \left\vert \int_0^1\mu(t) - \int_0^1\mu_0(t)\right\vert\leq \mu_{max}.
\end{equation}
Finally, by \eqref{Isigma0} and \eqref{astuuce} we get the result.
\end{proof}
\medskip
\begin{prop}
There exists some $\overline{\rho} = \overline{\rho}(\overline{\rho_0}, {\cal E}_0)>0$ such that
\begin{equation}\label{upperrho}
\rho \leq \overline{\rho} \quad\text{a.e.}.
\end{equation}
\end{prop}
\begin{proof}
Using \eqref{cquisort0} then \eqref{cont_L} we get 
\begin{equation*}
    D_t(\mu\ln(\rho)) = \mu D_t(\ln(\rho)) = -\mu\partial_x u = -\sigma - p
\end{equation*}
then, applying $I$, using \eqref{Ihihi} and the fact that $I(p)\geq 0$,
\begin{align*}
    \mu\ln(\rho) &= \mu_0(X(0,t,x))\ln(\rho_0(X(0,t,x))) - I(\sigma) - I(p)\label{macrondemission}
    \\&\leq \mu_{max}\vert \ln(\overline{\rho_0})\vert + \vert I(\sigma)\vert
    \\&\leq \mu_{max}\vert\ln(\overline{\rho_0})\vert + H_1.
\end{align*}
Dividing by $\mu$ then taking the exponential we get
\begin{equation}
\rho\leq \overline{\rho_0}\exp\left(\frac{\mu_{max}+H_1}{\mu_{min}}\right),
\end{equation}
hence the result.
\end{proof}
\begin{prop}
There exists some $\underline{\theta}=\underline{\theta}(\underline{\theta_0},\overline{\rho_0},{\cal E}_0)>0$ such that
\begin{equation}\label{lowertheta}
\theta\geq\overline{\theta}\quad \text{a.e.}.
\end{equation}
\end{prop}
\begin{proof}
Using \eqref{energint_L}, we have
\begin{align*}
    \mu\rho \dot{e} = \sigma(\mu\partial_x u) = \sigma^2 + \sigma p = (\sigma + p/2)^2 - p^2/4 \geq -p^2/4
\end{align*}
so
\begin{equation*}
    \dot{\theta} \geq - \frac{R_{max}^2\overline{\rho} \theta^2}{4\mu_{min}\cv_{min}}.
\end{equation*}
Then dividing by $-\theta^2$ and using \eqref{dercomposee}
\begin{equation*}
    D_t\left(\frac{1}{\theta}\right) \leq \frac{R_{max}^2\overline{\rho}}{4\mu_{min}\cv_{min}}.
\end{equation*}
Applying $I$ and using \eqref{Ihihi} we obtain
\begin{equation*}
    \frac{1}{\theta} \leq \frac{1}{\theta_0(X(0,t,x))} + t\frac{R_{max}^2\overline{\rho}}{4\mu_{min}\cv_{min}} \leq \frac{1}{\underline{\theta_0}} + T\frac{R_{max}^2\overline{\rho}}{4\mu_{min}\cv_{min}}
\end{equation*}
hence the result, denoting
\begin{equation*}
    \underline{\theta}:= \frac{1}{1/\underline{\theta_0}+T\dfrac{R_{max}^2\overline{\rho}}{4\mu_{min}\cv_{min}}}>0.
\end{equation*}
\end{proof}
\begin{prop}\label{sigma} 
Assuming that smooth solutions of the system occur with
\eqref{upperrho} and estimates in subsection \ref{lowertheta} satisfied. Then 
$\sigma$ satisfies the equation
$$\dot{\sigma} = \mu \partial_x\left(\frac{\partial_x\sigma}{\rho}\right)  - \gamma\sigma\partial_x u.
$$
Moreover, there exists some $C_1 = C_1(\overline{\rho_0},\overline{\theta_0},C_0)>0$ and $C_2 = C_2(\overline{\rho_0},\overline{\theta_0},C_0)>0$ such that
\begin{equation}\label{eq:Abound}
\sup_{[0,T]}\int_0^1\sigma^2 + \int_0^T\int_0^1 (\partial_x\sigma)^2\leq C_1,
\end{equation}
\begin{equation}\label{eq:sigmainf}
    \int_0^T\Vert\sigma\Vert_\infty^2 \leq C_2.
\end{equation}
\end{prop}
\begin{proof}
Using \eqref{NS_L} we get
\begin{align}
    \dot{\sigma} &= \mu D_t(\partial_x u) - RD_t(\rho \theta)\nonumber
    \\& = \mu \partial_x \dot{u} - \mu(\partial_x u)^2 - R\dot{\rho}\theta - R\rho \dot{\theta}\nonumber
    \\& = \mu \partial_x\left(\frac{\partial_x\sigma}{\rho}\right) - \mu(\partial_x u)^2  + R(\partial_x u)\rho\theta - (\gamma-1)\sigma\partial_x u\nonumber 
    \\& = \mu \partial_x\left(\frac{\partial_x\sigma}{\rho}\right)  - \gamma\sigma\partial_x u.\label{sigmapoint}
\end{align}
Thus if $\beta:\R\rightarrow\R$ is $C^1$,
\begin{align*}
    D_t(\beta(\sigma)) = \dot{\sigma}\beta'(\sigma) = \mu\partial_x\left(\frac{\partial_x\sigma}{\rho}\right) \beta'(\sigma) - \gamma(\sigma\partial_x u)\beta'(\sigma)
\end{align*}
hence, dividing by $\mu$,
\begin{equation}\label{sigmarenormalize}
    \partial_t \left(\frac{\beta(\sigma)}{\mu}\right) + \partial_x \left( u \frac{\beta(\sigma)}{\mu}\right) - \partial_x \left(\frac{\partial_x\sigma}{\rho}\right)\beta'(\sigma) = \frac{\partial_x u}{\mu}(\beta(\sigma)-\gamma\sigma\beta'(\sigma))
\end{equation}
Choosing $\beta(x)={(x)^+}^2/2$ in \eqref{sigmarenormalize} (where $(x)^+ = \max(0,x)$), then integrating in space, using the fact that
\begin{equation*}
(\partial_x \sigma)(\partial_x(\sigma)^+) = (\partial_x(\sigma)^+)^2,    
\end{equation*}
we get
\begin{equation}\label{sigmaplus}
    \frac{d}{dt}\int_0^1\frac{{(\sigma)^+}^2}{2\mu} + \int_0^1 \frac{(\partial_x(\sigma)^+)^2}{\rho} = \frac{\partial_x u}{\mu}\left(\frac{1}{2}-\gamma\right)\sigma_+^2
\end{equation}
But 
\begin{equation*}
    (\partial_x u) \I_{\sigma\geq 0} = \frac{\sigma + p}{\mu}\I_{\sigma\geq 0} \geq 0.
\end{equation*}
So \eqref{sigmaplus} gives
\begin{equation*}
    \frac{d}{dt}\int_0^1 \frac{{(\sigma)^+}^2}{2\mu} + \int_0^1 \frac{(\partial_x(\sigma)^+)^2}{2\rho} \leq 0
\end{equation*}
hence
\begin{equation*}
    \int_0^1 \frac{{(\sigma)^+}^2}{2\mu} + \int_0^T\int_0^1 \frac{(\partial_x(\sigma)^+)^2}{\rho} \leq \int_0^1 \frac{{{(\sigma_0)}^+}^2}{2\mu_0} \leq \int_0^1 \frac{\sigma_0^2}{2\mu_0}.
\end{equation*}
Finally, using the equality $\vert\sigma\vert = 2(\sigma)^+ - \sigma$, we get
\begin{align}\label{vertsigmavert}
    \int_0^1\vert \sigma\vert &\leq \mu_{max}\int_0^1 \frac{\vert \sigma\vert}{\mu} = \mu_{max}\int_0^1 \frac{2(\sigma)^+}{\mu} - \mu_{\max}\int_0^1\frac{\sigma}{\mu}.
\end{align}
But 
\begin{equation*}
    -\mu_{max}\int_0^1 \frac{\sigma}{\mu} = -\mu_{max}\int_0^1 \partial_x u + \mu_{max}\int_0^1 \frac{p}{\mu} \leq \frac{\mu_{max}}{\mu_{min}} \int_0^1 p \leq \frac{\mu_{max}}{\mu_{min}}(\gamma_{max}-1)\cE_0.
\end{equation*}
so by Young's inequality, \eqref{vertsigmavert} gives
\begin{equation*}
    \int_0^1 \vert \sigma\vert \leq \int_0^1 \frac{{(\sigma)^+}^2}{2\mu} + \int_0^1\frac{2}{\mu}+\frac{\mu_{max}}{\mu_{min}}(\gamma_{max}-1)\cE_0
\end{equation*}
hence by \eqref{vertsigmavert}
\begin{equation}
    \int_0^1 \vert \sigma\vert \leq \int_0^1 \frac{\sigma_0^2}{2\mu} + \frac{2}{\mu_{min}}+\frac{\mu_{max}}{\mu_{min}}(\gamma_{max}-1)\cE_0 =:H_2.
\end{equation}
Evaluating \eqref{sigmarenormalize} with $\beta(x)=x^2/2$, then integrating in space, we get
\begin{align}
    \frac{d}{dt} \int_0^1 \frac{\sigma^2}{2\mu} + \int_0^1\frac{(\partial_x\sigma)^2}{\rho} &= \int_0^1\frac{\partial_x u}{\mu}\left(\frac{1}{2}-\gamma\right)\sigma^2\nonumber
    \\&=\int_0^1\frac{\frac{1}{2}-\gamma}{\mu}\sigma^3 + \int_0^1\frac{\frac{1}{2}-\gamma}{\mu^2}\sigma^2p\nonumber
    \\&\leq \int_0^1\frac{1}{\mu^2}\left(\gamma-\frac{1}{2}\right)\vert \sigma\vert \sigma^2\nonumber
    \\&\leq \frac{H_2}{\mu_{min}^2}\left(\gamma_{max}-\frac{1}{2}\right)\Vert\sigma\Vert_\infty^2 =: H_3 \Vert\sigma\Vert_\infty^2.\label{sigmaineg}
\end{align}
Using now the Gagliardo-Nirenberg
\begin{equation}
    \Vert \sigma\Vert_\infty^2 \leq \Vert \sigma\Vert_{2}^2 + 2 \Vert \sigma\Vert_{2}\Vert \partial_x \sigma\Vert_{2}
\end{equation}
we get by Young's inequality
\begin{align*}
    H_3\Vert \sigma\Vert_\infty^2 &= H_3\int_0^1 \sigma^2 + 2H_3\left(\int_0^1 \sigma^2\right)^{1/2}\left(\int_0^1 (\partial_x\sigma)^2\right)^{1/2}
    \\&\leq 2\mu_{max}(H_3 + 2H_3^2\overline{\rho})\int_0^1 \frac{\sigma^2}{2\mu} + \frac{1}{2}\int_0^1 \frac{(\partial_x \sigma)^2}{\rho}
\end{align*}
then \eqref{sigmaineg} gives
\begin{align*}
    \frac{d}{dt}\int_0^1\frac{\sigma^2}{2\mu} + \frac{1}{2}\int_0^1 \frac{(\partial_x\sigma)^2}{\rho} \leq 2\mu_{max}(H_3+2 H_3^2\overline{\rho})\int_0^1 \frac{\sigma^2}{2\mu}.
\end{align*}
Hence, by Grönwall's lemma,
\begin{equation*}
   \int_0^1 \frac{\sigma^2}{2\mu} + \frac{1}{2}\int_0^t\int_0^1 \frac{(\partial_x\sigma)^2}{\rho} \leq \exp(2\mu_{max}(H_3+2H_3^2\overline{\rho})T)\int_0^1\frac{\sigma_0^2}{2\mu_0}.
\end{equation*}
As $\mu\geq \mu_{min}$ and $\rho\leq\overline{\rho}$, we get \eqref{eq:Abound}. Therefore, by the inclusions of Sobolev, there exists some $C_2>0$ such that \eqref{eq:sigmainf} holds.

\ 

Roughly speaking, \eqref{eq:Abound} is the first Hoff energy get in \cite{BrBuLa1} in the barotrope case. Indeed, by \eqref{momts_L}, $\rho \dot{u}^2=\dfrac{(\partial_x\sigma)^2}{\rho}$ and $\sigma^2 = (\mu \partial_x u - p)^2$ looks like  $(\partial_x u)^2$.
\end{proof}
\begin{prop}\label{prop:5}
There exists some $\overline{p}=\overline{p}(\overline{\rho_0},\overline{\theta_0}, C_0)>0$ such that 
\begin{equation}
p\leq \overline{p} \quad \text{a.e.}.
\end{equation}
\end{prop}
\begin{proof}
From \eqref{cont_L} and \eqref{energint_L} we deduce
\begin{align*}
    \mu\cv D_t(\rho\theta) &= \mu\dot{\rho}\cv\theta + \mu\cv \rho \dot{\theta} = -\cv\rho \theta (\mu\partial_x u) + \sigma(\mu\partial_x u) \\&= -\cv \rho\theta \sigma - \cv\rho\theta p  + \sigma^2 +\sigma p 
    \\& = -\cv R (\rho\theta)^2 + \cv(\gamma-2)\rho\theta\sigma + \sigma^2
    \\& = -\cv R\left(\rho\theta-\frac{\gamma-2}{2 R}\sigma\right)^2 + \left(1 + \frac{(\gamma-2)^2}{4(\gamma-1)}\right)\sigma^2
    \\&\leq \left(1+\frac{(\gamma-2)_{max}^2}{4(\gamma_{min}-1)}\right)\Vert \sigma\Vert_\infty^2
\end{align*}
hence, denoting
\begin{equation*}
    H_4 := \frac{1}{\mu_{min}\cv_{min}}\left(1+\frac{(\gamma-2)_{max}^2}{4(\gamma_{min}-1)}\right),
\end{equation*}
\begin{equation*}
    D_t(\rho\theta)\leq H_4 \Vert \sigma\Vert_\infty^2.
\end{equation*}
Multiplying by $R$, using \eqref{cquisort0}, then applying $I$ and using \eqref{Ihihi}, \eqref{cquisort} we get 
\begin{align*}
    p=R\rho\theta \leq R(H_1(X(0,t,x)))\rho_0(X(0,t,x))\theta_0(X(0,t,x)) + R_{max}H_4\int_0^t \Vert\sigma\Vert^2_\infty 
\end{align*}
then from \eqref{eq:sigmainf}
\begin{equation}\label{upperp}
    p\leq R_{max}\overline{\rho_0}\overline{\theta_0} + R_{max}H_4 C_2.
\end{equation}
\end{proof}
\medskip
\begin{prop}
There exists some $\underline{\rho}=\underline{\rho}(\underline{\rho_0},\overline{\rho_0},\overline{\theta_0}, C_0)>0$ and $\overline{\theta} = \overline{\theta}(\underline{\rho_0},\overline{\rho_0},\overline{\theta_0},C_0)>0$ such that
\begin{equation}\label{lowerrho}
\rho\geq \underline{\rho}\quad \text{a.e.},
\end{equation} 
\begin{equation}\label{uppertheta}
\theta\leq \overline{\theta}\quad\text{a.e.}.
\end{equation}
\end{prop}
\begin{proof}
Using \eqref{Isigma} and \eqref{upperp}, and imitating the computations of 3) we obtain
\begin{align*}
    \mu\ln(\rho) &= \mu_0(X(0,t,x))\ln(\rho_0(X(0,t,x))) - I(\sigma)-I(p)
    \\&\geq -\mu_{max}\vert\ln(\underline{\rho_0})\vert - H_1 - T\overline{p}.
\end{align*}
So, dividing by $\mu$ then applying the exponential function, we get \eqref{lowerrho}. Thanks to \eqref{upperp} and \eqref{lowerrho}, it is obvious to obtain \eqref{uppertheta}.
\end{proof}
\begin{prop}\label{dxu} 
There exists some $C_3 = C_3(\overline{\rho_0},\underline{\rho_0},\overline{\theta_0},C_0)>0$,  $C_4 = C_4(\overline{\rho_0},\underline{\rho_0},\overline{\theta_0},C_0)>0$,  $C_5 = C_5(\overline{\rho_0},\underline{\rho_0},\overline{\theta_0},C_0)>0$ such that
\begin{equation}\label{C8}
    \sup_{[0,T]}\int_0^1 (\partial_x u)^2\leq C_3,
\end{equation}
\begin{equation}
\vert u\vert^2\leq C_3\quad \text{ a.e.},
\end{equation}
\begin{equation}\label{C9}
    \int_0^T \Vert\partial_x u \Vert_\infty^2 \leq C_4,
\end{equation}
\begin{equation}\label{C12}
    \int_0^T \int_0^1 (\partial_t u)^2 \leq C_5.
\end{equation}
\end{prop}
\begin{proof}
Writing
\begin{equation*}
    \partial_x u = \frac{\sigma + p}{\mu}
\end{equation*}
and using Proposition \ref{sigma}, \ref{prop:5}, we obtain immediatly \eqref{C8} and \eqref{C9}. Moreover, writing for all $x,y\in\T$
\begin{equation}
u(x) = u(y) + \int_y^x \partial_z u 
\end{equation}
then multiplying by $\rho(y)$ and integrating on the torus, we get from \eqref{eq:momtsnul}
\begin{equation}
{\cal M}_0 u(x) = \int_0^1\rho(y)\int_y^x \partial_x u 
\end{equation}
hence
\begin{equation}
\vert u(x)\vert\leq \int_0^1\vert \partial_x u\vert \leq C_3^{1/2}.
\end{equation}
Finally, we have
\begin{equation}
\partial_t u = \frac{\partial_x\sigma}{\rho} - u\partial_x u.
\end{equation}
Then, taking the square, integrating on $[0,T]\times\T$ and using some Young's inequality, we obtain
\begin{align}
\int_0^T\int_0^1(\partial_t u)^2&\leq 2\int_0^T\int_0^1 \frac{(\partial_x\sigma)^2}{\rho^2} + 2\int_0^T\Vert u\Vert_\infty^2\int_0^1 (\partial_x u)^2
\\&\leq \frac{2}{\underline{\rho}^2}C_1 + 2TC_3^2.
\end{align}
\end{proof} 
Next, we follow D. Hoff's idea by introducing the weight function %$\kappa:[0,T]\mapsto \kappa(t)$ 
$\kappa(t)=\min(1,t)$
which cancels at $t=0$. Let's briefly explain the principle: if $f$ is a semi-norm of space, let us compare the calculations
\begin{equation}
\int_0^t f' = f(t)-f(0)
\label{sanskappa}
\end{equation}
and 
\begin{equation}
\int_0^t \kappa f' = \kappa(t)f(t)-\int_0^{\min(1,t)} f.    
\label{aveckappa}
\end{equation}
In \eqref{sanskappa}, having a bound on $f(0)$ requires more regularity in the initial conditions. This term does not appear in \eqref{aveckappa}, and the $\int_0^{\min(1,t)} f$ term that does appear has every chance of being bounded by the estimate obtained earlier, if $\kappa'$ doesn't explode.
\begin{prop}
There exists some $C_6 = C_6(\overline{\rho_0},\underline{\rho_0},\overline{\theta_0},C_0)>0$ such that
\begin{equation*}
    \sup_{0\leq t\leq T}\kappa(t)\int_0^1 (\partial_x\sigma)^2 + \int_0^t\kappa \int_0^1 \dot{\sigma}^2\leq C_6.
\end{equation*}
\end{prop}
\begin{proof}
Multiplying \eqref{sigmapoint} by $\dot{\sigma}/\mu$, we get
\begin{equation*}
    \frac{\dot{\sigma}^2}{\mu} = \partial_x \left(\frac{\partial_x\sigma}{\rho}\right)\dot{\sigma} - \frac{\gamma}{\mu}\sigma \dot{\sigma}\partial_x u.\label{64}
\end{equation*}
then integrating in space gives us
\begin{equation}\label{dell}
    \int_0^1\frac{\dot{\sigma}^2}{\mu} = -\int_0^1 (\partial_x \sigma)\frac{\partial_x\dot{\sigma}}{\rho}- \int_0^1\frac{\gamma}{\mu}\sigma \dot{\sigma}\partial_x u.
\end{equation}
But using \eqref{switch},
\begin{equation*}
    D_t\left(\frac{\partial_x\sigma}{\rho}\right) = \frac{D_t(\partial_x\sigma)\rho - D_t(\rho) \partial_x\sigma}{\rho^2}= \frac{D_t(\partial_x\sigma) +(\partial_x u)\partial_x \sigma}{\rho}=\frac{\partial_x \dot{\sigma}}{\rho}
\end{equation*}
so by \eqref{Dtdt}
\begin{equation*}
    \int_0^1 (\partial_x\sigma)\frac{\partial_x\dot{\sigma}}{\rho} = \int_0^1 \rho \frac{\partial_x\sigma}{\rho} D_t\left(\frac{\partial_x \sigma}{\rho}\right) = \frac{1}{2}\int_0^1 \rho D_t \left(\left[\frac{\partial_x\sigma}{\rho}\right]^2\right) = \frac{1}{2}\frac{d}{dt}\int_0^1 \frac{(\partial_x\sigma)^2}{\rho}
\end{equation*}
and we can rewrite \eqref{dell} as
\begin{equation*}
    \frac{1}{2}\frac{d}{dt}\int_0^1 \frac{(\partial_x\sigma)^2}{\rho}+\int_0^1\frac{\dot{\sigma}^2}{\mu}  = - \int_0^1\frac{\gamma}{\mu}\sigma \dot{\sigma}\partial_x u.
\end{equation*}
By Young's inequality we obtain for a certain constant $H_5>0$
\begin{equation*}
    \frac{1}{2}\frac{d}{dt}\int_0^1 \frac{(\partial_x\sigma)^2}{\rho}+\int_0^1\frac{\dot{\sigma}^2}{\mu}  \leq H_5\Vert\partial_x u\Vert_\infty^2(t) \int_0^1 \sigma^2 + \frac{1}{2}\int_0^1 \frac{\dot{\sigma}^2}{\mu}
\end{equation*}
so using \eqref{eq:Abound}
\begin{equation*}
    \frac{d}{dt}\int_0^1 \frac{(\partial_x\sigma)^2}{\rho}+\int_0^1\frac{\dot{\sigma}^2}{\mu}  \leq 2H_5\Vert\partial_x u\Vert_\infty^2(t) \int_0^1 \sigma^2\leq 2C_1H_5\Vert\partial_x u\Vert_\infty^2.
\end{equation*}
and finally, multiplying by $\kappa(t)$ and integrating in time, by \eqref{aveckappa} we get
\begin{align*}
\kappa(t)\int_0^1 \frac{(\partial_x\sigma)^2}{\rho} + \int_0^t\kappa \int_0^1 \frac{\dot{\sigma}^2}{\mu} &\leq 2C_1H_5\int_0^t\kappa\Vert \partial_x u\Vert_\infty^2 + \int_0^{\min(1,t)}\int_0^1 \frac{(\partial_x\sigma)^2}{\rho}
\\&\leq 2C_1C_4H_5 + C_1/\underline{\rho},
\end{align*}
hence the result.
\end{proof}

%From now on, we will call solution "à la Hoff" of \eqref{NS} a weak solution of \eqref{NS} verifying \eqref{D_UL_rho}, \eqref{D_UL_theta}, \eqref{D_A1}, \eqref{D_A1inf}, and moreover \eqref{secondhoff}. Of course, the uniqueness result stays true, and the stability result too.

\subsection{Smooth solutions and stability}
\begin{theorem}\label{stability}
    Let $(c^n,\rho^n,\theta^n,u^n)$ some solutions "à la Hoff" of \eqref{cont}--\eqref{energ} and \eqref{crho} with initial conditions $(c^n_0,\rho^n_0,\theta^n_0,u^n_0)$. Suppose moreover that this sequence is uniformly bounded, and that there exists some $(c_0,\rho_0,\theta_0,u_0)\in L^\infty(\T)^3\times H^1(\T)$ such that 
\begin{equation*}
    c^n_0,\rho^n_0,\theta^n_0 \tend{n}{\infty}c_0,\rho_0,\theta_0 \text{ in }L^2(\T)
\end{equation*}
\begin{equation*}
    u_0^n \tendf{n}{\infty} u_0 \text{ in } H^1(\T).
\end{equation*}
Then, up to a subsequence,  $(c^n,\rho^n,\theta^n,u^n)$ converges strongly in $L^2(0,T,L^2(\T))^3\times L^2(0,T,H^1(\T))$ to some solution "à la Hoff" of \eqref{cont}--\eqref{energ} and \eqref{crho} with initial conditions $(c_0,\rho_0,\theta_0,u_0)$.
\end{theorem}
\begin{proof}
    Using uniform bounds \eqref{upperrho}, \eqref{uppertheta}, \eqref{C8}, after extracting, there exists some 
    
    $(c,\rho,\theta,u)\in L^2(0,T,L^2(\T))^3\times L^2(0,T,H^1(\T))$ such that
\begin{equation}\label{weakconvcrt}
    c^n,\rho^n,\theta^n \tendf{n}{+\infty} c,\rho,\theta \text{ in }L^2(0,T,L^2(\T))
\end{equation}
\begin{equation}\label{weakconvu}
    u^n\tendf{n}{+\infty} u \text{ in }L^2(0,T,H^1(\T)).
\end{equation}
Passing to the limit, we obtain immediatly \eqref{D_UL_rho}-\eqref{D_A1inf}. It remains to show the strong convergence and that $(c,\rho,\theta,u)$ is a solution of \eqref{NS} in the weak sense.
\begin{itemize}
\item[Step 1:]
Let's show that 
\begin{equation}\label{XnstrongX}
    X^n\tend{n}{\infty} X \text{ in }C([0,T]^2\times \T).
\end{equation}
Writing 
\begin{equation*}
    X^n(s,t,x) = x + \int_t^s u^n(\tau,X^n(\tau,t,x))d\tau,
\end{equation*}
\begin{equation*}
    X(s,t,x) = x + \int_t^s u(\tau,X(\tau,t,x))d\tau
\end{equation*}
we get
\begin{align*}
    \vert X^n(s,t,x)&-X(s,t,x)\vert \leq \int_t^s \vert u^n(\tau,X^n(\tau,t,x)) - u(\tau,X(\tau,t,x))\vert d\tau 
    \\&\leq \int_t^s \vert u^n(\tau,X^n(\tau,t,x)) - u^n(\tau,X(\tau,t,x))\vert d\tau + \int_t^s \vert u^n(\tau,X(\tau,t,x))-u(\tau,X(\tau,t,x))\vert d\tau
    \\&\leq \int_t^s \Vert \partial_x u^n(t)\Vert_\infty \vert X^n(\tau,t,x)-X(\tau,t,x)\vert d\tau + \int_0^T\Vert u^n-u\Vert_{L^\infty(\T)}
\end{align*}
so using the Grönwall lemma
\begin{align}
    \vert X^n(s,t,x)-X(s,t,x)\vert &\leq \left(\int_0^T\Vert u^n-u\Vert_{L^\infty(\T)}\right)\exp{\int_t^s \Vert \partial_x u^n(\tau)\Vert_\infty}d\tau\nonumber
    \\&\leq \sqrt{T}e^{\sqrt{C_4T}}\left(\int_0^T\Vert u^n-u\Vert_{L^\infty(\T)}^2\right)^{1/2}.\label{unstrongu}
\end{align}
By \eqref{D_A1inf}, $(u^n)$ is Lipschitz in $C([0,T]\times\T)$. So by Ascoli's theorem, 
\begin{equation}\label{strongconvu}
    u^n \tend{n}{+\infty} u \text{ in }C([0,T]\times\T)
\end{equation}
and \eqref{unstrongu} gives \eqref{XnstrongX}.

\item[Step 2:] Now we are going to show that 
\begin{equation*}
    c^n\tend{n}{\infty} c \text{ in }L^\infty(0,T,L^1(\T)).
\end{equation*}
To do that, we will use the equality
\begin{equation*}
    c^n(t,x) = c_0^n(X^n(0,t,x)).
\end{equation*}
We will shorten $X(0,t,x)$ to $X$ and $X^n(0,t,x)$ to $X^n$ in the following.

We get
\begin{equation}\label{convc}
    \int_0^1 \vert c^n(t,x) - c_0(X)\vert \leq \int_0^1 \vert c_0^n(X^n)-c_0(X^n)\vert + \int_0^1\vert c_0(X^n) - c_0(X)\vert.  
\end{equation}
At this step, that is important to remark that for all $s,t\in [0,T]$ we have the equality 
\begin{equation}\label{changevariable}
    \int_0^1 f(X(s,t,x)) dx = \int_{X(s,t,0)}^{X(s,t,1)}f(y)\frac{\rho(s,y)}{\rho(t,X(t,s,y))}dy = \int_0^1 f(y)\frac{\rho(s,y)}{\rho(t,X(t,s,y))}dy
\end{equation}
and the same replacing $(\rho,X(s,t,x))$ by $(\rho^n, X^n(s,t,x))$.

Therefore, we infer that
\begin{equation}\label{decoup1}
    \int_0^1 \vert c_0^n(X^n)-c_0(X^n)\vert = \int_0^1 \vert c_0^n(y) - c_0(y)\vert \frac{\rho^n(s,y)}{\rho^n(t,X^n(t,0,y))}dy \leq \frac{\overline{\rho}}{\underline{\rho}} \int_0^1 \vert c_0^n(y)-c_0(y)\vert \tend{n}{\infty} 0.
\end{equation}
For the second term on the right hand side of \eqref{convc}, we mollify $c_0$, considering $c_{0,\eta}\in C^1([0,1])$ such that
\begin{equation*}
    c_{0,\eta}\tend{\eta}{0}c_0\text{ in }L^1([0,1]).
\end{equation*}
We have, using again \eqref{changevariable}
\begin{align}\label{blabla}
    \int_0^1 \vert c_0(X^n)-c_0(X)\vert &\leq \int_0^1 \vert c_0(X^n)-c_{0,\eta}(X^n)\vert + \int_0^1 \vert c_{0,\eta}(X^n)-c_{0,\eta}(X)\vert + \int_0^1 \vert c_{0,\eta}(X)-c_0(X)\vert\nonumber
    \\&\leq 2\frac{\overline{\rho}}{\underline{\rho}} \int_0^1 \vert c_0-c_{0,\eta}\vert + \Vert \partial_x c_{0,\eta}\Vert_\infty \Vert X^n-X\Vert_\infty.
\end{align}
Let $\varepsilon>0$. Let's fix $\eta>0$ such that
\begin{equation}
    2\frac{\overline{\rho}}{\underline{\rho}}\int_0^1 \vert c_0 - c_{0,\eta}\vert\leq\varepsilon.
\end{equation}
Then, there exists $N\in\N$ such that
\begin{equation}\label{blablabla}
    \forall n\geq N,\quad \Vert \partial_x c_{0,\eta}\Vert_\infty \Vert X^n-X\Vert_\infty \leq \varepsilon.
\end{equation}
So, from \eqref{blabla}--\eqref{blablabla},
\begin{equation}\label{decoup2}
    \forall n\geq N,\quad \int_0^1 \vert c_0(X^n)-c_0(X)\vert \leq 2\varepsilon.
\end{equation}
Finally, \eqref{decoup1} and \eqref{decoup2} give by \eqref{convc}
\begin{equation}\label{strongconvc}
    c^n \tend{n}{\infty} c_0(X) \text{ in }L^\infty(0,T,L^1(\T)).
\end{equation}
and by uniqueness of the limit, $c_0(X) = c$.

\item[Step 3:] Let $\sigma$ be the weak limit of $\sigma^n$ in $L^\infty(0,T,L^2(\T))\cap L^2(0,T,H^1(\T))$. We deduce from the strong convergence of $c^n$ that
\begin{equation*}
    \sigma = \mu(c)\partial_x u - p 
\end{equation*}
where $p$ is the weak limit of $R(c^n)\rho^n\theta^n$.
\begin{remark}
In the following reasoning, the writing "$\in X$" is to be understood as "uniformly bounded in $n$ in $X$".
\end{remark}
We will abbreviate $f(c^n)$ as $f^n$ for any function $f$ depending on $c^n$. From Proposition \ref{sigma} we get 
\begin{equation}\label{algebr}
    \partial_t\left(\frac{\sigma^n}{\mu^n}\right) =-\partial_x\left(u^n\frac{\sigma^n}{\mu^n}\right) + \partial_x \left(\frac{\partial_x\sigma^n}{\rho^n}\right) -\frac{\gamma^n-1}{\mu^n}\sigma^n\partial_x u^n.
\end{equation}

Let's look at each term of the right hand side of \eqref{algebr}. 

\begin{itemize}
    \item $\sigma^n\in L^\infty(0,T,L^2(\T))$, $\mu^n,u^n\in L^\infty([0,T]\times\T)$, so $\partial_x(\sigma^n/\mu^n)\in L^2(0,T,H^{-1}(\T))$. 
    \item $\partial_x\sigma^n/\rho^n\in L^2(0,T,L^2(\T))$, so $\partial_x(\partial_x\sigma^n/\rho^n)\in L^2(0,T,H^{-1}(\T))$.
    \item $\sigma^n\in L^\infty(0,T,L^2(\T))$, $\partial_x u^n\in L^2(0,T,L^\infty(\T))$ and $(\gamma^n-1)/(\mu^n)\in L^\infty([0,T]\times \T)$, so 
    \begin{equation*}
        -\frac{\gamma^n-1}{\mu^n}\sigma^n\partial_x u^n\in L^2(0,T,L^2(\T)).
    \end{equation*}
\end{itemize}
Finally,
\begin{equation*}
    \partial_t \left(\frac{\sigma^n}{\mu^n}\right)\in L^2(0,T,H^{-1}(\T)).
\end{equation*}
Moreover, because $c^n$ converges strongly, we have using \eqref{D_A1}
\begin{equation*}
    \frac{\sigma^n}{\mu^n}\tendf{n}{\infty} \frac{\sigma}{\mu} \text{ in }L^2(0,T,L^2(\T)).
\end{equation*}
So, by Aubin-Lions lemma,
\begin{equation*}
    \frac{\sigma^n}{\mu^n} \tend{n}{+\infty} \frac{\sigma}{\mu} \text{ in }L^2(0,T,H^{-1}(\T)).
\end{equation*}
But by \eqref{D_A1} we have moreover
\begin{equation*}
    \sigma^n \tendf{n}{\infty} \sigma \text{ in }L^2(0,T,H^1(\T))
\end{equation*}
hence
\begin{equation}\label{emma1}
    \int_0^T\int_0^1 \frac{(\sigma^n)^2}{\mu^n} = \int_0^T \PS{\frac{\sigma^n}{\mu^n}}{\sigma^n}_{H^{-1}\times H^1}\tend{n}{\infty} \int_0^T \PS{\frac{\sigma}{\mu}}{\sigma} =  \int_0^T\int_0^1 \frac{\sigma^2}{\mu}.
\end{equation}
But, as $c^n$ converges to $c$ strongly,
\begin{equation}\label{emma2}
    \frac{\sigma^n}{\sqrt{\mu^n}} \tendf{n}{\infty} \frac{\sigma}{\sqrt{\mu}} \text{ in }L^2(0,T,L^2(\T)).
\end{equation}
then using \eqref{emma1} and \eqref{emma2}
\begin{equation*}
    \frac{\sigma^n}{\sqrt{\mu^n}}\tend{n}{\infty}\frac{\sigma}{\sqrt{\mu}}\text{ in }L^2(0,T,L^2(\T)).
\end{equation*}
Therefore,  multiplying by $\sqrt{\mu^n}$ and using again the strong convergence of $c^n$ we get
\begin{equation}\label{strongconvsigma1}
    \sigma^n \tend{n}{\infty} \sigma \text{ in }L^2(0,T,L^2(\T)).
\end{equation}

\item[Step 4:] Using \eqref{energ}, and denoting $r^n = \rho^n\theta^n$, we get
\begin{equation}\label{didestou1}
    \partial_tr^n + \partial_x(u^n r^n) = \frac{\sigma^n\partial_x u^n}{\cv^n}
\end{equation}
with
\begin{equation}\label{didestou2}
    \sigma^n = \mu^n\partial_x u^n - R^n r^n.
\end{equation}
Hence, for all $\beta \in {\cal C}^1(\R,\R)$:
\begin{align*}
    \partial_t \beta(r^n) + \partial_x (u^n\beta(r^n)) &= \frac{\sigma^n\partial_x u^n}{\cv^n}\beta'(r^n) + (\beta(r^n)-r^n\beta'(r^n))\partial_x u^n
    \\
    \begin{split}
        &=\frac{(\sigma^n)^2}{\mu^n\cv^n}\beta'(r^n) + \frac{\sigma^n(\gamma^n-1)}{\mu^n}r^n\beta'(r^n) + (\beta(r^n)-r^n\beta'(r^n))\frac{\sigma^n}{\mu^n} \\&+ r^n(\beta(r^n)-r^n\beta'(r^n))\frac{R^n}{\mu^n}.
    \end{split}
\end{align*}
Choosing $\beta(x) = x\ln(x)$ in order to have $\beta(x)-x\beta'(x) = -x$, we get
\begin{equation}\label{rrenomarlize}
\begin{split}
     \partial_t (r^n\ln(r^n)) + \partial_x (u^n r^n\ln(r^n))
        &=\frac{(\sigma^n)^2}{\mu^n\cv^n}(\ln(r^n)+1) + \frac{\sigma^n(\gamma^n-1)}{\mu^n}r^n\ln(r^n)+ \frac{\sigma^n(\gamma^n-1-\cv^n)}{\mu^n\cv^n}r^n \\&-(r^n)^2\frac{R^n}{\mu^n}.
    \end{split}
\end{equation}
Remark that by \eqref{D_UL_rho} and \eqref{D_UL_theta}, for all continuous function $f: ]0,+\infty[\rightarrow \R$, $f(r^n)$ is bounded in $L^\infty(0,T,L^\infty(\T))$, so there exists such $\overline{f(r)}$ such that, up to a subsequence,
\begin{equation*}
    f(r^n)\tend{n}{\infty} \overline{f(r)} \text{ in }L^\infty([0,T]\times\T) \text{ weak}-\star. 
\end{equation*}
Note that we can choose some subsequence which does not depend on $f$.
Passing to the limit in \eqref{rrenomarlize}, and using \eqref{strongconvu}, \eqref{strongconvc}, and \eqref{strongconvsigma1}, we get
\begin{equation}\label{rrenomarlize2}
\begin{split}
     \partial_t (\overline{r\ln(r)}) + \partial_x (u \overline{r\ln(r)})
        &=\frac{\sigma^2}{\mu\cv}(\overline{\ln(r)}+1) + \frac{\sigma(\gamma-1)}{\mu}\overline{r\ln(r)}+ \frac{\sigma(\gamma-1-\cv)}{\mu\cv}\overline{r} \\&-\overline{r^2}\frac{R}{\mu}.
    \end{split}
\end{equation}
But passing to the limit in \eqref{didestou1} and \eqref{didestou2}, using \eqref{strongconvc}, \eqref{strongconvsigma1}, \eqref{strongconvu} and \eqref{weakconvu}, we get
\begin{equation*}
    \partial_t \overline{r} + \partial_x(u \overline{r}) = \frac{\sigma\partial_x u}{\cv^n}
\end{equation*}
\begin{equation*}
    \sigma = \mu\partial_x u - R\overline{r}.
\end{equation*}
So, by the same reasoning,
\begin{equation}\label{rrenormalize3}
    \begin{split}
     \partial_t (\overline{r}\ln(\overline{r})) + \partial_x (u \overline{r}\ln(\overline{r}))
        &=\frac{\sigma^2}{\mu\cv}(\ln(\overline{r})+1) + \frac{\sigma(\gamma-1)}{\mu}\overline{r}\ln(\overline{r})+ \frac{\sigma(\gamma-1-\cv)}{\mu\cv}\overline{r} \\&-\overline{r}^2\frac{R}{\mu}.
    \end{split}
\end{equation}
Substracting \eqref{rrenormalize3} from \eqref{rrenomarlize2}, then integrating in space and in time, using the fact that
\begin{equation}\label{r0}
    \overline{r_0 \ln(r_0)} = \overline{r_0}\ln(\overline{r_0})
\end{equation}
we obtain
\begin{equation}\label{stabilitykey}
    \begin{split}
        \int_0^1 \overline{r\ln(r)}-\overline{r}\ln(\overline{r}) &= \int_0^t\int_0^1 \frac{\sigma^2}{\mu\cv}(\overline{\ln(r)}-\ln(\overline{r})) + \int_0^t\int_0^1 \frac{\sigma R}{\mu\cv}(\overline{r\ln(r)}-\overline{r}\ln(\overline{r}))
        \\&- \int_0^t\int_0^1 \frac{R}{\mu}(\overline{r^2}-\overline{r}^2).
    \end{split}
\end{equation}
The functions $s\mapsto s\ln(s)$, $s\mapsto s^2$ and $s\mapsto -\ln(s)$ being convex, we have
\begin{equation*}
    \overline{r\ln(r)}\geq \overline{r}\ln(\overline{r}),\quad \overline{r^2}\geq \overline{r}^2,\quad -\overline{\ln(r)}\geq -\ln(\overline{r}).
\end{equation*}
Therefore we deduce from \eqref{stabilitykey}
\begin{align*}
    \int_0^1 \vert \overline{r\ln(r)}-\overline{r}\ln(\overline{r}) \vert &\leq \int_0^t\int_0^1 \frac{\sigma R}{\mu} \vert \overline{r\ln(r)} -\overline{r}\ln(\overline{r})\vert \\&\leq \frac{R_{max}}{\mu_{min}} \int_0^t \Vert \sigma(s)\Vert_{L^\infty(\T)} \int_0^1 \vert \overline{r\ln(r)}-\overline{r}\ln(\overline{r})\vert.
\end{align*}
So, by \eqref{r0} and Grönwall's lemma, we get
\begin{equation}\label{argh}
    r^n\ln r^n \tend{n}{\infty}\overline{r}\ln \overline{r} \text{ in }L^\infty(0,T,L^1(\T)).
\end{equation}
But by Taylor-Lagrange, for all $n\in\N$ there exists some $\xi^n \in ]\min(\overline{r},r^n),\max(\overline{r},r^n)[$ such that
\begin{equation*}
    r^n\ln r^n - \overline{r}\ln \overline{r} - (1 + \ln \overline{r})(r^n-\overline{r}) = \frac{1}{2\xi^n} (r^n-\overline{r})^2 \geq \frac{1}{2\overline{\rho}\overline{\theta}}(r^n-\overline{r})^2.
\end{equation*}
So, integrating in space and in time, then passing to the limit and using \eqref{argh}, we get
\begin{equation*}
    r^n \tend{n}{\infty}\overline{r} \text{ in }L^2(0,T,L^2(\T)).
\end{equation*}
hence 
\begin{equation}\label{peine}
    p^n \tend{n}{\infty} p \text{ in }L^2(0,T,L^2(\T)).
\end{equation}
\item[Step 5:] Using the strong convergence of $\sigma^n$ and $c^n$, we deduce from \eqref{peine}
\begin{equation}\label{partialxun}
    \partial_x u^n \tend{n}{\infty} \partial_x u \text{ in }L^2(0,T,L^2(\T)).
\end{equation}
Using \eqref{momts} we have
\begin{equation}\label{rhonlogrhon}
    \rho^n D_t (\ln\rho^n) = -\rho^n\partial_x u^n 
\end{equation}
so integrating in space then in time,
\begin{equation*}
    \int_0^1 \rho^n\ln\rho^n = \int_0^1\rho_0^n\ln\rho_0^n -\int_0^T\int_0^1 \rho^n\partial_x u^n.
\end{equation*}
Passing to the limit, using \eqref{partialxun}, \eqref{weakconvcrt} and the strong convergence of $\rho_0^n$,
\begin{equation*}
    \int_0^1 \overline{\rho\ln\rho} = \int_0^1 \rho_0\ln\rho_0 - \int_0^T\int_0^1 \rho \partial_x u.
\end{equation*}
We can now pass to the limit in equality
\begin{equation*}
    \int_0^1 (\rho^n\ln(\rho^n) -\rho\ln(\rho)) = \int_0^1 (\rho^n_0\ln(\rho^n_0) - \rho_0\ln(\rho_0)) - \int_0^T\int_0^1 (\rho^n \partial_x u^n - \rho \partial_x u).
\end{equation*}
But by \eqref{unstrongu} and \eqref{weakconvcrt}, we can pass to the limit in \eqref{momts} to obtain
\begin{equation*}
    \partial_t \rho + \partial_x (\rho u) =0.
\end{equation*}
So \eqref{rhonlogrhon} is still true, and the same reasoning gives
\begin{equation*}
    \int_0^1 \rho\ln\rho = \int_0^1 \rho_0\ln\rho_0 - \int_0^T\int_0^1 \rho\partial_x u. 
\end{equation*}
Therefore,
\begin{equation*}
    \int_0^1\rho^n\ln\rho^n \tend{n}{\infty} \int_0^1 \rho\ln\rho
\end{equation*}
and using Taylor-Lagrange again we obtain
\begin{equation*}
\rho^n \tend{n}{\infty} \rho\text{ in }L^2(0,T,L^2(\T)).
\end{equation*}
\end{itemize}
\end{proof}

\subsection{Uniqueness}

\begin{theorem}
Let $(c_0,\rho_0,\theta_0,u_0)\in L^\infty(\T)^3\times H^1(\T)$ such that there exists some $\underline{\rho_0},\overline{\rho_0},\underline{\theta_0},\overline{\theta_0}>0$ such that

\begin{equation*}
    \forall x\in\T,\quad \underline{\rho_0}\leq \rho_0(x)\leq\overline{\rho_0},\quad \underline{\theta_0}\leq \theta(x)\leq \overline{\theta_0}, \quad 0\leq c_0(x)\leq 1.
\end{equation*}
Then there exists at most one solution "à la Hoff" of \eqref{cont}--\eqref{energ} and \eqref{crho} with initial conditions $(c_0,\rho_0,\theta_0,u_0)$.
\end{theorem}

\begin{proof}
We choose to switch to Lagrangian coordinates in this proof. For a map $f:[0,T]\times\T \rightarrow \R$, we denote
\begin{equation*}
    \forall (t,x)\in [0,T]\times\T,\quad \widetilde{f}(t,x)= f(t,X(t,0,x)).
\end{equation*}
Using the formula
\begin{equation*}
    \partial_x X(t,0,x) = \frac{\rho_0(x)}{\rho(t,X(t,0,x))}
\end{equation*}
we get
\begin{equation*}
    \widetilde{D_t(f)} = \partial_t \widetilde{f}
\text{ and } 
    \rho_0 \widetilde{\partial_x f} = \widetilde{\rho}\partial_x \widetilde{f}.
\end{equation*}
This allows us to move from the Eulerian formulation of \eqref{NS} to the Lagrangian one: 
\begin{subequations}\label{NSlagrang}
\begin{empheq}[left=\empheqlbrace]{alignat=1}
\partial_t \tilde{c}&=0
\label{clagrang}\\
\rho_0\partial_t \left(\frac{1}{\tilde{\rho}}\right) &=\partial_x \tilde{u}
\label{contlagrang}\\
\rho_0\partial_t \Tilde{u} &= \partial_x \Tilde{\sigma}
\label{momtslagrang}\\
\rho_0\Tilde{\cv}\partial_t \Tilde{\theta} &= \Tilde{\sigma}\partial_x \Tilde{u}
\label{energlagrang}\\
\Tilde{\sigma} &= \Tilde{\mu}\frac{\Tilde{\rho}}{\rho_0}\partial_x \Tilde{u} - \Tilde{R}\Tilde{\rho}\Tilde{\theta}.
\end{empheq}
\end{subequations}
Thanks to Duhamel's formula, we can express $\Tilde{\rho}$ and $\Tilde{\theta}$ as functions of $\partial_x \Tilde{u}$. Using \eqref{contlagrang} and \eqref{energlagrang}, we obtain
\begin{equation}\label{duhamrhot}
    \tilde{\rho} = \frac{\rho_0}{1 + \int_0^t \partial_x \Tilde{u}}
\end{equation}
\begin{equation}\label{duhamelthetat}
    \Tilde{\theta} = \theta_0 \exp\left(\int_0^t -\frac{\Tilde{R}}{\Tilde{\cv}\rho_0}\Tilde{\rho}(\partial_x \Tilde{u})ds\right) + \int_0^t \frac{\Tilde{\mu}}{\Tilde{\cv}}\frac{\Tilde{\rho}}{\rho_0}(\partial_x \Tilde{u})^2 \exp\left(\int_s^t -\frac{\Tilde{R}}{\Tilde{\cv}\rho_0}\Tilde{\rho}(\partial_x \Tilde{u}) d\tau \right)ds.
\end{equation}

Let $(c_1,\rho_1,\theta_1,u_1)$ and $(c_2,\rho_2,\theta_2,u_2)$ be two solutions "à la Hoff" of \eqref{NS} with same initial conditions. Because $\partial_x u_1,\partial_x u_2\in L^1([0,T],L^\infty(\T))$, they are solutions of \eqref{NSlagrang} too.

By \eqref{clagrang}, $\Tilde{c}_1 = \Tilde{c}_2$. We will denote $\delta \Tilde{f} := \Tilde{f}_2 -\Tilde{f}_1$. All the following constants are noted $C$, even if they are different from each other. We get from \eqref{duhamrhot}
\begin{equation*}
    \delta \tilde{\rho} = -\frac{\Tilde{\rho}_1\Tilde{\rho}_2}{\rho_0}\left(\frac{\rho_0}{\Tilde{\rho}_2}-\frac{\rho_0}{\Tilde{\rho}_1}\right) = -\frac{\Tilde{\rho}_1\Tilde{\rho}_2}{\rho_0}\int_0^t \partial_x \delta \Tilde{u}
\end{equation*}
therefore
\begin{equation}\label{deltarho}
    \vert \delta \Tilde{\rho} \vert \leq \frac{\overline{\rho_1}\cdot\overline{\rho_2}}{\underline{\rho_0}}\sqrt{T} \left(\int_0^t  (\partial_x \delta \Tilde{u})^2\right)^\frac{1}{2}.
\end{equation}

Now let's bound the temperature difference. Writing $\Tilde{\theta}=:\theta_0 A_1 + A_2$ in \eqref{duhamelthetat}, we get, using the mean value theorem
\begin{align*}
\vert\delta A_1\vert &\leq \int_0^t \frac{\gamma_{max}-1}{\underline{\rho_0}}\vert\delta(\tilde{\rho}(\partial_x\tilde{u}))\vert \exp\left(\int_0^T \frac{\gamma_{max}-1}{\underline{\rho_0}}\overline{\rho}\Vert \partial_x \tilde{u}\Vert_\infty\right). 
\end{align*}
Using the formula
\begin{equation*}
    \delta(fg) = (\delta f)g_2 + f_1(\delta g)
\end{equation*}
we get

\begin{equation*}
    \begin{split}
    \vert \delta A_1\vert \leq  \left(\frac{\gamma_{max}-1}{\underline{\rho_0}} \left(\int_0^T \Vert \partial_x \tilde{u}\Vert_\infty\right)\vert \delta \tilde{\rho}\vert \right.+\frac{\gamma_{max}-1}{\underline{\rho_0}}\overline{\rho}&\left.\sqrt{T}\left(\int_0^t(\partial_x \delta \tilde{u})^2\right)^{1/2}\right)
    \\&\times\exp{\left(\int_0^T \frac{\gamma_{max}-1}{\underline{\rho_0}}\overline{\rho}\Vert \partial_x \tilde{u}\Vert_\infty\right)}
    \end{split} 
\end{equation*}
thus
\begin{equation}\label{deltaA1}
    \vert\delta A_1\vert \leq C\left(\int_0^t (\partial_x \delta \tilde{u})^2\right)^{1/2}.
\end{equation}

Similarly, we get
\begin{equation*}
\begin{split}
    \delta A_2 = \int_0^t \frac{\tilde{\mu}}{\Tilde{\cv}\tilde{\rho_0}}(\delta \tilde{\rho}) (\partial_x \tilde{u})^2_2 (A_1(t)_2-A_1(s)_2) &+ \int_0^t \frac{\tilde{\mu}}{\tilde{\cv}\tilde{\rho_0}}\tilde{\rho_1} (\partial_x \delta \tilde{u})(\partial_x \tilde{u}_1 + \partial_ x \tilde{u}_2)(A_1(t)_2-A_1(s)_2) \\&+ \int_0^t \frac{\tilde{\mu}}{\tilde{\cv}\tilde{\rho_0}}\tilde{\rho_1}(\partial_x \tilde{u})_1^2 (\delta A_1(t) - \delta A_2(t))
\end{split}
\end{equation*}
so using \eqref{deltaA1}
\begin{align}\nonumber
\begin{split}
    \vert \delta A_2\vert &\leq \vert \delta \tilde{\rho}\vert \int_0^T \frac{\mu_{max}}{\cv_{min}\underline{\rho_0}}(\partial_x \tilde{u})_2^2(\vert A_1(t)_2\vert + \vert A_1(s)_2\vert) \\&+ \frac{2\mu_{max}}{\cv_{min}\underline{\rho_0}}\overline{\rho}\left(\int_0^t (\partial_x \delta \tilde{u})^2\right)^{1/2}\left(\int_0^T((\partial_x \tilde{u}_1)^2+(\partial_x \tilde{u}_2)^2)(A_1(t)_2^2 + A_1(s)_2^2)\right)^{1/2}
    \\&+ 2C \left(\int_0^t (\partial_x \delta \tilde{u})^2\right)^{1/2} \frac{\mu_{max}}{\cv_{min}\underline{\rho_0}}\overline{\rho}\int_0^T (\partial_x \tilde{u}_1)^2
\end{split}
\\&\leq C'\left(\int_0^t (\partial_x \delta \tilde{u})^2\right)^{1/2}.\label{deltaA2}
\end{align}
Therefore, \eqref{deltaA1} and \eqref{deltaA2} give
\begin{equation}\label{deltatheta}
    \vert \delta \Tilde{\theta} \vert \leq (\overline{\theta_0}C + C')\left(\int_0^t (\partial_x \delta \Tilde{u})^2\right)^\frac{1}{2}.
\end{equation}
From \eqref{momtslagrang} we deduce that
\begin{equation*}
    \rho_0 \partial_t \delta \Tilde{u} = \partial_x \delta \Tilde{\sigma} = \partial_x(\Tilde{\mu} \frac{\Tilde{\rho}_1}{\rho_0}\partial_x \delta\Tilde{u}) +  \partial_x \left(\Tilde{\mu} \frac{\delta \Tilde{\rho}}{\rho_0}\partial_x \Tilde{u}_2 -\Tilde{R}\Tilde{\theta}_1\delta\Tilde{\rho} - \Tilde{R}\Tilde{\rho}_2\delta \Tilde{\theta} \right).
\end{equation*}
Multiplying by $\delta \Tilde{u}$ and integrating by parts, we get
\begin{equation*}
    \frac{d}{dt}\int_0^1 \rho_0 (\delta \Tilde{u})^2 + \int_0^1 \Tilde{\mu}\frac{\Tilde{\rho}_1}{\rho_0}(\partial_x \delta \Tilde{u})^2 = -\int_0^1 \left(\Tilde{\mu} \frac{\delta \Tilde{\rho}}{\rho_0}\partial_x \Tilde{u}_2 -\Tilde{R}\Tilde{\theta}_1\delta\Tilde{\rho} - \Tilde{R}\Tilde{\rho}_2\delta \Tilde{\theta} \right) \partial_x \delta \Tilde{u}
\end{equation*}
then by Young's inequality,
\begin{equation*}
\begin{split}
     \frac{d}{dt}\int_0^1 \rho_0 (\delta \Tilde{u})^2 + \int_0^1 \Tilde{\mu}\frac{\Tilde{\rho}_1}{\rho_0}(\partial_x \delta \Tilde{u})^2\leq \frac{1}{2} \int_0^1 \tilde{\mu}\frac{\tilde{\rho_1}}{\rho_0}(\partial_x\delta \Tilde{u})^2 &+ \frac{2\overline{\rho_0}}{\underline{\rho_1}\mu_{min}} \left(\frac{\mu_{max}^2}{\underline{\rho_0}^2}\Vert \partial_x \tilde{u}_2\Vert_\infty^2 + R_{max}^2\overline{\theta_1}^2\right)\int_0^1 (\delta \tilde{\rho})^2 
     \\&+ \frac{\overline{\rho_0}}{\underline{\rho_1}\mu_{min}}R_{max}^2\overline{\rho}_2^2\int_0^1 (\delta \tilde{\theta})^2.
\end{split}
\end{equation*}
Finally, using \eqref{deltarho} and \eqref{deltatheta}, we get
\begin{equation}\label{uniquenessgronwfin}
    \frac{d}{dt}\int_0^1 \rho_0 (\delta \Tilde{u})^2 + \frac{1}{2}\int_0^1 \Tilde{\mu}\frac{\Tilde{\rho}_1}{\rho_0}(\partial_x \delta \Tilde{u})^2 \leq \frac{C''}{2}(1+\Vert \partial_x \tilde{u}_2\Vert_\infty^2)\int_0^t\int_0^1 \tilde{\mu}\frac{\tilde{\rho_1}}{\rho_0}(\partial_x \delta \tilde u)^2. 
\end{equation}
Denoting
\begin{equation*}
    y(t) := \int_0^1 \rho_0 (\delta \Tilde{u})^2 + \frac{1}{2}\int_0^t\int_0^1 \Tilde{\mu}\frac{\Tilde{\rho}_1}{\rho_0}(\partial_x \delta \Tilde{u})^2 
\end{equation*}
From \eqref{uniquenessgronwfin} we deduce
\begin{equation*}
    y'(t)\leq \frac{C''}{2}(1+\Vert \partial_x u_2\Vert_\infty)y(t).
\end{equation*}

Since $\Vert\partial_x\Tilde{u}_2\Vert_{L^\infty}$ is bounded in $L^2([0,T])$, we deduce by Gr\"onwall
\begin{equation*}
    \int_0^1 \rho_0 (\delta \Tilde{u})^2 + \frac{1}{2}\int_0^T\int_0^1 \Tilde{\mu}\frac{\Tilde{\rho}_1}{\rho_0} (\partial_x\delta \Tilde{u})^2 \leq 0.
\end{equation*}
Therefore $\Tilde{u}_1=\Tilde{u}_2$. Hence by \eqref{deltarho} and \eqref{deltatheta}, $\Tilde{\rho}_1 = \Tilde{\rho}_2$ and $\Tilde{\theta}_1 = \Tilde{\theta}_2$. Moreover, using the equality
\begin{equation*}
    \forall (t,x)\in[0,T]\times\T,\quad X_1(t,0,x) = x + \int_0^t u_1(s,X_1(s,0,x)) ds  = x + \int_0^t u_2(s,X_2(s,0,x)) = X_2(t,0,x)
\end{equation*}
we get, inverting in space at fixed $t>0$
\begin{equation*}
    \forall (t,x)\in [0,T]\times\T,\quad X_1(0,t,x) = X_2(0,t,x).
\end{equation*}
we have, for $f=c,\rho,\theta,u$:
\begin{equation*}
    \forall (t,x)\in [0,T]\times\T,\quad f_1(t,x) = \tilde{f_1}(t,X_1(0,t,x)) = \tilde{f_2}(t,X_2(0,t,x)) = f_2(t,x). 
\end{equation*}
So $(c_1,\rho_1,\theta_1,u_1) = (c_2,\rho_2,\theta_2,u_2)$.
\end{proof}

%%%%%%%% Section 5
\section{Averaging procedure}\label{averaging}

This section is dedicated to the proof of the convergence of the unknowns (Theorem \ref{Theorem_main1}), obtaining the kinetic equation (Theorem \ref{Theorem_main2}), and then the uniqueness of the solution for the kinetic equation and the construction of a particular solution, to establish Theorem \ref{Theorem_main3}.

 \subsection{Convergence of the unknowns using the uniform bounds}

 Fix $T>0$. Consider a sequence of uniformly bounded initial conditions $(c_{0}^{\varepsilon},\rho_{0}^{\varepsilon},\theta
_{0}^{\varepsilon},u_{0}^{\varepsilon})_{\varepsilon}$, $(c_{0},\rho
_{0},\theta_{0},u_{0})\in L^{\infty}(\T)^{3}\times H^{1}(\T)$ verifying \eqref{initialcond} with uniform bounds and
$\left(  \text{\ref{ini_2}}\right)  $. For all $\varepsilon>0,$ consider
$(c^{\varepsilon},\rho^{\varepsilon},\theta^{\varepsilon},u^{\varepsilon})$
the solution "\`{a}~la Hoff" of \eqref{NS} with initial data $(c_{0}%
^{\varepsilon},\rho_{0}^{\varepsilon},\theta_{0}^{\varepsilon},u_{0}%
^{\varepsilon})$. Then, the existence of $(c,\rho,\theta)\in L^{\infty
}([0,T]\times\T)^{3},$ $\sigma\in L^{\infty}(0,T,L^{2}(\T))$ and $u\in
H^{1}([0,T]\times\T)$ along with relations \eqref{faibleconv} and \eqref{ineglim} follow immediately by using the uniform bounds granted by \eqref{D_UL_rho}, \eqref{D_UL_theta} which are verified by solutions "\`{a}~la Hoff". From \eqref{D_A1inf} we deduce using Rellich's theorem that \eqref{forteconv} holds.
%all we have to do is show tha $(u^\varepsilon)$ is bounded in $H^1([0,T]\times \T)$ uniformly in $\varepsilon$. %{\color{red} A SUPPRIMER CE SERA MIS DANS LA PARTIE ESTIMEES A PRIORI Using \eqref{D_A1},  $\partial_x u^\varepsilon$ is uniformly bounded in $L^\infty(0,T,L^2(\T))$. Moreover, by \eqref{D_A1inf}, $u^\varepsilon$ is uniformly bounded in $L^2(0,T,L^\infty(\T))$. So $u^\varepsilon\partial_x u^\varepsilon$ is uniformly bounded in $L^2(0,T,L^2(\T))$. But $\dot{u^\varepsilon} = \partial_x\sigma^\varepsilon/\rho^\varepsilon$ is uniformly bounded in $L^2(0,T,L^2(\T))$ by \eqref{D_A1}. So $\partial_t u^\varepsilon$ is uniformly bounded in $L^2(0,T,L^2(\T))$, hence the result.}  
The strong convergence of $\sigma^\varepsilon$ is a consequence of \eqref{secondhoff}. Using \eqref{D_A1}, $\sigma^\varepsilon$ is bounded in $L^\infty(0,T,L^2(\T))$. In particular, there exists some $\sigma\in L^\infty(0,T,L^2(\T))$ such that 
\begin{equation*}
    \sigma^\varepsilon \tend{\varepsilon}{0} \sigma \text{ in }L^\infty(0,T,L^2(\T))\text{ weakly}-\star.
\end{equation*}
In the following reasoning, the writing "$\in X$" is to be understood as "uniformly bounded in $\varepsilon$ in $X$". 
Let's consider the equality
\begin{equation}\label{siousiou}
    \partial_t(\kappa \sigma^\varepsilon) = \kappa\partial_t\sigma^\varepsilon + \kappa'\sigma^\varepsilon = \kappa \dot{\sigma^\varepsilon} -u^\varepsilon\kappa\partial_x\sigma^\varepsilon+ \kappa'\sigma^\varepsilon.
\end{equation}

\begin{itemize}
    \item As $u^\varepsilon\in L^\infty(0,T,L^\infty(\T))$ and $\partial_x\sigma^\varepsilon\in L^2(0,T,L^2(\T))$, we get $u^\varepsilon\kappa\partial_x\sigma^\varepsilon\in L^2(0,T,L^2(\T))$.
    \item Similarly, we have $\sigma^\varepsilon\in L^2(0,T,L^2(\T))$, so $\kappa'\sigma^\varepsilon\in L^2(0,T,L^2(\T))$.
    \item Finally, using \eqref{secondhoff}, we get $\sqrt{\kappa} \dot{\sigma^\varepsilon}\in L^2(0,T,L^2(\T))$. So $\kappa \dot{\sigma^\varepsilon}\in L^2(0,T,L^2(\T))$.
\end{itemize}
Therefore, by \eqref{siousiou}, $\partial_t(\kappa\sigma^\varepsilon)\in L^2(0,T,L^2(\T))$. Because $\kappa\sigma^\varepsilon\in L^2(0,T,L^2(\T))$, using the Aubin-Lions lemma we get

\begin{equation}\label{convpoids}
    \kappa \sigma^\varepsilon \tend{\varepsilon}{0} \kappa \sigma \text{ in }L^2(0,T,L^2(\T)).
\end{equation}

Let $\eta \in ]0,1[$. Then

\begin{align}
     \int_0^T\int_0^1 \vert\sigma^\varepsilon - \sigma\vert^2  &= \int_0^\eta \int_0^1 \vert \sigma^\varepsilon - \sigma\vert^2 + \int_\eta^T\int_0^1 \vert \sigma^\varepsilon-\sigma\vert^2\nonumber \\&\leq 2\eta \left(\sup_{[0,T]}\int_0^1 (\sigma^\varepsilon)^2 + \sup_{[0,T]}\int_0^1 \sigma^2\right) + \frac{1}{\eta^2}\int_\eta^T \kappa^2 \int_0^1 \vert \sigma^\varepsilon - \sigma \vert^2\nonumber
     \\&\leq 4C_1\eta + \frac{1}{\eta^2}\int_0^T\int_0^1 \vert \kappa \sigma^\varepsilon - \kappa\sigma\vert^2.\label{convfine}
\end{align}
Let us now use \eqref{convpoids}: let $\varepsilon_0>0$ such that

\begin{equation*}
\forall \varepsilon\in]0,\varepsilon_0[,\quad \int_0^T\int_0^1 \vert \kappa\sigma^\varepsilon - \kappa\sigma\vert^2 \leq \eta^3.
\end{equation*}
Then from \eqref{convfine}

\begin{equation*}
\forall \varepsilon\in ]0,\varepsilon_0[,\quad \int_0^T\int_0^1 \vert \sigma^\varepsilon-\sigma\vert^2 \leq (4C_1+1)\eta
\end{equation*}
hence the result. This concludes the proof of \eqref{strongconvsigmaeps} and with it the first part of Theorem \ref{Main2.1}.

As explained in Remark \ref{Main2.1}, the form \eqref{system_general} is obtained using the definition of the Young measures. 

\subsection{Kinetic equation}

In this section we show that the limiting behaviour of the sequence of solutions can be described by the equation \eqref{Young}. We begin with the following
%\subsection{Young measure and justification of the formal derivation}
%\subsubsection{Young measures introduction and first properties}
%%%% aici era partea doi din teorema din MR.

\begin{lemma} Let $\nu^\varepsilon$ be defined in Theorem \ref{Theorem_main2}. Then 
$\nu^\varepsilon\tend{\varepsilon}{0}\nu$ in $C(0,T,\cM(\T\times K))$.
\end{lemma}

\begin{proof}
All we need is to prove that for all $(\varphi,\beta)\in C(\T)\times C(K)$, 
\begin{equation}\label{strongconvnu}
     \ps{\nu^\varepsilon}{\varphi\otimes\beta} \tend{\varepsilon}{0}\ps{\nu}{\varphi\otimes\beta} \text{ in }C([0,T]).
\end{equation}
where $(\nu^\varepsilon,\nu): [0,T]\rightarrow \cM(\T\times\K)^2, t\mapsto (\nu_t^\varepsilon,\nu_t)$.
Let $\varphi\in C^1(\T),\beta\in C^1(K)$. Then abbreviating $\beta(c^\varepsilon,\rho^\varepsilon,\theta^\varepsilon)$ by $\beta^\varepsilon$, we get
\begin{align}
D^\varepsilon_t\beta^\varepsilon &= D_tc^\varepsilon \partial_c\beta^\varepsilon + D_t\rho^\varepsilon\partial_\rho\beta^\varepsilon + D_t\theta^\varepsilon \partial_\theta\beta^\varepsilon\nonumber
    \\&=-\rho^\varepsilon(\partial_x u^\varepsilon)\partial_\rho\beta^\varepsilon + \frac{\sigma^\varepsilon\partial_x u^\varepsilon}{\rho^\varepsilon\cv^\varepsilon}\partial_\theta\beta^\varepsilon.\label{startkineq}
\end{align}
In particular, $D_t^\varepsilon\beta^\varepsilon$ is bounded in $L^\infty(0,T,L^1(\T))$, uniformly in $\varepsilon$. But
\begin{align*}
    \frac{d}{dt}\ps{\nu^\varepsilon}{\varphi\otimes\beta} &= \frac{d}{dt}\int_0^1 \beta^\varepsilon\varphi(x)dx = \int_0^1 \frac{d}{dt}\beta^\varepsilon\varphi(x)dx
    \\&= \int_0^1 (D_t^\varepsilon\beta^\varepsilon)\varphi dx - \int_0^1 u^\varepsilon (\partial_x \beta^\varepsilon)\varphi
    \\&=\int_0^1 (D_t^\varepsilon\beta^\varepsilon)\varphi dx + \int_0^1 \beta^\varepsilon (\partial_x u^\varepsilon)\varphi dx + \int_0^1 \beta^\varepsilon u^\varepsilon\varphi'dx.
\end{align*}
Thus $\frac{d}{dt}\ps{\nu^\varepsilon}{\varphi\otimes \beta}$ is bounded in $L^\infty([0,T])$ uniformly in $\varepsilon$. By Ascoli, we obtain \eqref{strongconvnu}. As $C^1(\T)\times C^1(K)$ is dense in $(C(\T)\times C(K),\Vert\cdot\Vert_\infty)$, we get the result using a simple density argument.
\end{proof}

Concerning the Young measures associated with the initial data, by construction, we have
\begin{equation*}
    \nu_{0,x}^\varepsilon = c_0^\varepsilon(x)\delta_{(1,\rho_{0,+}(x),\theta_{0,+}(x))} + (1-c^\varepsilon_0(x))\delta_{(0,\rho_{0,-}(x),\theta_{0,-}(x))}\quad a.e.\quad  x\in \T,
\end{equation*}
so taking the limit as $\varepsilon\rightarrow 0$,
\begin{equation*}
    \nu_{0,x} = \alpha_{0,+}(x)\delta_{(1,\rho_{0,+}(x),\theta_{0,+}(x))} + \alpha_{0,-}(x)\delta_{(0,\rho_{0,-}(x),\theta_{0,-}(x))}\quad a.e.\quad x\in \T
\end{equation*}
where $\alpha_{0,+}, \alpha_{0,-}$ are the weak limit of $c_0^\varepsilon$, $(1-c_0^\varepsilon)$. Next, let us show how to obtain the equation \eqref{KinEq}. In order to do this recall that for almost every $(t,x)\in [0,T]\times \T$, the measures $\nu_{t,x} \in {\cal P}(K)$ are characterized by
\begin{equation*}
    \forall \beta\in C(K),\quad \ps{\nu_{t,x}}{\beta} = \overline{\beta(c,\rho,\theta)}, 
\end{equation*}
see Remark \ref{Main2.1}. Let $\beta\in C^1(K)$. Imitating computations we did in Section \ref{formderiv}, starting with \eqref{startkineq} and writing $D_t\beta^\varepsilon = \partial_t \beta^\varepsilon + \partial_x(u^\varepsilon \beta^\varepsilon) - (\partial_x u^\varepsilon) \beta^\varepsilon$, we get
\begin{equation*}
    \partial_t \beta^\varepsilon + \partial_x (u^\varepsilon\beta^\varepsilon) - \frac{\sigma^\varepsilon + p^\varepsilon}{\mu^\varepsilon}\beta^\varepsilon + \rho^\varepsilon\frac{\sigma^\varepsilon + p^\varepsilon}{\mu^\varepsilon}\partial_\rho \beta^\varepsilon - \frac{\sigma^\varepsilon(\sigma^\varepsilon + p^\varepsilon)}{\cv^\varepsilon \mu^\varepsilon\rho^\varepsilon} \partial_\theta \beta^\varepsilon = 0
\end{equation*}
i.e.
\begin{equation}\label{niania}
\begin{split}
    \partial_t \beta^\varepsilon + \partial_x (u^\varepsilon\beta^\varepsilon) - \sigma^\varepsilon \left(\frac{\beta}{\mu(c)}\right)^\varepsilon  &- \left(\frac{p(c,\rho,\theta)\beta}{\mu(c)}\right)^\varepsilon + \sigma^\varepsilon \left(\frac{\rho\partial_\rho\beta}{\mu(c)}\right)^\varepsilon  + \left(\frac{\rho p(c,\rho,\theta)\partial_\rho\beta}{\mu(c)}\right)^\varepsilon \\& - (\sigma^\varepsilon)^2 \left(\frac{\partial_\theta\beta}{\cv(c)\mu(c)\rho}\right)^\varepsilon - \sigma^\varepsilon \left(\frac{p(c,\rho,\theta)\partial_\theta \beta}{\cv(c) \mu(c)\rho}\right)^\varepsilon=0.
\end{split}
\end{equation}
Using \eqref{forteconv} and \eqref{strongconvsigmaeps} we can pass to the limit in \eqref{niania} and obtain
\begin{equation*}
\begin{split}
    \partial_t \overline{\beta} + \partial_x (u\overline{\beta}) - \sigma \overline{\left(\frac{\beta}{\mu(c)}\right)} &- \overline{\left(\frac{p(c,\rho,\theta)\beta}{\mu(c)}\right)} + \sigma \overline{\left(\frac{\rho\partial_\rho\beta}{\mu(c)}\right)} \\& +\overline{\left(\frac{\rho p(c,\rho,\theta)\partial_\rho\beta}{\mu(c)}\right)} - \sigma^2 \overline{\left(\frac{\partial_\theta\beta}{\cv(c)\mu(c)\rho}\right)}  - \sigma \overline{\left(\frac{p(c,\rho,\theta)\partial_\theta \beta}{\cv(c) \mu(c)\rho}\right)} = 0.
\end{split}
\end{equation*}
Therefore, $\nu_{t,x}$ verify the kinetic equation \eqref{KinEq}.
%%voir après si l'on garde ça
%\begin{equation}\label{KinEq}
%    \partial_t \nu_{t,x} + %\partial_x(u \nu_{t,x}) - %\frac{\sigma + p}{\mu}\nu_{t,x} - %\partial_\rho %\left(\frac{\rho(\sigma + %p)}{\mu}\nu_{t,x}\right) + %\partial_\theta %\left(\frac{\sigma(\sigma + %p)}{\cv\mu\rho}\nu_{t,x}\right) = 0. 
%\end{equation}
This concludes the proof of Theorem \ref{Theorem_main2}.

\subsection{Uniqueness of solutions for the kinetic equation}
As explained in Section 3, the proof of Theorem \ref{Theorem_main3} is divided into 2 parts:
\begin{itemize}
\item In a first step, we show that given $u=u\left(  t,x\right)$, $\sigma=\sigma(t,x)$ having the regularity announced in Definition \ref{defhoff}, the equation $\left(  \text{\ref{KinEq}%
}\right)  $ has at most one solution;
\item in a second step, we construct a solution that verifies  $\left(
\text{\ref{KinEq}}\right)  $ and using uniqueness we conclude. 
\end{itemize}
We begin by proving uniqueness. We will give two different demonstrations.

\paragraph{First proof.} We prove the following somehow general result. The proof follows the lines of Lemma 12 in \cite{BrHi}.
\begin{lemma}\label{lemmaunicity}
Let $d>1$. Let $u\in L^1(0,T,C_c([0,1]\times\R^{d-1},\R^d))$, where $u(t,x) = (u_1(t,x),u_2(t,x))\in \R\times \R^{d-1}$, such that
$$\nabla u = \left(\begin{array}{cc}
     \partial_1 u_1  & 0\\
     \partial_1 u_2  & \nabla_2 u_2
\end{array}\right)$$
with 
\begin{equation*}
    \partial_1 u_1 \in L^1(0,T,L^\infty([0,1]\times\R^{d-1})),\quad \partial_1 u_2\in L^1(0,T,L^1([0,1]\times\R^{d-1})),\quad \nabla_2 u_2\in L^1(0,T,L^\infty([0,1]\times\R^{d-1})).  
\end{equation*}

Let now $g\in L^1(0,T,C_c([0,1]\times \R^{d-1},\R))$ such that $\partial_1 g\in L^1(0,T,L^1([0,1]\times \R^{d-1}))$ and $\nabla_2 g\in L^1(0,T,L^\infty([0,1]\times\R^{d-1}))$. 
Let $\nu \in C(0,T,\cM_c(\R^d))$ such that
\begin{equation}\label{KinEqTh}
    \partial_t \nu + \Div (u \nu) + g \nu = 0.
\end{equation}
Then $\nu$ is uniquely determined by $\nu_0$.
\end{lemma}

\begin{proof}
We extend $u$ and $g$ to $[0,T]\times \R^d$ by periodicity. If $\varphi\in W^{1,1}(0,T,C_c^1(\R^d))$, then (using $W^{1,1}([0,T])\hookrightarrow C([0,T])$), $\partial_t \varphi + u\cdot\nabla\varphi -g\varphi \in L^1(0,T,C(\R^d))$. Therefore, by \eqref{KinEqTh},
    \begin{equation}\label{dualeq}
        \forall t\in[0,T],\quad\forall \varphi\in W^{1,1}(0,T,C^1(\R^d)),\quad \ps{\nu}{\partial_t\varphi + u\cdot\nabla \varphi - g\varphi} = - \ps{\nu_t}{\varphi(t,\cdot)}.
    \end{equation}
    Let's fix $t\in [0,T]$ and $\psi\in C^\infty_c(]0,T[\times\R^d)$. Our aim is to show that $\ps{\nu_t}{\psi}=0$ to conclude the proof. We consider $u^\eta$ and $g^\eta$ some regularizations of $u$ and $g$ in the first space variable only. Denoting $\varphi^\eta$ a solution of
    \begin{empheq}[left=\empheqlbrace]{alignat=1}
    \partial_t\varphi^\eta + u^\eta\cdot\nabla\varphi^\eta - g^\eta\varphi^\eta &= 0 \text{ on }]0,t[\times \R^d\label{regul}\\
    \varphi^\eta(t,\cdot) &= \psi \text{ on }\R^d,\label{regulcondi}
    \end{empheq}
    it can be shown that $\varphi^\eta\in W^{1,1}(0,t,C^1(\R^d))$. Then, \eqref{dualeq}, \eqref{regul} and \eqref{regulcondi} give
    \begin{equation*}
        \ps{\nu}{(u-u^\eta)\cdot\nabla\varphi^\eta + (g^\eta-g)\varphi^\eta} = - \ps{\nu_t}{\psi}.
    \end{equation*}
It remains to show
\begin{equation}\label{dette1}
    (u^\eta - u)\cdot\nabla\varphi^\eta \tend{\eta}{0} 0 \text{ in }L^1(0,t,C_c(\R^d))
\end{equation}
\begin{equation}\label{dette2} 
(g^\eta-g)\varphi^\eta\tend{\eta}{0} 0\text{ in }L^1(0,t,C_c(\R^d)).
\end{equation}
By the Grönwall lemma, as $g^\eta$ is uniformly bounded in $L^1(0,t,L^\infty(\R^d))$, we obtain from \eqref{regul} the existence of $C>0$ independant of $\eta$ such that
\begin{equation}\label{unifphieps}
    \forall s\in ]0,t[,\quad \Vert \varphi^\eta(s,\cdot)\Vert_{L^\infty} \leq C.
\end{equation}
Then, as $g^\eta\tend{\eta}{0}g$ in $L^1(0,t,C_c([0,1]\times\R^{d-1}))$, we get \eqref{dette2}.
Applying $\nabla_2$ to \eqref{regul}, using the fact that
\begin{equation*}
    ({}^t \nabla u\nabla\varphi)_2 = {}^t \nabla_2 u_2 \nabla_2\varphi,
\end{equation*}
we get
\begin{equation*}
    \partial_t \nabla_2 \varphi^\eta + u^\eta\cdot \nabla(\nabla_2\varphi^\eta) - (g^\eta I_d - {}^t \nabla_2 u^\eta_2)\nabla_2\varphi^\eta = \varphi^\eta\nabla_2 g^\eta.
\end{equation*}
As $\nabla_2 g^\eta$, $g^\eta$ and $\nabla_2 u_2^\eta$ are bounded in $L^1(0,t,L^\infty(\T))$ uniformly in $\eta$, then using \eqref{unifphieps} we get by the Grönwall lemma 
\begin{equation*}
    \forall s\in ]0,t[,\quad \Vert \nabla_2\varphi^\eta(s,\cdot)\Vert_\infty \leq C'.
\end{equation*}
Therefore
\begin{equation*}
    \Vert(u_2^\eta-u_2)\cdot \nabla_2 \varphi^\eta\Vert_{L^1(0,T,L^\infty(\R^d))} \leq C'\Vert u_2^\eta - u_2\Vert_{L^1(0,T,C_c([0,1]\times\R^{d-1}))} \tend{\eta}{0} 0
\end{equation*}
so to obtain \eqref{dette1} it's sufficiant to show
\begin{equation}\label{cfinii}
    (u_1^\eta-u_1)\partial_1\varphi^\eta \tend{\eta}{0} 0 \text{ in }L^1(0,T,L^\infty(\R^d)).
\end{equation}
Applying $\partial_1$ to \eqref{regul}, we get
\begin{equation}\label{d1drom}
\partial_t\partial_1\varphi^\eta +u^\eta\nabla\partial_1\varphi^\eta - (g^\eta - \partial_1 u_1^\eta)\partial_1\varphi^\eta = -(\partial_1 u_2^\eta) \nabla_2\varphi^\eta + \varphi^\eta\partial_1 g^\eta.
\end{equation}
If $g^\eta$, $\partial_1 u_1^\eta$ are bounded in $L^1(0,T,L^\infty(\R^d))$ uniformly in $\eta$, it's not the case of $\partial_1 u_2^\eta$ and $\partial_1 g^\eta$. Choosing $g^\eta := g\star_1\omega^\eta$ and $u_2^\eta := u_2\star_1 \omega^\eta$, where $\omega^\eta = \eta^{-1}\omega(\eta^{-1}\cdot)$ and $\omega$ is a mollifier, we obtain
\begin{equation*}
    \partial_1 u_2^\eta = u_2\star_1 \partial_1 \omega^\eta
\end{equation*}
hence 
\begin{equation*}
    \Vert \partial_1 u_2^\eta\Vert_{L^1(0,T,L^\infty(\R^d))}\leq \Vert u_2\Vert_{L^1(0,T,L^\infty(\R^d))}\Vert \omega^\eta\Vert_{L^1(0,T,L^\infty(\R^d))}\leq \frac{C''}{\eta}.
\end{equation*}
Similarly,
\begin{equation*}
    \Vert \partial_1 g^\eta\Vert_{L^1(0,T,L^\infty(\R^d))} \leq \frac{C''}{\eta}.
\end{equation*}
Therefore, by Grönwall, \eqref{d1drom} gives
\begin{equation}\label{papapa}
    \forall s\in ]0,t[,\quad \Vert \partial_1\varphi^\varepsilon(s,\cdot)\Vert_{L^\infty(\R^d)}\leq \frac{C''}{\eta}. 
\end{equation}
We choose now $u_1^\eta:=u_1\star_1 \omega^{(\eta^2)}$.
Then,
\begin{align*}
    \vert u_1^\eta - u_1\vert(x) &= \left\vert \int_\R u_1(x-y)\omega^{(\eta^2)}(y)dy - u(x)\right\vert
    \leq \int_\R \vert u_1(x-y) - u_1(x)\vert \omega^{(\eta^2)}(y)dy
    \\&\leq \int_\R \left\vert\int_x^{x-y}\vert \partial_1 u_1(z) \vert dz \right\vert \omega^{(\eta^2)}(y)dy
    \leq \Vert \partial_1 u_1\Vert_{L^\infty(\R)}\int_\R \vert y \vert \omega^{(\eta^2)}(y)dy
    \\&\leq \eta^2\Vert \partial_1 u_1\Vert_{L^\infty(\R)} \int_\R \vert y \vert \omega(y)dy \leq C'''\eta^2.
\end{align*}
So
\begin{equation}\label{papapa2}
    \Vert u_1^\eta - u_1\Vert_{L^1(0,T,L^\infty(\R^d))} \leq C'''\eta^2
\end{equation}
and \eqref{papapa}, \eqref{papapa2} give \eqref{cfinii}.
\end{proof}

\medskip

Remark that we cannot directly apply lemma \ref{lemmaunicity} with $d=4$ with
\begin{equation*}
    \textbf{u} = \left(u,0, \frac{-\rho(\sigma+p)}{\mu},\frac{\sigma(\sigma+p)}{\cv\mu\rho}\right)
\end{equation*}
\begin{equation*}
    \textbf{g} = -\frac{\sigma+p}{\mu}
\end{equation*}
because we have not then $\textbf{u},\textbf{g}\in L^1(0,T,C_c([0,1],\R^{3}))$. Introducing $\chi \in C^\infty_c(\R^3)$ such that $\chi_{|K} = 1$, we have
\begin{equation*}
  \forall (t,x)\in [0,T]\times\T,\quad \forall \beta\in C(K),\quad (\chi\beta)(c^\varepsilon(t,x),\rho^\varepsilon(t,x),\theta^\varepsilon(t,x) = \beta(c^\varepsilon(t,x),\rho^\varepsilon(t,x),\theta^\varepsilon(t,x))
\end{equation*}
hence
\begin{equation*}
    \forall (t,x)\in [0,T]\times\T,\quad \forall \beta\in C(K),\quad \ps{\nu^\varepsilon_{t,x}}{\chi\beta} = \ps{\nu^\varepsilon_{t,x}}{\beta}
\end{equation*}
then by uniqueness of the limit
\begin{equation*}
    \nu = \chi \nu.
\end{equation*}
Therefore
\begin{equation*}
    \partial_t \nu + \Div(\textbf{u}\chi \nu) + \textbf{g}\chi\nu = 0
\end{equation*}
and $\chi\textbf{u}$, $\chi\textbf{g}$ have the desired regularity.

% \paragraph{Second proof : combine a uniqueness result from Ambrosio and DiPerna Lions' theory}

% Considering $\nu \exp(\int_0^t g)$ instead of $\nu$, we are reduced to studying the uniqueness of the continuity equation

% \begin{equation}\label{conteq}
% \partial_t \nu + \Div(\textbf{u}\nu) = 0.
% \end{equation}

% But according to a theorem of Ambrosio (see \cite{AmCr}, theorem 3.1), if $\nu_0$ is a probability measure, and if $A\subset \R^d$ is a Borel set such that $\nu_0(A)=0$, then \eqref{conteq} has a unique solution if and only if, for all $x\in \R^d\backslash A$, there is a unique characteristic associate to $\textbf{u}$ starting from $x$. Because $\Div \textbf{u}\in L^1(0,T,L^\infty(\T))$, according to DiPerna Lions' theory, this uniqueness is guarantee for a certain $A$ negligeable for the Lebesgue measure. Therefore, if $\nu_0$ is a density probability, the problem is solved. In the general case, we take $\Upsilon_0$ some density probability, and construct $\Upsilon$ a solution to \eqref{conteq}. Taking $\Tilde{\nu}$ the solution of the kinetic equation construct above, we remark that $\nu-\Tilde{\nu}+\Upsilon$ is solution of \eqref{conteq} with for initial condition $\Upsilon_0$. So by uniqueness, $\nu - \Tilde{\nu}+\Upsilon = \Upsilon$, thus $\nu = \Tilde{\nu}$.

\paragraph{Second proof}
It turns out that if the initial velocity is in a higher Lebesgue space, we can recover that the transport velocity in the kinetic term is $L^1_tW^{1,\infty}_x$, which guarantees uniqueness. More precisely, we have the following
\begin{prop}
Suppose that $u_0\in W^{1,2+\eta}(\T)$ for some $\eta>0$. Then $\nabla\textbf{u}\in L^1(0,T,L^\infty(\R^4))$.
\end{prop}
\begin{proof}
We just have to prove that under these conditions,
\begin{equation*}
    \partial_x \sigma \in L^1(0,T,L^\infty(\T)),\quad \partial_x (\sigma^2)\in L^1(0,T,L^\infty(\T)).
\end{equation*}
Using \eqref{sigmapoint}, \eqref{D_A1} and \eqref{secondhoff}, we get
\begin{equation*}
\kappa^{1/2}\partial_x\left(\frac{\partial_x\sigma}{\rho}\right) \in L^2(0,T,L^2(\T)).
\end{equation*}
Thus, by the Gagliardo-Nirenberg inequality, and using that $\partial_x \sigma\in L^2(0,T,L^2(\T))$, we get
\begin{equation*}
\kappa^{1/4}\Vert \sigma\Vert_\infty\in L^\infty([0,T]),\quad \kappa^{1/4}\Vert\partial_x\sigma\Vert_\infty \in L^2([0,T])
\end{equation*}
In particular, 
\begin{align*}
    \int_0^T \Vert\partial_x \sigma\Vert_\infty = \int_0^T \kappa^{1/4} \Vert \partial_x \sigma\Vert_\infty \kappa^{-1/4} &\leq \left(\int_0^T \kappa^{1/2}\Vert\partial_x\sigma\Vert_\infty^2\right)^{1/2}\left(\int_0^T \kappa^{-1/2}\right)^{1/2} \\&\leq (T+1)^{1/2}\left(\int_0^T \kappa^{1/2}\Vert\partial_x\sigma\Vert_\infty^2\right)^{1/2}\leq C.
\end{align*}
Moreover, we have by the Gagliardo-Nirenberg inequality
\begin{equation}\label{ayounger}
    \Vert \partial_x (\sigma^2)\Vert_\infty^2 \leq C\Vert \partial_x(\sigma^2)\Vert_2 \Vert \partial_x(\partial_x(\sigma^2)/\rho)\Vert_2.
\end{equation}
But
\begin{equation*}
    \partial_x \left(\frac{\partial_x(\sigma^2)}{\rho}\right) = 2\sigma\partial_x\left(\frac{\partial_x\sigma}{\rho}\right) + \frac{(\partial_x \sigma)^2}{\rho} = 2\sigma \dot{\sigma} + \left(\frac{\gamma-1}{\cv}+1\right)\sigma^2\partial_x u + \frac{(\partial_x\sigma)^2}{\rho}   
\end{equation*}
hence
\begin{equation*}
    \int_0^1 \left(\partial_x \left(\frac{\partial_x(\sigma^2)}{\rho}\right)\right)^2 \leq C \int_0^1 \sigma^2\dot{\sigma}^2 + C\int_0^1 \sigma^4 (\partial_x u)^2 + C\int_0^1 (\partial_x \sigma)^4.
\end{equation*}
But
\begin{equation*}
    \int_0^T\kappa^{3/2}\int_0^1 \sigma^2 \dot{\sigma}^2 \leq \sup_{[0,T]}(\kappa^{1/2}\Vert \sigma\Vert_\infty^2) \int_0^T\kappa\int_0^1 \dot{\sigma}^2 \leq C
\end{equation*}
\begin{equation*}
    \int_0^T \kappa^{3/2}\int_0^1 \sigma^4 (\partial_x u)^2\leq \sup_{[0,T]}(\kappa^{1/2}\Vert \sigma\Vert_\infty^2)\sup_{[0,T]}\int_0^1 (\partial_x u)^2\int_0^T\Vert\sigma\Vert_\infty^2\leq C
\end{equation*}
\begin{equation*}
    \int_0^T \kappa^{3/2}\int_0^1 (\partial_x \sigma)^4 \leq \sup_{[0,T]}(\kappa \int_0^1(\partial_x\sigma)^2)\int_0^T \kappa^{1/2}\Vert\partial_x\sigma\Vert^2_\infty \leq C
\end{equation*}
so
\begin{equation}\label{firstmorceau}
    \int_0^T \kappa^{3/2}\int_0^1 \left(\partial_x \left(\frac{\partial_x(\sigma^2)}{\rho}\right)\right)^2 \leq C.
\end{equation}
Evaluating \eqref{sigmarenormalize} with $\beta(\sigma)=\vert \sigma\vert^{2+\eta}$, then integrating in space, we get
\begin{equation*}
    \frac{d}{dt} \int_0^1 \frac{\vert \sigma\vert^{2+\eta}}{\mu} +(2+\eta)(1+\eta)\int_0^1 \frac{(\partial_x \sigma)^2 \vert \sigma\vert^\eta}{\rho} \leq C\int_0^1\vert \partial_x u\vert \vert \sigma\vert^{2+\eta}. 
\end{equation*}
Because $\partial_x u\in L^1(0,T,L^\infty(\T))$ and because $\sigma_0\in W^{1,2+\eta}(\T)$, we can close the inequality to get
\begin{equation*}
    \int_0^1 \frac{\vert \sigma\vert^{2+\eta}}{\mu} + \int_0^T\int_0^1 \frac{(\partial_x\sigma)^2\vert \sigma\vert^\eta}{\rho} \leq C.
\end{equation*}
Then
\begin{equation}\label{secondmorceau}
    \int_0^T\kappa^{1/2-\eta/4}\int_0^1 (\partial_x \sigma)^2\sigma^2 \leq \sup_{[0,T]}(\kappa^{1/4}\Vert\sigma\Vert_\infty)^{2-\eta}\int_0^T \int_0^1 (\partial_x \sigma)^2 \vert \sigma\vert^\eta\leq C.
\end{equation}
Using \eqref{ayounger}, \eqref{firstmorceau} and \eqref{secondmorceau} we get
\begin{align*}
    \int_0^T \kappa^{1-\eta/8}\Vert \partial_x(\sigma^2)\Vert_\infty^2 &\leq C\int_0^T \kappa^{1/4-\eta/8}\Vert \partial_x(\sigma^2)\Vert_2 \kappa^{3/4}\Vert \partial_x(\partial_x(\sigma^2)/\rho)\Vert_2
    \\&\leq \frac{C^2}{2}\int_0^T \kappa^{1/2-\eta/4}\int_0^1 (\partial_x\sigma)^2\sigma^2 + \frac{C^2}{2}\int_0^T \kappa^{3/2}\int_0^1 \left(\frac{\partial_x(\sigma^2)}{\rho}\right)^2 
    \\&\leq C'.
\end{align*}
Therefore
\begin{align*}
    \int_0^T \Vert \partial_x(\sigma^2)\Vert_\infty &= \int_0^T \kappa^{-1/2+\eta/16}\kappa^{1/2-\eta/16}\Vert \partial_x(\sigma^2)\Vert_\infty
    \\&\leq \left(\int_0^T\kappa^{-1+\eta/8}\right)^{1/2}\left(\int_0^T\kappa^{1-\eta/8}\Vert \partial_x(\sigma^2)\Vert_\infty^2\right)^{1/2}
    \\&\leq C''.
\end{align*}
\end{proof}

\subsection{Construction of a particular solution and conclusion}

The objective of this section is to construct a particular solution for the kinetic equation. Using the uniqueness of solutions for the former we will identify on the one hand $\nu_{t,x}$, the family of Young measures associated to the sequence of solutions à la Hoff and on the other hand the particular solution that we will construct below. 

Remark that \eqref{systmacro} can be formally rewritten as
\begin{subequations}
\label{eqmacro}
    \begin{empheq}[left=\empheqlbrace]{alignat=1}
        \partial_t \alpha_\pm + \partial_x(\alpha_\pm u) &= \frac{\alpha_\pm}{\mu_\pm}\sigma + \alpha_\pm \frac{\gamma_\pm-1}{\mu_\pm}\rho_\pm\theta_\pm\\
        \alpha_\pm (\partial_t \rho_\pm + u\partial_x \rho_\pm) &= \alpha_\pm\left( -\frac{\rho_\pm}{\mu_\pm}\sigma - \frac{\rho_\pm (\gamma_\pm-1)\rho_\pm \theta_\pm}{\mu_\pm}\right)\\
        \alpha_\pm (\partial_t (\rho_\pm\theta_\pm) + u\partial_x (\rho_\pm \theta_\pm)) &= \alpha_\pm \left(\frac{\sigma^2}{\cv_\pm\mu_\pm} + \frac{\gamma_\pm-2}{\cv_\pm\mu_\pm}\rho_\pm\theta_\pm\sigma - \frac{\gamma_\pm-1}{\cv_\pm\mu_\pm}(\rho_\pm\theta_\pm)^2\right).\label{eqh}
    \end{empheq}
\end{subequations}
Moreover, because $\partial_x u\in L^1(0,T,L^\infty(\T))$, it can be shown that \eqref{systmacro} and \eqref{eqmacro} are equivalent.
\begin{lemma}
Given $u$ and $\sigma$ as in the theorem \ref{Theorem_main1}, the equation
\begin{subequations}\label{168c}
    \begin{empheq}[left=\empheqlbrace]{alignat=1}
    \partial_t h + u\partial_x h &= \frac{\sigma^2}{\cv_\pm\mu_\pm} + \frac{\gamma_\pm-2}{\cv_\pm\mu_\pm}h\sigma - \frac{\gamma_\pm-1}{\cv_\pm\mu_\pm}h^2\\
    h(0,\cdot)&=\rho_{\pm,0} \theta_{\pm,0}
    \end{empheq}
\end{subequations}
    has a unique global solution $h$ in $L^\infty(0,T,L^\infty(\T))$. Moreover, there exists some $c>0$ such that
    \begin{equation*}
        \forall (t,x)\in [0,T]\times \T,\quad h(t,x)\geq c.
    \end{equation*}
\end{lemma}
\begin{proof}
Let's denote $C:=\left(\int_0^T \Vert\sigma\Vert_\infty^2\right)^\frac{1}{2}$.
To begin with, let's introduce the function $\Phi_1 : L^1(0,T,L^\infty(\T))\rightarrow L^\infty(0,T,L^\infty(\T)), z \mapsto h$ which associate to a data $z$ the solution of

\begin{subequations}
        \begin{empheq}[left=\empheqlbrace]{align*}
        \partial_t h + u\partial_x h &= z\\
        h(0,\cdot) &= \rho_{\pm,0}\theta_{\pm,0}.  
        \end{empheq}
    \end{subequations}
    It can be shown that this function is well defined and $1$-Lipschitz. In particular we have
    \begin{equation*}
        \forall h\in L^\infty(0,T,L^\infty(\T)),\quad \Vert \Phi_1(h)\Vert_\infty \leq \Vert \rho_{\pm,0}\theta_{\pm,0}\Vert_\infty + \int_0^T \Vert z\Vert_\infty.
    \end{equation*}
    Let's define now
    $$\Phi_2 : \left|\begin{array}{ccc}
        L^\infty(0,T,L^\infty(\T)) & \rightarrow & L^1(0,T,L^\infty(\T) \\
        h & \mapsto & \dfrac{\sigma^2}{\cv_\pm\mu_\pm} + \dfrac{\gamma_\pm - 2}{\cv_\pm\mu_\pm} h\sigma - \dfrac{\gamma_\pm-1}{\cv_\pm\mu_\pm}h^2
    \end{array}\right..$$
    We obtain the inequality
    \begin{equation*}
        \forall h\in L^\infty(0,T,L^\infty(\T)),\quad \int_0^T \Vert \Phi_2(h)\Vert_\infty \leq \frac{C^2}{\cv_\pm\mu_\pm} + \sqrt{T}\Vert h\Vert_\infty \frac{\vert \gamma_\pm -2\vert}{\cv_\pm\mu_\pm}C + T\Vert h\Vert_\infty^2\frac{\gamma_\pm -1}{\cv_\pm\mu_\pm}.
    \end{equation*}
    So we get
    \begin{equation*}
        \Vert \Phi_1(\Phi_2(h))\Vert_\infty \leq \Vert \rho_{\pm,0} \theta_{\pm,0}\Vert_\infty + \frac{C^2}{\cv_\pm\mu_\pm} + \sqrt{T}\Vert h\Vert_\infty \frac{\vert \gamma_\pm -2\vert}{\cv_\pm\mu_\pm}C + T\Vert h\Vert_\infty^2\frac{\gamma_\pm -1}{\cv_\pm\mu_\pm}.
    \end{equation*}
    To have
    \begin{equation*}
        \Phi_1\circ \Phi_2(B(0,M))\subset B(0,M)
    \end{equation*}
    for a certain $M>0$, we just have to assure
    \begin{equation*}
        \Vert \rho_{\pm,0} \theta_{\pm,0} \Vert_\infty + \frac{C^2}{\cv_\pm\mu_\pm} + \sqrt{T}M \frac{\vert \gamma_\pm -2\vert}{\cv_\pm\mu_\pm}C + TM^2\frac{\gamma_\pm -1}{\cv_\pm\mu_\pm} \leq M
    \end{equation*}
    which is verified assuming for instance
    \begin{equation*}
        M \geq \Vert \rho_{\pm,0} \theta_{\pm,0} \Vert_\infty + \frac{C^2}{\cv_\pm\mu_\pm} + \frac{\vert \gamma_\pm -2\vert}{\cv_\pm\mu_\pm}C + \frac{\gamma_\pm -1}{\cv_\pm\mu_\pm},\quad T\leq \frac{1}{M^2}. 
    \end{equation*}
    We deduce from that
    \begin{equation*}
        \forall h_1,h_2\in L^\infty(0,T,L^\infty(\T)),\quad \int_0^T \Vert \Phi_2(h_2)-\Phi_2(h_1)\Vert_\infty \leq \Vert h_2-h_1\Vert_\infty \left(\sqrt{T}\frac{\vert \gamma_\pm-2\vert}{\cv_\pm\mu_\pm}C+2TM\frac{\gamma_\pm-1}{\cv_\pm\mu_\pm}\right)
    \end{equation*}
    and so
    \begin{equation*}
        \forall h_1,h_2\in L^\infty(0,T,L^\infty(\T)),\quad \Vert \Phi_1(\Phi_2(h_2))-\Phi_1(\Phi_2(h_1))\Vert_\infty \leq \Vert h_2 - h_1\Vert_\infty \left(\sqrt{T}\frac{\vert\gamma_\pm-2\vert}{\cv_\pm\mu_\pm}C + 2TM\frac{\gamma_\pm-1}{\cv_\pm\mu_\pm}\right). 
    \end{equation*}
    To obtain the contraction property of $\Phi_1\circ \Phi_2$, we just have to take 
    \begin{equation*}
        T < \min\left(1, \frac{\cv_\pm\mu_\pm}{\vert \gamma_\pm-2\vert C + 2M(\gamma_\pm - 1)}\right)^2.
    \end{equation*}
    So, using the Banach fixed-point theorem, we find the existence and uniqueness of a solution $h$ to \eqref{168c} with initial condition $\rho_{\pm,0} \theta_{\pm,0}$, if $T$ is smaller than a certain $T_0>0$. This $T_0$ depends on the initial data only through $\Vert \rho_{\pm,0}\theta_{\pm,0}\Vert_\infty$, in an non-decreasing manner. So in order to use a bootstrap argument to have existence and uniqueness of a global solution, we just have to find a uniform bound on $\Vert h\Vert_\infty$.

    \ 

    From \eqref{eqh} we get
    \begin{equation*}
        \frac{d}{dt} \Tilde{h}\leq \frac{\Tilde{\sigma}^2}{\cv_\pm\mu_\pm} + \frac{\gamma_\pm-2}{\cv_\pm\mu_\pm}\Tilde{h}\Tilde{\sigma}
    \end{equation*}
    hence
    \begin{align*}
        \Tilde{h}&\leq \rho_{\pm,0}\theta_{\pm,0} \exp\left(\frac{\gamma_\pm - 2}{\cv_\pm\mu_\pm}\int_0^t \Tilde{\sigma}\right) + \int_0^t \frac{\Tilde{\sigma}^2}{\cv_\pm\mu_\pm} \exp\left(\frac{\gamma_\pm-2}{\cv_\pm\mu_\pm}\int_s^t \Tilde{\sigma}\right)ds
        \\&\leq \left(\Vert \rho_{\pm,0} \theta_{\pm,0}\Vert_\infty + \frac{C^2}{\cv_\pm\mu_\pm}\right)\exp\left(\frac{\vert\gamma_\pm-2\vert}{\cv_\pm\mu_\pm}\sqrt{T}C\right).
    \end{align*}
    To finish with, let's show the existence of a $c>0$ such that
    \begin{equation*}
        (\forall t\in [0,t_0],\Tilde{h}(t)>0)\quad \Longrightarrow \quad \Tilde{h}(t_0)>c. 
    \end{equation*}
        Rewriting \eqref{eqh}
        \begin{equation*}
            \partial_t h + u\partial_x h = \frac{1}{\cv_\pm\mu_\pm}\left(\sigma + \frac{\gamma_\pm-2}{2}h\right)^2 - \frac{1}{\cv_\pm\mu_\pm}\left(\left(\frac{\gamma_\pm-2}{2}\right)^2 + \gamma_\pm-1\right)h^2
        \end{equation*}
        and denoting
        \begin{equation*}
            \beta = \frac{1}{\cv_\pm\mu_\pm}\left(\left(\frac{\gamma_\pm-2}{2}\right)^2 + \gamma_\pm-1\right)
        \end{equation*}
        we obtain
        \begin{equation*}
            \frac{d}{dt}\Tilde{h} \geq -\beta \Tilde{h}^2.
        \end{equation*}
        Dividing by $\Tilde{h}^2$, we get
        \begin{equation*}
            \frac{d}{dt}\left(\frac{1}{\Tilde{h}}\right) \leq \beta
        \end{equation*}
        hence
        \begin{equation*}
            \Tilde{h}(t)\geq \frac{\Tilde{h}(0)}{1+T\beta \Tilde{h}(0)}\geq \frac{\inf \rho_{\pm,0}\theta_{\pm,0}}{1+T\beta\inf \rho_{\pm,0}\theta_{\pm,0}} :=c.
        \end{equation*}
\end{proof}

\begin{prop}
    Given $u,\sigma$ and $(\alpha_{\pm,0}, \rho_{\pm,0},\theta_{\pm,0})$ as in the theorem \ref{Theorem_main1}, The system \eqref{eqmacro} has at least one global solution
    $(\alpha_\pm, \rho_\pm, \theta_\pm)\in (L^\infty(0,T,L^\infty(\T)))^3$.
\end{prop}
\begin{proof}
    Any equation of the type
    \begin{empheq}[left=\empheqlbrace]{align*}
        \partial_t b + u\partial_x b &= c b\\
        b(0,\cdot) &= b_0
    \end{empheq}
    \begin{empheq}[left=\empheqlbrace]{align*}
        \partial_t b' + u\partial_x b' &= c' b'\\
        b'(0,\cdot) &= b'_0
    \end{empheq}
with $b_0,b'_0\in L^\infty(0,T,L^\infty(\T))$ and $c_0, c_0', u\in L^1(0,T,L^\infty(\T))$ has a unique global solution in $L^\infty(0,T,L^\infty(\T))$. So, the system of equations
  \begin{subequations}
    \begin{empheq}[left=\empheqlbrace]{align*}
        \partial_t f + \partial_x (fu) &= f\left(\frac{\sigma}{\mu_\pm}  + \frac{\gamma_\pm-1}{\mu_\pm} h \right)\\
        \partial_t g + u\partial_x g &= g\left(\frac{\sigma}{\mu_\pm} - \frac{\gamma_\pm-1}{\mu_\pm} h\right) \label{168b}
    \end{empheq}
\end{subequations}
with initial data $(\alpha_{\pm,0},\rho_{\pm,0})$ has a unique solution. Moreover,
\begin{equation*}
    \Tilde{g} = \rho_{\pm,0} \exp\left(\frac{1}{\mu_\pm}\int_0^t \sigma - \frac{\gamma_\pm-1}{\mu_\pm}\int_0^t h\right) \geq \inf \rho_{\pm,0} \exp\left(-\frac{1}{\mu_\pm}\sqrt{T}C-\frac{\gamma_\pm-1}{\mu_\pm}T\Vert h\Vert_\infty\right).
\end{equation*}
Denoting now $\alpha_\pm = f$, $\rho_\pm = g$, $\theta_\pm = h/g$, we have built the desired solution.
\end{proof}
\paragraph{Uniqueness of the Kinetic equation and conclusion}
It is cumbersome yet straightforward to verify that by construction, the familly of measures $\Tilde{\nu}$ defined by  
\begin{equation*}
    \Tilde{\nu}_{t,x}:= \alpha_+(t,x) \delta_{(1,\rho_+(t,x),\theta_+(t,x))} + \alpha_-(t,x)\delta_{(0,\rho_-(t,x),\theta_-(t,x))}
\end{equation*}
are a solution of \eqref{KinEq}.
By uniqueness, we have that 
\[\Tilde{\nu}_{t,x}=\nu_{t,x}.\] 
This concludes the proof of Theorem \ref{Theorem_main3}.

%%%%%%%%%% Section 6

\section{Numerical illustrations}\label{num}

In order to illustrate our results, we consider a discretization of the equation systems mesoscopic \eqref{NS}-\eqref{c} and macroscopic \eqref{systmacro2} that is a quite straightforward extension of the one described more precisely in \cite{BrBuLa2}. We consider a moving-mesh scheme, such that, in the mesoscopic case, each cell is occupied by a pure fluid. To observe the oscillations that the change of phase produce, we consider only $J=100$ meshes, thus a discretization of the torus $\T$ (periodic boundary conditions with period 1) on the order of $\varepsilon=1/J = 0.01$. This is not a problem in general, as convergence seems to be very rapid. However, in a few illustrations we take $J=1000$ to observe continuity and shock properties of certain solutions.
Let us take 
\[ \mu_+ = 0.1, \quad \gamma_+ = 2, \quad \cv_+ = 1\]
and
\[ \mu_- = 0.2, \quad \gamma_- = 3, \quad \cv_- = 1.\]
We choose as periodic initial conditions for all $x\in \mathbb{T}$ and therefore on $(0,1)$:
$$u_0(x)=0$$
and
$$\alpha_0^+(x)= 
    1/2  \text{ if } x \in [0.25,0.75]  \text{ and }
    1 \text{ elsewhere} 
$$ 
$$ \rho_0^+(x)=  
    2 \text{ if } x \in [0.25,0.75]
    \text{ and }
    0.2 \text{ elsewhere} 
$$
$$\theta_0^+(x)= 
    2  \text{ if } x \in [0.25,0.75] 
    \text{ and }
    0.2  \text{ elsewhere} 
$$
and 
$$\alpha_0^-(x)= 
    1/2  \text{ if } x \in [0.25,0.75] 
    \text{ and }
    0  \text{ elsewhere} 
$$
$$ \rho_0^-(x)= 
    1 \text{ if } x \in [0.25,0.75] \text{ and } 
    0.2 \text{ elsewhere} 
$$
$$ \theta_0^-(x)= 
    1 \text{ if } x \in [0.25,0.75] \text{ and } 
    0.2  \text{ elsewhere.} 
$$
Solutions marked with an epsilon are mesoscopic solutions, the others are macroscopic solutions.

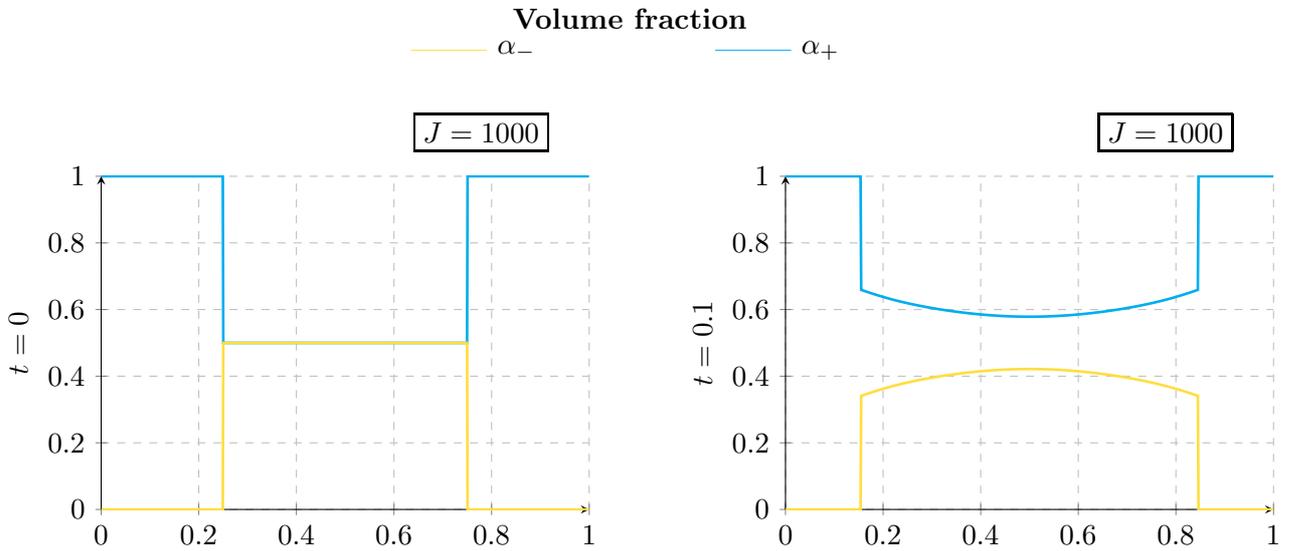
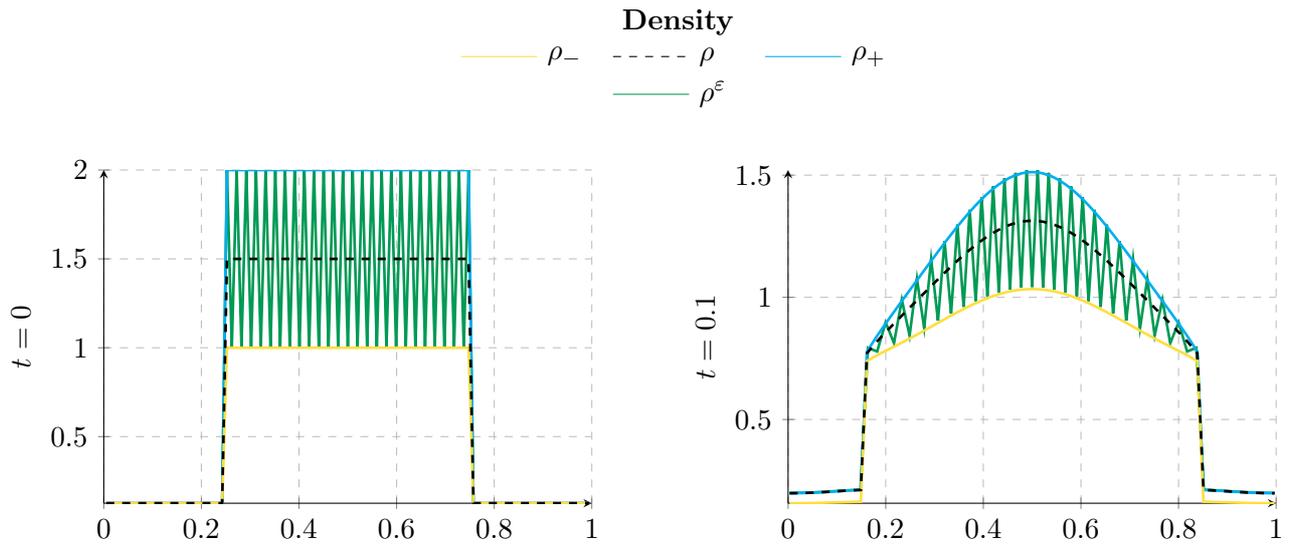
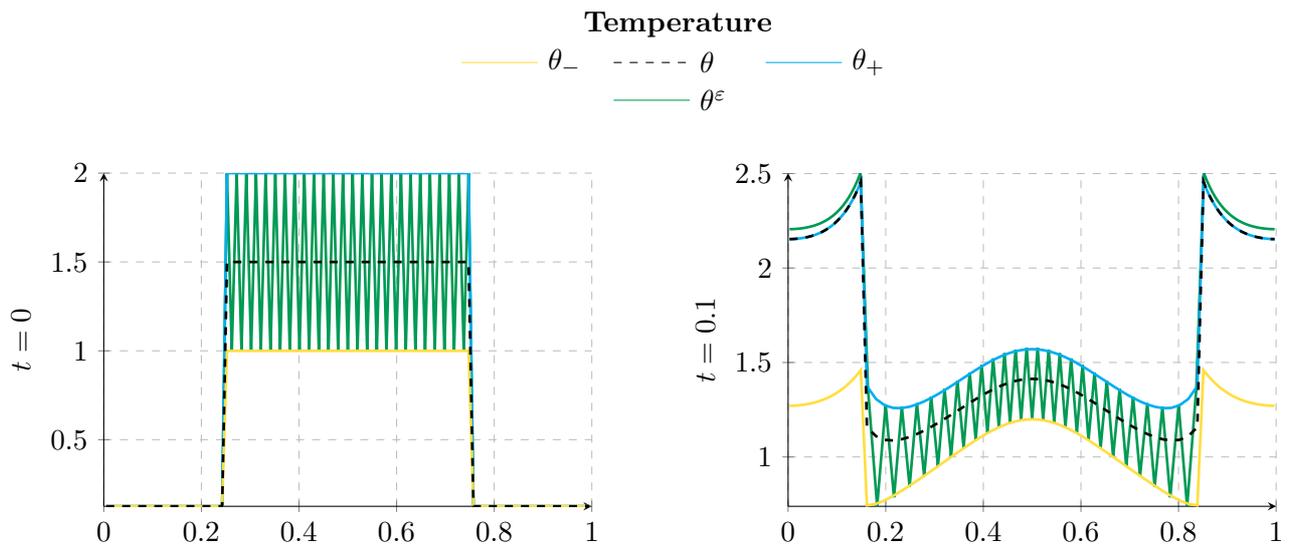
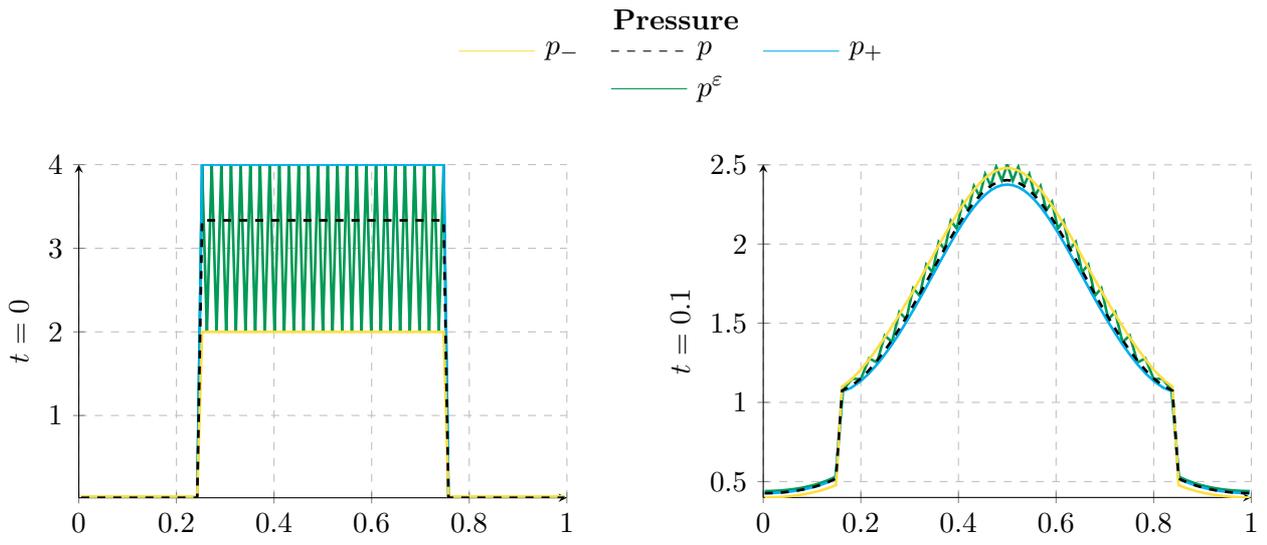
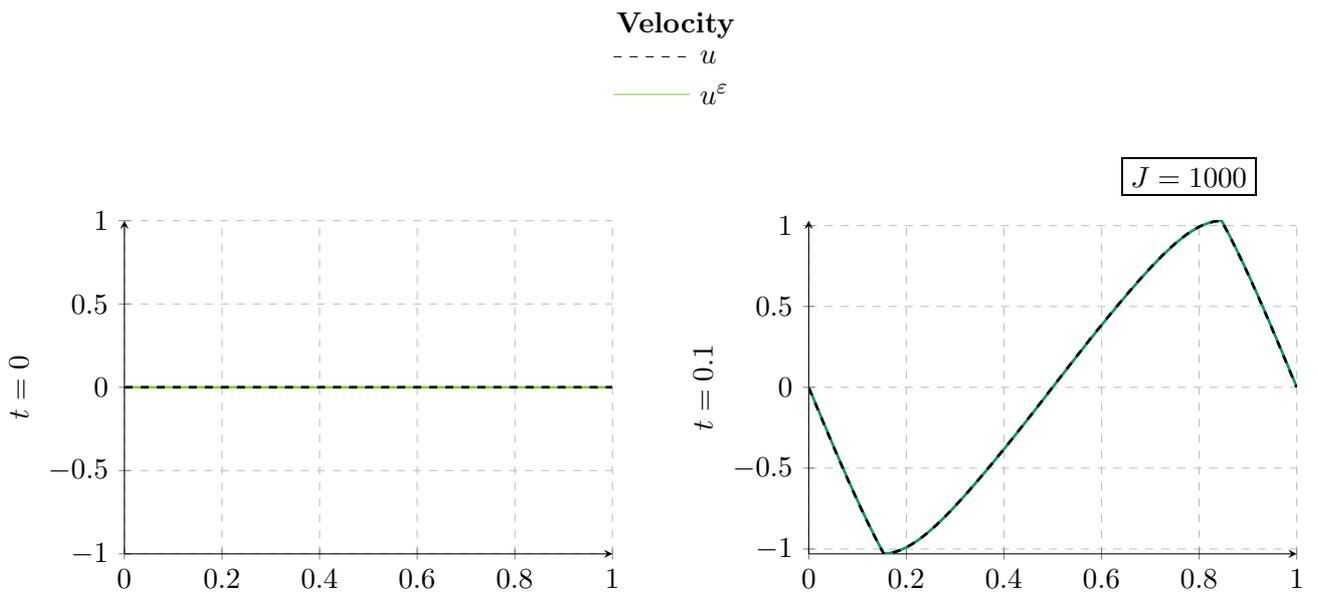
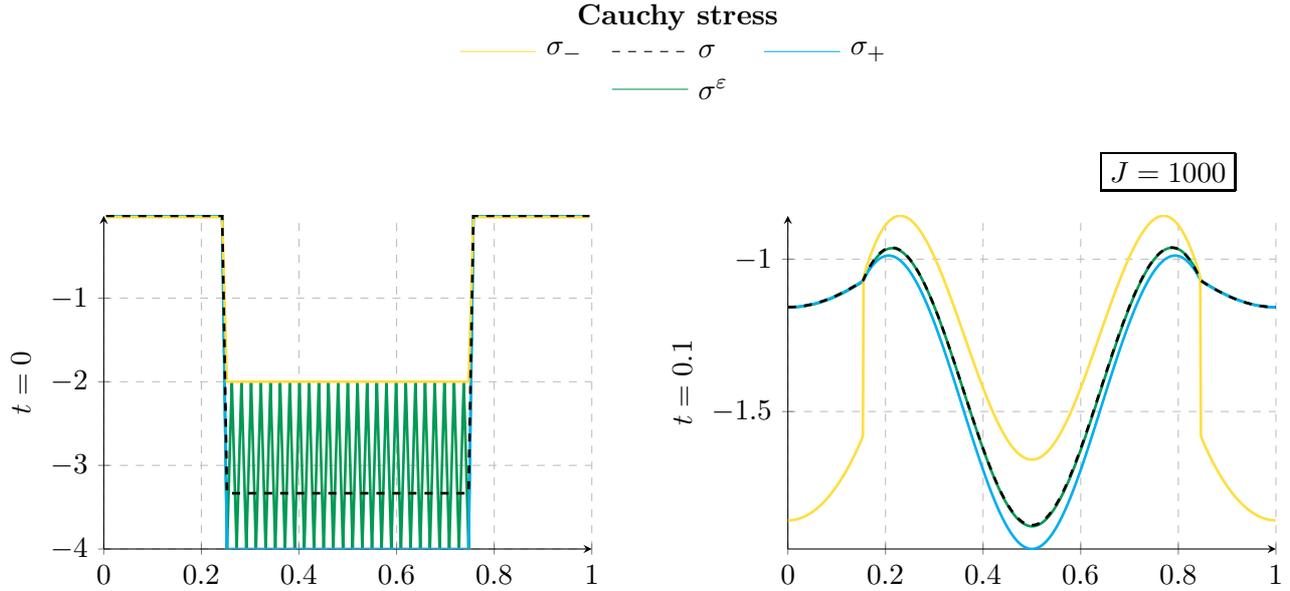
\vspace{2cm}

\begin{center}
   {\bf Volume fraction}


\noindent

\vspace{2cm}
These numerical results appear as an illustration of the mathematical results proven here above in the sense that the macroscopic densities, temperatures and pressures appear as envelopes for the mesoscopic ones, and both macroscopic and mesoscopic velocities coincide well. In order to "quantify" this, we propose an experimental estimation of the convergence order between the mesoscopic and macroscopic velocity and volume fraction $\alpha_+$. In order to compare the mesoscopic and macroscopic volume fraction $\alpha_+$ we evaluate the local mean value (between two consecutive cells) of the mesoscopic one. See Figure \ref{cvorder}. The order of convergence at the final time with respect to $1/J$ (that is also here related to the length of pure zones) seems to be 1 in the $L^\infty$ norm in space for the velocity (that is continuous with respect to the space variable), and in the $L^1$ norm for the volume fraction (that is {\em not} continuous). 

An interesting feature of these solutions is that they exhibit discontinuities for densities, pressures, volume fractions and temperatures, and that the amplitude of these discontinuities seems to decrease exponentially in time, which is reminiscent from \cite{Hoff1986} in the isentropic (barotropic) case. 

%\begin{table}[H]
%\begin{center}
%\caption{Error with respect to the number of cells $J$ (seems to be at first order with respect to $1/J$)}
%\label{cvorder}
%\begin{tabular}{l|c|c} 
%\textbf{Number of cells} & \textbf{$L^\infty$ error for the velocity} & \textbf{$L^1$ error for the volume fraction}\\
%\hline
%100 & 0.03843599999999997 & 0.006022668975000002\\
%200 & 0.018955999999999973 & 0.003028179522999994\\
%400 & 0.007919999999999983 & 0.0014352642650000004\\
%800 & 0.004050999999999971 & 0.000713681556\\
%1600 & 0.0021220000000000128 & 0.0003553959819999935\\
%3200 & 0.0011610000000000231 & 0.00017734317499998979\\%
%6400 & 0.0004970000000000252 & 0.00008978221100000116\\
%\hline
%\hline
%100 & 0.06528899999999999 & 0.014817149088999986\\
%200 & 0.03193500000000005 & 0.005097041359999999\\
%400 & 0.016430999999999973 & 0.0034820208419999977\\
%800 & 0.008045999999999998 & 0.0013023214330000104\\
%1600 & 0.004237000000000046 & 0.0008484118410000185\\
%3200 & 0.0021140000000000048 & 0.00048786051599999946\\
%6400 & 0.0010569999999999746 & 0.00021457286499999652\\
%\hline
%\end{tabular}
%\end{center}
%\end{table}

\begin{figure}[H]
\centering
\caption{Error with respect to the number of cells $J$}
\label{cvorder}
\begin{tikzpicture}
  \begin{loglogaxis}[
    width=11cm,
    height=7cm,
    xlabel={Number of cells},
    ylabel={Error},
    title={},
    grid=both,
    grid style={dotted,gray!50},
    legend style={at={(0.05,0.05)}, anchor=south west, font=\small},
    legend cell align={left},
    ticklabel style = {font=\small},
  ]

    % --- Données : L∞ (vitesse)
    \addplot[
      mark=*,
      color=blue,
      thick,
    ] coordinates {
      (100, 0.065289)
      (200, 0.031935)
      (400, 0.016431)
      (800, 0.008046)
      (1600, 0.004237)
      (3200, 0.002114)
      (6400, 0.001057)
    };
    \addlegendentry{$L^\infty$ error for the velocity}

    % --- Données : L1 (fraction volumique)
    \addplot[
      mark=square*,
      color=red,
      dashed,
      thick,
    ] coordinates {
      (100, 0.0148171)
      (200, 0.0050970)
      (400, 0.0034820)
      (800, 0.0013023)
      (1600, 0.0008484)
      (3200, 0.0004879)
      (6400, 0.0002146)
    };
    \addlegendentry{$L^1$ error for the volume fraction}

    % --- Droites de référence (ordre ~1.0 et ~1.3 environ)
    % (ajustées pour passer près des courbes)
    \addplot[
     domain=100:6400,
      color=gray,
      densely dotted,
    ] {0.928*(1/x)^0.949};
    \node[anchor=south west, font=\scriptsize, rotate=-27] at (axis cs:500,0.016) {order 0.949};

    \addplot[
      domain=100:6400,
      color=black,
     densely dotted,
    ] {6.0533*x^(-0.987)};
    \node[anchor=south west, font=\scriptsize, rotate=-25] at (axis cs:500,0.003) {order 0.987};

  \end{loglogaxis}
\end{tikzpicture}
\end{figure}

\section*{Acknowledgements}  
This work was partially supported by the Agence Nationale pour la Recherche grant ANR-23-CE400014-01 (Bourgeons project) and by the ANR under France 2030 bearing the reference ANR-23-EXMA-004 (Complexflows).
 
\printbibliography

\end{document}